\documentclass[notitlepage]{amsart}
\usepackage{amssymb,amsfonts,latexsym, esint}
\usepackage{graphics,verbatim}
\usepackage{graphicx}

\newtheorem{theorem}{Theorem}[section]
\newtheorem{proposition}{Proposition}[section]
\newtheorem{lemma}{Lemma}[section]

\newtheorem{remark}{Remark}[section]
\newtheorem{definition}{Definition}[section]

\newcommand{\eps}{\varepsilon}

\newcommand{\To}{\longrightarrow}

\newcommand{\be} {\begin{equation}}
\newcommand{\ee} {\end{equation}}
\newcommand{\bea} {\begin{eqnarray}}
\newcommand{\eea} {\end{eqnarray}}
\newcommand{\Bea} {\begin{eqnarray*}}
\newcommand{\Eea} {\end{eqnarray*}}
\newcommand{\pa} {\partial}

\newcommand{\al} {\alpha}
\newcommand{\ba} {\beta}
\newcommand{\de} {\delta}
\newcommand{\na}{\nabla}
\newcommand{\ga} {\gamma}
\newcommand{\Ga} {\Gamma}
\newcommand{\Om} {\Omega}
\newcommand{\om} {\omega}
\newcommand{\De} {\Delta}
\newcommand{\la} {\lambda}

\newcommand{\nequiv} {\not\equiv}

\newcommand{\no} {\nonumber}
\newcommand{\noi} {\noindent}
\newcommand{\lab} {\label}
\newcommand{\va} {\varphi}
\newcommand{\f}{\frac}
\newcommand{\R}{\mathbb R}

\newcommand{\Rn}{\mathbb R^N}
\newcommand{\Iom}{\int_{\Omega}}
\catcode`\@=11

\makeatletter \@addtoreset{equation}{section} \makeatother

\begin{document}

\title{On singular equations with critical and supercritical exponents}

\author{ Mousomi Bhakta}
\address{M. Bhakta, Department of Mathematics, Indian Institute of Science Education and Research, Dr. Homi Bhabha Road,
Pune 411008, India}
\email{mousomi@iiserpune.ac.in}

\author{ Sanjiban Santra}
\address{S. Santra, Department of Basic Mathematics\\
Centro de Investigaci\'one en Mathematic\'as \\ Guanajuato,
M\'exico}
\email{sanjiban@cimat.mx}

\date{}

\begin{abstract} We study the problem
\begin{equation*}
(I_{\eps})
\left\{\begin{aligned}
      -\De u- \frac{\mu u}{|x|^2}&=u^p -\eps u^q \quad\text{in }\quad \Om, \\
      u&>0 \quad\text{in }\quad \Omega,\\
      u &\in H^1_0(\Om)\cap L^{q+1}(\Om),
          \end{aligned}
  \right.
\end{equation*}
where $q>p\geq 2^*-1$, $\eps>0$,
$\Om\subseteq\Rn$ is a bounded domain with smooth boundary,  $0\in \Omega$, $N\geq 3$ and $0<\mu<\bar\mu:=\big(\f{N-2}{2}\big)^2$. 
%We prove at $0$, any solution of $(I_{\eps})$ has the singularity of  order $|x|^{-\nu}$ when $q<\f{2+\nu}{\nu}$ and of the order $|x|^{-\f{2}{q-1}}$, when $q>\f{2+\nu}{\nu}$, where $\nu=\sqrt{\bar\mu}-\sqrt{\bar\mu-\mu}$. Moreover, we show that when $q=\f{2+\nu}{\nu}$ and $u$ is radial, $u\sim |x|^{-\nu}|\log|x||^{-\f{\nu}{2}}$. This gives the 
We completely classify the singularity of solution at $0$ in the supercritical case.  Using the transformation $v=|x|^{\nu}u$, we reduce the problem $(I_{\eps})$ to $(J_{\eps})$
\begin{equation*}
(J_{\eps}) \left\{\begin{aligned}
      -div(|x|^{-2\nu} \na v)&=|x|^{-(p+1)\nu} v^p -\eps |x|^{-(q+1)\nu} v^q \quad\text{in }\quad \Om,\\    v&>0 \quad\text{in }\quad \Omega, \\   v& \in H^1_0(\Om, |x|^{-2\nu} )\cap L^{q+1}(\Om, |x|^{-(q+1)\nu} ),
          \end{aligned}
  \right.
\end{equation*}
and then formulating a variational problem for $(J_{\eps})$, we establish the existence of a variational solution $v_{\eps}$ and characterise the asymptotic behaviour of $v_{\eps}$ as $\eps\to 0$ by variational arguments and when $p=2^*-1$.

This is the first paper where the results have been established with super critical exponents for $\mu>0$. 
\end{abstract}

\maketitle

\tableofcontents

\section{ Introduction}
In this paper,  we consider the following family of singular problems:
\begin{equation}
  \label{eq:a3'}
\left\{\begin{aligned}
      -\De u- \frac{\mu u}{|x|^2}&=u^p -\eps u^q &&\text{in }\quad \Om,\\ u&>0  &&\text{ in }\quad \Om, \\
      u &\in H^1_0(\Om)\cap L^{q+1}(\Om),
          \end{aligned}
  \right.
\end{equation}
and
\begin{equation}
  \label{eq:a3}
\left\{\begin{aligned}
      -div(|x|^{-2\nu} \na v)&=|x|^{-(p+1)\nu} v^p -\eps |x|^{-(q+1)\nu} v^q \quad\text{in }\quad &&\Om,\\ v&>0 \quad \text{ in }\quad &&\Om, \\
      v & \in H^1_0(\Om, |x|^{-2\nu} )\cap L^{q+1}(\Om, |x|^{-(q+1)\nu} ),
          \end{aligned}
  \right.
\end{equation}
where $q>p\geq 2^*-1=\f{N+2}{N-2}$, $\eps>0$ is a parameter,
$\Om\subseteq\Rn$ is a star-shaped bounded domain with smooth boundary,  $0\in \Omega$, $N\geq 3$  and  $0<\mu<\bar\mu:=(\f{N-2}{2})^2$  and  $\nu\in(0, \f{N-2}{2})$.  By the  Pohozaev's identity, we know that when $\eps=\mu=\nu=0$ and $\Om$ is star shaped, \eqref{eq:a3'} and  \eqref{eq:a3}  have no solutions.  In this paper, we are mainly concerned with two issues. One of them is to classifying the nature of singularity to the  solutions of Eq.\eqref{eq:a3'} and the other one is  to study  the asymptotic behavior of solutions of the problem \eqref{eq:a3} as $\eps\to 0$ . When $\mu=0$, the asymptotic behaviour of this class of equation with supercritical exponent was studied by Merle and Peletier in \cite{MP1, MP2}.  Also see Mcleod \emph{et.al.} \cite{MMPT} for the uniqueness proof for the entire solution in the supercritical case, Han \cite{Z} and Brezis--Peletier\cite{BP} for the subcritical blow up.  As per our knowledge, there is no existing result with supercritical exponents for $\mu>0$.

We assume that \be\label{sup} v_{\eps}(0)= \max_{\Om} v_{\eps}(x).\ee

\vspace{2mm}

If we look for radial solutions of Eq. \eqref{eq:a3'}, we
would expect $u$ as a function of the radial variable $r$ to behave like $Ar^{-m}$ near $0$, where $A$ and $m$ satisfy
\be\lab{14-4-1}
A\displaystyle\left[-m(m+1)+m(N-1)-\mu\right]r^{-m-2}=-(1+o(1))A^qr^{-mq},
\ee
so that either
\be\lab{14-4-2}
m(m-N+2)+\mu>0, \quad
m+2=mq \Longrightarrow m=\f{2}{q-1} \quad\text{and}\quad q>\f{\mu+2\nu'}{\mu}
\ee
or
\be\lab{14-4-3}
m+2<mq, \quad m(m+1)-m(N-1)+\mu=0 \quad\Longrightarrow\quad m=\nu\ \text{or}\ \nu',
\ee
for \be\lab{14-nu}\nu:=\sqrt{\bar\mu}-\sqrt{\bar\mu-\mu};\quad\nu':=\sqrt{\bar\mu}+\sqrt{\bar\mu-\mu}.\ee Note that $\nu<\nu'$. Also, one can readily check $\f{\mu+2\nu'}{\mu}=\f{2+\nu}{\nu}$. In the region where $q<\f{2+\nu}{\nu}$, we have $\nu<\f{2}{q-1}$.  Therefore in this region
\be\lab{15-4-1}r^{-\nu}<c\, \min\{r^{-\f{2}{q-1}},\, r^{-\nu'}\}.\ee
  On the other hand, in the region where $q\geq\f{2+\nu}{\nu}$ we have \be\lab{15-4-2}r^{-\f{2}{q-1}}\leq c\,\min\{r^{-\nu},\, r^{-\nu'}\}.\ee
It is easy to check from \eqref{14-4-1} that for the blow up with  exponent $\f{2}{q-1}$ (see \eqref{14-4-2}) the constant $A$ would be determined, whereas for the second type of blow up it would appear to be free.

In Section 3 we prove that near $0$, any solution $u$ of Eq.\eqref{eq:a3'} satisfies
$$C_1|x|^{-\nu}\leq u(x)\leq C_2|x|^{-\nu} \quad\text{if}\quad  2^*-1\leq p<q<\f{2+\nu}{\nu},$$ and
$$C_3|x|^{-\f{2}{q-1}}\leq u(x)\leq C_4|x|^{-\f{2}{q-1}} \quad\text{if}\quad q>\max\bigg\{p, \f{2+\nu}{\nu}\bigg\},$$
 for some positive constants $C_1, C_2, C_3$ and $C_4$. Moreover when $u(x)=u(|x|)$ and $q=\f{2+\nu}{\nu}$,
 $$u(|x|)\sim |x|^{-\nu}\, |\log|x||^{-\f{\nu}{2}},\quad  \text{ as } |x| \to 0.$$
 More precisely,  if
$$-\De u- \frac{\mu u}{|x|^2}= f_{i}(u), \; i=1,2,$$ we can  classify the singularity of $u(x)$ near the origin
with the nonlinearities $f_{1}(u)=u^{p} + \eps u^q$ where $1\leq q<p=2^*-1$  or $f_{2}(u)=u^{p} -\eps u^q$ where $2^*-1\leq p<q$ in the following way.

\bigskip
\noi
\begin{tabular}{c|c|c|cl}
$ Range~ of~  (p, q) $  & $ 1\leq q<p=2^*-1 $ & $2^*-1\leq p<q<   \frac{2}{\nu}+1 $ & $q>\max\{p, \frac{2}{\nu}+1\}$   \\
\hline
$ Singularity~ at~ 0 $ & $C_1\leq |x|^{\nu} u(x)\leq C_2 $ & $ C_1\leq |x|^{\nu} u(x)\leq C_2 $ & $ C_1\leq |x|^{\frac{2}{q-1}} u(x)\leq C_2 $ \end{tabular}

\bigskip
for some $C_1>0, C_2>0.$
For the subcritical case see \cite{CP, Han}.

\vspace{3mm}

Near $0$,  Eq.\eqref{eq:a3'} can be written as $-\De u-\mu\f{u}{|x|^2}=-(1+o(1))u^q$.   Therefore, if $u$ is radial then by setting $v(r)=r^{\nu}u(r)$, the above equation reduces to
\be\lab{25-4-1}
v''+\f{N-1-2\nu}{r}v'=Ar^{-(q-1)\nu}v^q \quad\text{in}\quad (0, a),
\ee for some $a>0$ and $1-\de<A<1+\de$, for some $\de>0$ small. Using the Emden-Fowler transformation $t=(\f{\al}{r})^{\al}$ and $y(t)=\al^{-\nu}v(r)$, where $\al=N-2-2\nu$, \eqref{25-4-1} reduces to the so-called Emden-Fowler type equation
$$y''(t)=At^{\f{-(2\al+2)+(q-1)\nu}{\al}}y^q, \quad t\geq R ,$$ for some $R>0$ large. These type of equations have several interesting  applications in mathematical physics . It appears in astrophysics in the form of Emden equation and in atomic physics in the form of Thomas-Fermi statistical model of atoms. Emden-Fowler type equations appears in modelling the thermal behaviour of a spherical cloud of gas acting under
the mutual attraction of its molecules and subject to the classical laws of thermodynamics. For more details see \cite{Bell, Fow, Hille, Som}.

\vspace{2mm}

Recently, a great  deal of attention is given to the mathematical study of following class of  semilinear elliptic problem
\begin{equation}
  \label{a1}
\left\{\begin{aligned}
      \De u+\frac{\mu}{|x|^2} u+ f(u)&=0  &&\text{in } \Om,\\
      u &= 0 &&\text{on }\pa \Om,
    \end{aligned}
  \right.
\end{equation}
where $f$ is a  super-linear function;
$0\in \Om $ is a smooth bounded domain in $\mathbb{R}^N$, $0\leq \mu< \overline{\mu}=\frac{(N-2)^2}{4}$ and $N\geq 3$. This class of problems is of particular interest as this arises in mathematical models related to reaction diffusion equations and celestial mechanics. We recall the classical Hardy's inequality: if $u \in H^{1}_{0} (\Om)$, then \be \int_{\Om} |\nabla u|^2dx \geq \overline{\mu} \int_{\Om} \frac{u^2}{|x|^2}dx,\ee  where $\bar\mu$ is never achieved by any $u\in H^1_0(\Om)$.

 We denote by $D^{1,2}(\R^N)$, the closure of $C^{\infty}_{0}(\R^N)$ with respect to the norm\\ $\big(\int_{\R^N}|\nabla u|^2dx\big)^{\frac{1}{2}}$. When  $0<\mu<\bar\mu$, it is easy to check that the following is an equivalent norm for $D^{1,2}(\Rn)$:
\be \|u\|_{D^{1, 2}(\R^N)}:= \bigg(\int_{\R^N}|\nabla u|^2dx- \mu\int_{\R^N} \frac{|u|^2}{|x|^2} dx\bigg)^{\frac{1}{2}}.\ee

When $\mu=0$, define
\be \mathcal{S}(v):=\frac{\displaystyle\int_{\Om}|\nabla v|^2 dx}{ \left(\displaystyle\int_{\Om}u^{2^{\star}} dx\right)^{\frac{2}{2^{\star}}}}, \quad \mathcal{S}_{N}=\inf_{v\in D^{1, 2}(\R^N)\setminus \{0\}}\frac{\displaystyle\int_{\R^N}|\nabla v|^2 dx}{ \left(\displaystyle\int_{\R^N}v^{2^{\star}} dx\right)^{\frac{2}{2^{\star}}}}.\ee
For $N\geq 3$ and $\nu=0$,   Merle--Peletier \cite{MP2} proved :

\noi{\bf Theorem}(Merle-Peletier, \cite{MP2})
\begin{itemize}
\item[(i)] There exist  $\eps$ and $\theta_{\eps}$ with $\eps\to 0$ and $\theta_{\eps}$ is uniformly above and away from $0$, such that there exists a solution $u_{\eps}$ of Eq. \eqref{eq:a3} with $\nu=0$ and  Furthermore, if $p=2^*-1$, then $S(\theta_{\eps} u_{\eps})  \to \mathcal{S}_{N}$ as $\eps \to 0$ and there exists constants $A, B$ such that $A<\int_{\Om}u_{\eps}^{p+1}<B$.

if $p> 2^*-1,$ then $K(\theta_{\eps} u_{\eps})  \to K_{N}$ as $\eps \to 0$ and   $\displaystyle\int_{\Om} u_{\eps}^{p+1} dx\to 0$ as $\eps \to 0,$ where
\be \no \no K( u) =\frac{\displaystyle\int_{\Om} |\nabla u|^2}{\displaystyle\int_{\Om} | u|^{p+1}} +  \frac{\displaystyle\int_{\Om} |u|^{q+1}}{\bigg(\displaystyle\int_{\Om} | u|^{p+1}\bigg)^{l}} ; \quad l= \frac{2(q+1)-N(p-1)}{2(p+1)- N(p-1)}, \ee
and $$K_{N}= \inf \bigg\{ K(u): u\in D^{1, 2}(\R^N) \cap L^{q+1}(\R^N) , \int_{\Om} | u|^{p+1}=1\bigg\}.$$
\item[(ii)] Let $x_{\eps}$ be a point such that $u_{\eps}(x_{\eps})=\| u_{\eps}\|_{\infty} $ and assume that up to a subsequence  $x_{\eps}\to x_{0}$ as $\eps \to 0$. Then in the case $p=2^*-1$,
\begin{equation}
     \eps^{\frac{1}{q-p+2}}\| u_{\eps}\|_{\infty}\sim  A(q, N)\no \quad\text{as}\quad \eps \to 0.
\end{equation} and when $x\not=0$
\be   \eps^{-\frac{1}{q-p+2}}  u_{\eps} (x)\rightarrow [N(N-2)]^{\frac{N-2}{2}}\frac{ G(x, x_{0})}{ A(q, N) R(x_{0})}\no \quad\text{as}\quad \eps \to 0.
\no\ee
where
\be A(q, N)= \bigg[\frac{N^2c (q, N)}{[N(N-2)]^{\frac{N}{2}}}  B\bigg(\frac{N}{2}, q\frac{N-2}{2}-1\bigg)\bigg]; \quad c(q, N)= \frac{(N-2)q- (N+2)}{2(q+1)} \no \ee
 and
$B(a,b)$ denotes the Beta function \cite{MP2} defined by
\be\lab{eq:beta}
B(a,b)=\int_{0}^{\infty}t^{a-1}(1+t)^{-a-b}.\ee
$G$ is the Green function and $x_{0}$ is the critical point of the Robin function, see \eqref{Robin} with $\nu=0$.
Moreover, if $p>2^*-1$, then
\begin{equation}
     \eps^{\frac{1}{q-p}}\| u_{\eps}\|_{\infty}\sim  c^{*} \no \quad\text{as}\quad \eps \to 0,
\end{equation}
and when $x\not=0$,
\be \eps^{-{\theta}}  u_{\eps} (x)\rightarrow  (c^*)^{-\theta} (J_{p}- c^* J_{q})G(x, x_{0}), \quad\text{as}\quad \eps \to 0 ,\ee
where
\be \theta =  \frac{(N-2)p- N}{2(q-p)};  \quad J_{p} =\int_{\R^N} V^p dx \no \ee and  $(c^*, V)$ is the unique solution of
\begin{equation*}
  \left\{
    \begin{aligned}
      -\De V&= V^{p }- c^* V^q \quad\text{in }\ \ \R^N,\\
V &\in D^{1, 2}(\R^N) \cap L^{q+1}(\R^N)  &&\text{}.
    \end{aligned}
  \right.
\end{equation*}
\end{itemize}

\vspace{5mm}

In \cite{J, RS} the following problem with critical exponent and Hardy potential was studied:
\begin{equation}
  \label{a2b}
  \left\{
    \begin{aligned}
      -\De u -\frac{\mu}{|x|^2} u&= u^{2^*-1 }+ \eps u\; &&\text{in } \Om,\\
 u &> 0 &&\text{in }  \Om,  \\
u &= 0 &&\text{on } \pa \Om,
    \end{aligned}
  \right.
\end{equation}
where $0\in \Om \subset \R^N$; $N\geq 3$, $0<\mu< \overline{\mu}$ and $\eps>0$ is a parameter.
Jannelli \cite{J} proved the following:
 If $ 0 \leq \mu\leq \overline{\mu} -1,$ then  (\ref{a2b}) has a positive solution $u\in H^{1}_{0}(\Om)$ for all $0<\eps< \lambda_{1}.$
Furthermore,  he proved that if  $\overline{\mu}-1<\mu< \overline{\mu}$, Eq.(\ref{a2b}) has a positive solution $u\in H^{1}_{0}(\Om)$
if and only if $\eps \in (\lambda^{\star}, \lambda_1)$ for some $\lambda^{\star}\in (0, \lambda_1)$, when  $\Om$ is the ball then Eq.(\ref{a2b}) has no positive solution for all $\eps \leq \lambda^{\star}$.
Cao-Peng \cite{CP} studied  problem similar to Eq.(\ref{a2b}) for the almost critical case.  Cao-Peng \cite{CP} and Ramaswamy-Santra \cite{RS}
used the radial nature of the positive solution to obtain the global uniqueness and blow-up profile as $\eps \rightarrow 0$. It was proved in \cite{RS}, when
$N \geq 5$ and $v_{\eps}\in H^{1}_{0}( \Om, |x|^{-2\nu})$ is a  solution  of Eq. (\ref{a2b}) satisfying \be \mathcal{S}=\lim_{\eps\rightarrow 0}\frac{\displaystyle ||x|^{-\nu}|\nabla v_{\eps}||_{L^2(\Om)}}{\||x|^{-\nu}v_{\eps}\|^{2}_{L^{2^{\star}}(\Om)}}, \quad \mu<\overline{\mu}-1,\no\ee then along a subsequence
\be \lim_{\eps \rightarrow 0}\eps \|v_{\eps}\|^{\frac{2(N-2\nu-4)}{N-2-2\nu}}_{\infty}= \frac{(N-2)^2}{2N^2(N-2-2\nu) b_{n}} \mathcal{S}(\mu)^{-\frac{N}{2}}\sigma_{N}| R(0)|\ee where
$b_{n}=\displaystyle\int_{0}^{\infty}\frac{t^{N-2\nu-1}}{(1+ t^{\frac{2\al}{N-2}})^{N-2}}dt;$ and when $x\neq 0$
\begin{eqnarray*}  \lim_{\eps \rightarrow 0} v_{\eps}(x) \|v_{\eps}\|_{\infty}= \frac{N-2}{N(N-2-2\nu)}\om_{N}G(x, 0),\end{eqnarray*}
where $R(0)$ and $G(x,0)$ are as defined in \eqref{Robin} and \eqref{green} respectively.

\vspace{3mm}

Define,
\be \label{go} \displaystyle\mathcal{S}=\inf_{u\in D^{1, 2}(\R^N, |x|^{-2\nu})\setminus \{0\}}\frac{\displaystyle\int_{\R^N}|x|^{-2\nu} |\nabla u|^2 dx }{\left(\displaystyle\int_{\R^N}|x|^{-2^*\nu} u^{2^{\star}}dx\right)^{\frac{2}{2^{\star}}}},\ee
where $ \nu\in\displaystyle\left[0, \f{N-2}{2}\right)$.
Thanks to Caffarelli-Kohn-Nirenberg inequality \cite{CKN}, we have $\mathcal{S}>0$. It is also well-known that $\mathcal{S}$ in the above expression is same as
$$
\inf_{u\in H^{1}_{0}(\Om)\setminus \{0\}}\frac{\displaystyle\int_{\Om}\left(|\nabla u|^2-\mu \frac{u^2}{|x|^2}\right)dx}{\left(\displaystyle\int_{\Om}u^{2^{\star}}dx\right)^{\frac{2}{2^{\star}}}}$$
and  independent of the domain $\Om$, where $\mu\in(0,\bar\mu)$. In the above two expression of $\mathcal{S}$, the parameters $\mu$ and $\nu$ are related by
 $\nu=\sqrt{\bar\mu}-\sqrt{\bar\mu-\mu}$.
From \cite{CZ}, we know $$\mathcal{S}=\mathcal{S}_{N}\bigg(1-\frac{4\mu}{(N-2)^2}\bigg)^{\frac{N-1}{N}},$$ where $\mathcal{S}_N$ is the usual Sobolev constant. Moreover, Catrina-Wang \cite{CZ, AG} proved that $\mathcal{S}$ is achieved by
\be\label{ent-U} U(x)=\bigg(\frac{N\al^2}{N-2}\bigg)^{\frac{N-2}{4}}\big(1+|x|^\frac{2\alpha}{N-2}\big)^{-\frac{N-2}{2}},\ee where \be \label{qes} \al=N-2-2\nu .\ee
Furthermore,  by \cite{T}, $U$ is the unique solution (dilation invariant) of the following entire problem:
\begin{equation}
  \label{ent}
  \left\{
    \begin{aligned}
      - \nabla  (|x|^{-2\nu} \nabla U ) &= |x|^{-2^*\nu}   U^{2^*-1} \; &&\text{in } \R^N, \\
 U &> 0 &&\text{in }  \R^N\setminus \{0\},  \\
U &\in  D^{1,2}(\R^N, |x|^{-2\nu}).
    \end{aligned}
  \right.
\end{equation}

Define the Green's function $G$ as
\be\lab{eq:gr-H}
H(x, y)= G(x, y)+ F (x, y),
\ee
where $G(x, y)$ is  defined by
\begin{equation}
  \label{green}
  \left\{
    \begin{aligned}
      \nabla_{x} (|x|^{-2\nu} \nabla_{x} G(x, y))&=  \de_{y} \; &&\text{in } \Om, \\
G(x, y) &=0  \; &&\text{on } \pa \Om,
    \end{aligned}
  \right.
\end{equation}
and $H(x, y)$ is the regular part of the Green function
\begin{equation}
  \label{rob}
  \left\{
    \begin{aligned}
      \nabla_{x} (|x|^{-2\nu} \nabla_{x} H(x, y))&=  0 \; &&\text{in } \Om, \\
H(x, y) &=F(x, y)  \; &&\text{on } \pa \Om,
    \end{aligned}
  \right.
\end{equation}
for any fixed $y\in \Om$
and
\be\lab{eq:gr-F} F(x, y)= -\frac{1}{(N-2-2\nu) \om_{N}|x-y|^{N-2\nu-2}}
\ee
 is the fundamental solution of the  non-degenerate elliptic operator $\nabla (|x|^{-2\nu} \nabla ).$  By construction, $H(x, 0)$ is negative and H\"older continuous near the origin \cite{CW}. Define the Robin function as
\begin{equation}\label{Robin}
R(x)= H(x, x).
\end{equation}
Hence $R$ is continuous  at the origin and we can write $$\lim_{|x|\rightarrow 0}R(x)= R(0).$$

For the supercritical case ($p>2^*-1$), we define the functional
\be\label{a22} F(v, \Om)= \frac{1}{2}\frac{\displaystyle\int_{\Om} |x|^{-2\nu}|\nabla v|^2 dx}{\displaystyle\int_{\Om}|x|^{-(p+1)\nu} v^{p+1} dx}+\frac{1}{q+1}\frac{\displaystyle\int_{\Om}|x|^{-(q+1)\nu}v^{q+1} dx}{\displaystyle\left(\int_{\Om}|x|^{-(p+1)\nu} v^{p+1} dx\right)^l}, \ee
where
\be\lab{14:l}
l=\frac{2(q+1)-N(p-1)}{2(p+1)-N(p-1)}, \ee and
$v\in H^1_0(\Om, |x|^{-2\nu})\cap L^{q+1}(\Om, |x|^{-(q+1)\nu})$.\\
Also define,
\be\lab{14-K}
\mathcal{K}:= \inf\bigg\{F(v, \Rn): v\in D^{1,2}(\Rn, |x|^{-2\nu})\cap L^{q+1}(\Rn, |x|^{-(q+1)\nu}), \int_{\Rn}|x|^{(p+1)\nu} |v|^{p+1}=1\bigg\}.
\ee
For the critical case ($p=2^*-1$), we consider the usual functional
\be\lab{sep-17-1}
S(v)=\frac{\displaystyle\int_{\Om} |x|^{-2\nu}|\nabla v|^2 dx}{\displaystyle\left(\int_{\Om}|x|^{-(p+1)\nu} v^{p+1} dx\right)^\f{2}{p+1}},
\ee
where $v\in H^1_0(\Om, |x|^{-2\nu})\cap L^{q+1}(\Om, |x|^{-(q+1)\nu})$.\\
We turn now to a brief description of the results presented below. The first result concerns the non-existence result when $p=2^*-1$.
\begin{theorem}\label{1.1} Let  $0\leq \mu<\bar\mu$ and
\begin{equation}
  \label{eq:a1}
\left\{\begin{aligned}
      -\Delta u-\mu \frac{u}{|x|^2}&=u^{2^*-1} -u^q \quad\text{in }\quad \R^N,\\ u&>0, \\
      u &\in D^{1,2}(\Rn)\cap L^{q+1}(\Rn),
         \end{aligned}
  \right.
\end{equation}
where $q>2^*-1.$ Then Eq.\eqref{eq:a1} does not admit any solution.
\end{theorem}
The proof of this theorem is based on the Pohozaev identity. The difficulty in applying this identity comes from the fact that  the solution blows up at origin(see Section 3).

\vspace{2mm}

Setting the transformation $v=|x|^{\nu}u$ in \eqref{a22} and \eqref{14-K} we obtain
\be\label{a21} F(u, \Om)=  \frac{1}{2}\frac{\displaystyle\int_{\Om}\left( |\nabla u|^2 -\mu  \frac{| u|^2}{|x|^2}\right)dx}{\displaystyle\int_{\Omega} u^{p+1} dx}+\frac{1}{q+1}\frac{\displaystyle\int_{\Om}u^{q+1} dx}{\displaystyle\left(\int_{\Omega} u^{p+1} dx\right)^l},\ee
where $p>2^*-1$ and $l$ is same as in \eqref{14:l},  $u\in H_0^1(\Om)\cap L^{q+1}(\Om)$ (see \cite[Theorem 1.1]{Han}) and
\be\lab{a2}
\mathcal{K}:= \inf\bigg\{F(u, \Rn): u\in D^{1,2}(\Rn)\cap L^{q+1}(\Rn),\,  \int_{\Rn}|u|^{p+1}=1\bigg\}.
\ee
\begin{theorem}\label{1.2}
Let  $0\leq \mu<\bar{\mu}$ , $N\geq 3$  and  $q>p> 2^*-1$. Then
$\mathcal{K}$ in \eqref{a2} is achieved by a radially decreasing function  in $D^{1, 2}(\R^N) \cap L^{q+1}(\R^N)$.   Furthermore, there exists a constant $\lambda>0$
\begin{equation}
  \label{a3}
\left\{\begin{aligned}
      -\Delta u-\mu \frac{u}{|x|^2}&=\lambda u^{p} -u^q \quad\text{in }\quad \R^N,\\
       u&>0   \quad\text{in }\quad \R^N,\\
      u &\in D^{1,2}(\Rn)\cap L^{q+1}(\Rn).
         \end{aligned}
  \right.
\end{equation}
\end{theorem}

\begin{theorem}\label{t:low-est}
 Assume $2^*-1\leq p<q<\frac{2+\nu}{\nu}$ and  $u$ be any solution (whenever exists) of
\begin{equation}\label{a4}
\left\{\begin{aligned}
      -\Delta u-\mu \frac{u}{|x|^2}&=u^p -  u^q &&\text{in } \Om,\\ u&>0 && \text{ in } \Om, \\
      u &\in H^1_0(\Omega)\cap L^{q+1}(\Omega) ,
          \end{aligned}
  \right.
\end{equation}
with $0<\mu<\bar\mu$ and $\Om$ be any smooth domain (bounded or unbounded).
Then there exists $r_0>0$(small) and $C_1>0$ ($r_0$ and $C_1$ independent of $u$) such that $u$ satisfies
$$u(x)\geq C_1|x|^{-\nu} \quad\forall\ \ x\in B_{r_0}(0)\setminus\{0\}.$$
\end{theorem}
\iffalse
\begin{remark}
Here the restriction on the upper bound of $\mu$ comes from the fact that $(2^*-1, \f{1+\nu}{\nu})$ should be a nonempty set. To ensure this, we need $\nu<\f{N-2}{4}$ and that in turn demands $\mu$ to in $\displaystyle\left(0,\, \bar\mu - \f{N-2}{4}\right)$.
\end{remark}
\fi
\begin{remark}
Standard methods of finding lower estimate, e.g. the methods of \cite{CP,  Han} do not work here. In Section 3, we have shown that to get the estimate $|u(x)|\geq C|x|^{-\nu}$, it is enough to show that solution of the following equation is bounded away from $0$,
$$-\text{div}(|x|^{-2\nu}\na w)+|x|^{-(q+1)\nu}w^q=0.$$ To show the existence of positive solution of this equation with suitable boundary data and which is bounded away from $0$, we have used ODE technique, Banach fixed point theorem  and  comparison principle.
\end{remark}

\vspace{2mm}

\begin{theorem}\label{t:up-est}
(i) If $p=2^*-1$, then any solution $u$ of Eq. \eqref{a4} satisfies
$$u(x)\leq C |x|^{-\nu} \quad\forall\ \ x\in B_{\rho_0}(0)\setminus\{0\},$$ where $\rho_0>0$ is sufficiently small.

\vspace{2mm}

(ii) If $p>2^*-1$ and $q>(p-1)\frac{N}{2}-1$ then the same conclusion holds as in (i).
\end{theorem}
\begin{remark}
Since $p=2^*-1$ implies $(p-1)\frac{N}{2}-1=2^*-1$, the condition $q>(p-1)\frac{N}{2}-1$ is readily satisfied in the case $p=2^*-1$ as $q$ is supercritical.
\end{remark}
\begin{theorem}\label{t:up-est-2}
Let $\mu\in(0,\bar\mu)$ and  $q>\max\{p, \f{2+\nu}{\nu}\}$. Then any solution of Eq.\eqref{a4} satisfies $$u(x)\leq C|x|^{-\f{2}{q-1}} \quad\forall\ \ x\in B_{\rho}(0)\setminus\{0\},$$ where $\rho>0$ is sufficiently small.
\end{theorem}

\begin{theorem}\lab{t:low-est-2}
Let $\mu\in(0,\bar\mu)$ and  $q>\max\{p, \f{2+\nu}{\nu}\}$. Then any solution of Eq.\eqref{a4} satisfies $$u(x)\geq C|x|^{-\f{2}{q-1}} \quad\forall\ \ x\in B_{R}(0)\setminus\{0\},$$ where $R>0$ is sufficiently small.
\end{theorem}

\begin{theorem}\lab{t:critical}
Let $\mu\in(0,\bar\mu)$ and  $q=\f{2+\nu}{\nu}$. Then any radial solution $u$ of Eq.\eqref{a4} satisfies
$$\lim_{|x|\to 0}\f{u(x)}{|x|^{-\nu}\big|\log |x|\big|^{-\f{\nu}{2}}}=\displaystyle\left(\f{\al\nu}{2}\right)^\f{\nu}{2}.$$
where $\al=N-2-2\nu$.
\end{theorem}

\begin{theorem}\lab{t:grad}
Let $2^*-1\leq p\leq (p-1)\f{N}{2}-1<q$ and $\tilde\rho:=\f{1}{2}\min\{\rho_0, \rho\}$, where $\rho_0$ and $\rho$ be as in Theorem \ref{t:up-est} and Theorem \ref{t:up-est-2} respectively. Then there exists $\mu^*=\mu^*(N,q)>0$ and  a constant $C$ depending on $N, p, q, \mu$ such that any solution $u$ of Eq. \eqref{a4} satisfies
\begin{center}
$\displaystyle|\na u(x)|\leq \left\{\begin{array}{lll}
C|x|^{-(\nu+1)} \quad\text{if}\quad \mu\in[0,\mu^*),\\
C|x|^{-(\f{q+1}{q-1})}\quad\text{if}\quad \mu\in [\mu^*,\bar\mu),
\end{array}
\right.$
\end{center}
for $0<|x|<\tilde\rho$.
\end{theorem}
\begin{remark}
In the above theorem $\mu^*=\displaystyle\left(\f{N-2}{2}\right)^2-\left(\f{N-2}{2}-\f{2}{q-1}\right)^2$. It's easy to note that $\mu<\mu^*\Longleftrightarrow q<\f{2+\nu}{\nu}$. From Theorem \ref{t:low-est} and Theorem \ref{t:up-est}, it follows any solution $u$ has singularity of the order $\nu$ when $q<\f{2+\nu}{\nu}$. Therefore in this region of $q$, it is anticipated that $|\na u|\leq C|x|^{-(\nu+1)}$ near $0$. On the other hand when $q>\f{2+\nu}{\nu}$, from Theorem \ref{t:up-est-2} and Theorem \ref{t:low-est-2}, we have singularity of $u$ at $0$ is of order $\f{2}{q-1}$. Consequently in this region we expect  $|\na u|\leq C|x|^{-(\f{q+1}{q-1})}$.
\end{remark}

\begin{theorem}\lab{t:moving pl}
Let $2^*-1\leq p \leq (p-1)\frac{N}{2}-1<q $ and $0\leq\mu<(\frac{N-2}{2})^2$. Then any positive solution $u\in D^{1,2}(\Rn)\cap L^{q+1}(\Rn)$ of Eq. \eqref{a3} is radially symmetric with respect to origin and radially decreasing.
\end{theorem}

To discuss the asymptotic behaviour of problem \eqref{eq:a3} for general domain, we first formulate a variational problem for \eqref{eq:a3}. Then we establish existence of variational solution $v_{\eps}$ for small positive values of $\eps$ and finally we derive the asymptotic behavior of $v_{\eps}$ as $\eps\to 0$, using variational arguments again. This is the first result for the problem with critical and supercritical exponent in the singular case and the case appears to be more complicated than the smooth case.

\begin{theorem}\label{main1}
There exists $\eps_n>0$ and $\la_n>0$ with $\eps_n\to 0$ as $n\to\infty$ and $\la_n$ uniformly bounded above and away from zero, such that
\begin{itemize}
\item[(i)] there exists a solution $u_n$ to Eq. \eqref{eq:a3} corresponding to $\eps=\eps_n$;
\item[(ii)] if $p>2^*-1$, then $F(\la_n u_n)\to \mathcal{K}$ and $\displaystyle\int_{\Om} |x|^{-(p+1)\nu} u_n^{p+1}dx\to 0$ as $n\to\infty$;
\item[(iii)]if $p=2^*-1$, then $S(\la_n u_n)\to\mathcal{S}$ as $n\to\infty$ and there exist  constants $A, B>0$ such that for all $n\geq 1$, it holds $A<\displaystyle\int_{\Om} |x|^{-(p+1)\nu} u_n^{p+1}dx< B$ ,
\end{itemize}
where $F(.)$,\, $\mathcal{K}$ and $S(.)$, $\mathcal{S}$ are defined as in \eqref{a22}, \eqref{14-K} and \eqref{sep-17-1}, \eqref{go} respectively.
\end{theorem}

\begin{theorem}\label{main2}  Let  $\nu\in(0,\f{N-2}{4})$, $2^*-1=p<q<\frac{1+\nu}{\nu}$ and \\ $v_{\eps}\in H^{1}_{0}( \Om, |x|^{-2\nu})$ be a  solution  of Eq. (\ref{eq:a3})  such that
 $$S(\la_{\eps}v_{\eps}) \to \mathcal{S} \quad\text{and}\quad
 A<\displaystyle\int_{\Om} |x|^{-(p+1)\nu} v_{\eps}^{p+1}dx< B,$$ where $S(.)$, $\mathcal{S}$ are as in \eqref{sep-17-1} and \eqref{go} respectively. Moreover, assume  \eqref{sup} is satisfied. Then along a subsequence
\begin{eqnarray*} & \lim_{\eps \rightarrow 0}\eps \|v_{\eps}\|_{\infty}^{\frac{q(N-2)-(N+2)+2\al}{\al}}\\
&=\frac{\om_N|R(0)|}{C_{q,N}}\frac{(N-2-2\nu)^\frac{\big(N-(q+1)\nu\big)(N-2)-4N\nu}{\al}(N-2)^\frac{\big(N-(q+1)\nu\big)(N-2)-2\al(N-1)}{2\al}}{N^\frac{\big(N-(q+1)\nu-2\al\big)(N-2)}{2\al}}\\
&\times\bigg[B\bigg(\frac{N-2}{2\al}(N-(q+1)\nu), \frac{N-2}{2\al}\{q(N-2-\nu)-(2+\nu)\}\bigg)\bigg]^{-1},
\end{eqnarray*}
where \be\label{C_qn}C_{q,N}=\frac{(N-2)q-(N+2)}{2(q+1)},\ee $R(0)$ and $B(a,b)$ are as defined in \eqref{Robin} and \eqref{eq:beta} respectively. Furthermore, for $x\neq 0$,
\begin{equation} \label{nsd2} \lim_{\eps \rightarrow 0} v_{\eps}(x) \|v_{\eps}\|_{\infty} = \om_{N}(N-2-2\nu)^{N-1}\bigg(\frac{N}{N-2}\bigg)^\frac{N-2}{2}G(x, 0),
\end{equation}
where $G(x,0)$ is the Green function as defined in \eqref{green}.
\end{theorem}

\begin{remark} Now we  point out the difference between the supercritical and subcritical case. First we notice there is a critical exponent $q^*:=\frac{2+\nu}{\nu}$ which plays a huge role in  determining  the singularity of solution \eqref{eq:a3'}.  This implies that there is some competition between the $\mu$ and $q$ (or equivalently between $\nu$ and $q$) which never arise in the subcritical case.
\end{remark}
\vspace{3mm}

\begin{remark}  In a forthcoming paper, we show this phenomena holds for the fractional laplacian case with $\mu=0.$\\
\end{remark}

{\bf Notation:} Throughout this paper $C$ denotes the generic constants which may vary from line to line. Below are few notations which we use throughout the paper:
\begin{itemize}
\item$\bar\mu:=\displaystyle\left(\f{N-2}{2}\right)^2$
\item $\nu:=\sqrt{\bar\mu}-\sqrt{\bar\mu-\mu}$
\item$\al:=N-2-2\nu$
\item $\om_N:=$ surface measure of unit ball.
\end{itemize}

\section{ Existence and non-existence of entire solution}
In this section, we will study the existence and non-existence result of entire problem with critical and supercritical exponents.
 We first establish the general Pohozaev identity which will also be used in the next sections.

\begin{proposition}\label{p:Poho} Let $\Om$ be a smooth domain, $0\in\Omega$,  $0\leq \mu<\bar\mu$ , $N\geq 3$, $2^*-1\leq p<q$  and $u$ be a solution of
\begin{equation}
  \label{eq:poho-1}
\left\{\begin{aligned}
      -\Delta u-\mu \frac{u}{|x|^2}&=u^{p} -\eps u^{q} \quad\text{in }\quad \Om,\\ u&>0, \\
      u &\in D^{1,2}(\Om)\cap L^{q+1}(\Om),
         \end{aligned}
  \right.
\end{equation}
Then $u$ satisfies:
\bea \label{poh1'}&&\frac{1}{2}\int_{\pa \Om}  |\nabla u|^2 \langle x, n \rangle dS + \frac{N-2}{2}\int_{ \pa \Om}  u \frac{\pa u}{\pa n}dS+\frac{\mu}{2}\int_{\pa\Om}\frac{u^2}{|x|^2} \langle x, n \rangle dS
  \no\\&= & \frac{\eps}{q+1}\int_{\pa \Om} u^{q+1} \langle x, n \rangle dS -\eps \bigg(\frac{N}{q+1}-\frac{N-2}{2}\bigg) \int_{\Om} u^{q+1} dx \no\\&+& \bigg(\frac{N}{p+1}-\frac{N-2}{2}\bigg)\int_{\Om}  u^{p+1}dx- \frac{1}{p+1}\int_{\pa \Om} u^{p+1} \langle   x, n \rangle  dS.
\eea
In particular, if $u=0$ on $ \pa \Om$ we have
\bea \label{poh2'} && \frac{1}{2} \int_{\pa \Om}  |\nabla u|^2\langle x, n \rangle dS \no \\&=& \bigg(\frac{N}{p+1}- \frac{N-2}{2}\bigg)\int_{\Om}  u^{p+1} dx+ \eps  \bigg( \frac{N-2}{2}- \frac{N}{q+1}\bigg) \int_{\Om}  u^{q+1} dx .
\eea
\end{proposition}

\begin{proof}
We multiply Eq. \eqref{eq:poho-1} by a suitable test function and to make the test function smooth we introduce cut-off functions and then pass to the limit.

For $\de>0$ and $R>0$, we define $\phi_{\de,R}(x)=\phi_{\de}(x)\psi_{R}(x)$ where
$\phi_{\de}(x)=\phi(\frac{|x|}{\de})$ and $\psi_R(x)=\psi(\frac{|x|}{R})$, $\phi$ and $\psi$ are smooth functions in $\R$ with the properties $0\leq\phi,\psi\leq 1$, with supports of $\phi$ and $\psi$ in $(1,\infty)$ and $(-\infty, 2)$ respectively and $\phi(t)=1$ for $t\geq 2$, and $\psi(t)=1$ for $t\leq 1$.

Let $u$ be a solution of Eq. \eqref{eq:poho-1}. Then $u$ is smooth away from the origin and hence $(x\cdot\na u)\phi_{\de,R}\in C^2_c(\Om)$. Multiplying Eq.\eqref{eq:poho-1} by this test function and integrating by parts, we obtain
\begin{eqnarray} \label{non-1}
&\displaystyle\int_{\Om}\na u\cdot\na\left((x\cdot\na u)\va_{\de,R}\right)dx-\mu\int_{\Om}\frac{u(x\cdot\na u)\va_{\de,R}}{|x|^2}dx -\int_{\pa\Om}\frac{\pa u}{\pa n}(x\cdot\na u)\phi_{\de, R} dS \no\\
&=\displaystyle\int_{\Om} (u^{p}-\eps u^q)(x\cdot\na u)\va_{\de,R}dx.
\end{eqnarray}
Now the RHS of \eqref{non-1} can be simplified as
\begin{eqnarray*}
RHS &=&-\frac{N}{p+1}\int_{\Om} u^{p+1}\va_{\de, R}dx-\frac{1}{p+1}\int_{\Om} u^{p+1}\big[x\cdot(\psi_R\na\va_\de+\va_\de \na\psi_R)]dx\\
&+&\eps\frac{N}{q+1}\int_{\Om} u^{q+1}\va_{\de, R}dx+\frac{\eps}{q+1}\int_{\Om} u^{q+1}\big[x\cdot(\psi_R\na\va_\de+\va_\de\na\psi_R)]dx\\
&+&\frac{1}{p+1}\int_{\pa\Om}u^{p+1}\langle x, n\rangle\phi_{\de, R} dS-\frac{\eps  }{q+1}\int_{\pa\Om}u^{q+1}\langle x, n\rangle\phi_{\de, R}dS.
\end{eqnarray*}
Note that $|x\cdot(\psi_R\na\va_\de+\va_\de \na\psi_R)|\le C $ and hence using the dominated convergence theorem we get,
\begin{eqnarray}\label{non-2}
 \lim_{R\rightarrow\infty}\big[\lim_{\de\rightarrow 0}RHS ] &=&-\frac{N}{p+1}\int_{\Om} u^{p+1}dx+\frac{N\eps }{q+1}\int_{\Om}u^{q+1}dx \no\\ &+&\frac{1}{p+1}\int_{\pa\Om}u^{p+1}\langle x, n\rangle dS- \frac{\eps}{q+1}\int_{\pa\Om}u^{q+1}\langle x, n\rangle dS .
\end{eqnarray}
By a direct calculation and integration by parts, LHS of \eqref{non-1} simplifies as,
\begin{eqnarray}\label{non-3}
\text{LHS} &=&\int_{\Om} |\na u|^2\va_{\de,R}+\sum_{i=1}^n\sum_{j=1}^n\frac{1}{2}\int_{\Om}((u_{x_i})^2)_{x_j} x_j \va_{\de,R}+\int_{\Om}(x\cdot\na u)(\na u\cdot\na\va_{\de,R})\no\\
&+& \frac{\mu N}{2}\int_{\Om}\frac{u^2}{|x|^2}\va_{\de, R}dx+\frac{\mu}{2}\int_{\Om}\frac{u^2}{|x|^2}(x\cdot\na\va_{\de, R})dx-\mu\int_{\Om}\frac{u^2}{|x|^2}\va_{\de, R}dx\no\\
&-&\int_{\pa\Om}|\na u|^2\langle x, n\rangle\phi_{\de, R} dS-\frac{\mu}{2}\int_{\pa\Om}\frac{u^2}{|x|^2}\langle x, n\rangle \phi_{\de, R} dS\no\\
&=& -\frac{N-2}{2}\displaystyle\left(\int_{\Om}|\na u|^2\va_{\de,R}-\mu\int_{\Om}\frac{u^2}{|x|^2}\va_{\de,R}\right)dx -\frac{1}{2}\int_{\pa\Om}|\na u|^2\langle x, n\rangle\phi_{\de, R} dS\no\\
&\ &-\frac{1}{2}\int_{\Om}\displaystyle\left(|\na u|^2-\mu\frac{u^2}{|x|^2}\right)\big[(x\cdot\na\va_{\de})\psi_R+(x\cdot\na\psi_R)\va_{\de}]dx\no\\
&\ &+\int_{\Om}(x\cdot\na u)\big[(\na u\cdot\na \va_{\de})\psi_R+(\na u\cdot\na\psi_R)\va_{\de}]dx-\frac{\mu}{2}\int_{\pa\Om}\frac{u^2}{|x|^2}\langle x, n\rangle \phi_{\de, R} dS.
\end{eqnarray}
Also we note that,
\begin{eqnarray*}
 \lim_{R\rightarrow\infty}\lim_{\de \rightarrow 0}|\int_{\Rn}|(x\cdot\na u)(\na u\cdot\na\va_{\de})\psi_R)|dx &\leq& C\lim_{\de\to 0}\int_{\de\leq|x|\leq 2\de}|\na u|^2\frac{|x|}{\de}dx\\
 &\leq& 2C\lim_{\de\to 0}\int_{\de\leq|x|\leq 2\de}|\na u|^2dx=0.
\end{eqnarray*}
Similarly
$$\lim_{R\rightarrow\infty}\lim_{\de\rightarrow 0}|\int_{\Rn}|(x\cdot\na u)(\na u\cdot\na\psi_R)\va_{\de})|dx \leq C\lim_{R\to\infty}\int_{R\leq|x|\leq 2R}|\na u|^2\frac{|x|}{R}dx=0.$$
Using the above estimates and taking the limit using dominated convergence theorem and using the fact $|x\cdot(\psi_R\na\va_\de+\va_\de \na\psi_R)|\le C $ , we get from \eqref{non-3},
\begin{eqnarray}\label{non-4}
 \lim_{R\rightarrow\infty}\displaystyle[\lim_{\de\rightarrow 0}LHS ] &=& -\frac{N-2}{2}\left(\int_{\Om}|\na u|^2-\mu\frac{u^2}{|x|^2}\right)dx \no\\
 &-& \frac{1}{2}\int_{\pa\Om}|\na u|^2\langle x, n\rangle dS
 -\frac{\mu}{2}\int_{\pa\Om}\frac{u^2}{|x|^2}\langle x, n\rangle dS.
\end{eqnarray}
Moreover, multiplying the Eq. \eqref{eq:poho-1} by $u$, we have
\begin{equation}\label{non-4'}
\int_{\Om}|\na u|^2 dx-\int_{\pa\Om}u(\na u\cdot n)dS-\mu\int_{\Om}\frac{u^2}{|x|^2}dx=\int_{\Om}(u^{p+1}-\eps u^{q+1}) dx
\end{equation}

Substituting \eqref{non-2} and \eqref{non-4} in \eqref{non-1} and using \eqref{non-4'}  we get
\begin{eqnarray}
&-&\frac{N-2}{2}\displaystyle\left(\int_{\Om}u^{p+1}dx-\eps\int_{\Om}u^{q+1}dx+\int_{\pa\Om}u
\frac{\pa u}{\pa n} dS\right)-\frac{1}{2}\int_{\pa\Om}|\na u|^2\langle x, n\rangle dS \no\\
&-&\frac{\mu}{2}\int_{\pa\Om}\frac{u^2}{|x|^2}\langle x, n\rangle dS  =-\frac{N}{p+1}\int_{\Om}u^{p+1}dx+\frac{N\eps}{q+1}\int_{\Om}u^{q+1}dx\no\\
&+& \frac{1}{p+1}\int_{\pa\Om}u^{p+1}\langle x, n\rangle dS- \frac{\eps}{q+1}\int_{\pa\Om}u^{q+1}\langle x, n\rangle dS .
\end{eqnarray}
This implies \eqref{poh1'}. If $u=0$ on $\pa\Om$, it is easy to see that \eqref{poh2'} follows from \eqref{poh1'}.
\end{proof}

\vspace{2mm}

\begin{proof}[{\bf Proof of Theorem \ref{1.1}}]
If $u$ is a  solution of Eq.\eqref{eq:a1}, then it follows from Proposition \ref{p:Poho} that
 $$\displaystyle\left(\frac{N-2}{2}-\frac{N}{q+1}\right)\int_{\Rn}u^{q+1}dx=0,$$ which is a contradiction as $q>2^*-1$ and $u>0$. This proves the theorem.
\end{proof}

\begin{proof}[\bf Proof of Theorem \ref{1.2}]
We are going to work on the manifold $$N= \bigg\{ u\in D^{1, 2}(\R^N) \cap L^{q+1}(\R^N):  \int_{\R^N} u^{p+1}dx=1 \bigg\}.$$
Then $F$ reduces to
$$F(u)=  \frac{1}{2}\int_{\R^N} |\nabla u|^2 dx -   \frac{\mu}{2} \int_{\R^N} \frac{| u|^2}{|x|^2} dx+  \frac{1}{q+1}\int_{\R^N}u^{q+1} dx . $$
Let
\be \mathcal{K}= \inf_{N} F(u).\ee
Let $u_{n} $ be a minimizing sequence in $N$ such that
$$F(u_{n}) \to \mathcal{K} \text{ with } \int_{\R^N} u_{n}^{p+1} dx =1.$$

As $\mu<\overline{\mu}$ implies $||u||:=\displaystyle\left(\int_{\Rn}|\na u|^2-\mu\frac{u^2}{|x|^2}\right)^\frac{1}{2}$ is an equivalent norm in $D^{1,2}(\Rn)$, we have $\{u_n\}$ is a bounded sequence in $D^{1,2}(\Rn)$ and $L^{q+1}(\Rn)$. Therefore there exists $u\in D^{1,2}(\Rn)$ and $L^{q+1}(\Rn)$ such that $u_n \rightharpoonup u$ in $D^{1,2}(\Rn)$ and $L^{q+1}(\Rn)$. Consequently $u_{n}\to u$ pointwise almost everywhere.

Using symmetric rearrangement technique, without loss of generality we can assume that $u_n$ is radially symmetric. Hence
 $u_{n}(x)= u_{n}(r)$, where $r=|x|$, and we can write
\be \no  u_{n}(r)=- \int_{r}^{\infty} u'_{n}(s) ds.\ee
Using a standard argument it can be shown that $u_n$ satisfies Strauss type  uniform estimate
 \be\label{strauss}  |u_{n}(r)|\leq C r^{-\frac{N-2}{2}} \ee
 for some $C>0$.
We claim that $u_{n} \to u$ in $L^{p+1}(\R^N).$\\
To see the claim, we note that $u_{n}^{p+1}\to u^{p+1}$ pointwise almost everywhere.  Since $\{u_n\}$ is uniformly bounded in $L^{q+1}(\Rn)$, using Vitali's convergence theorem, it is easy to check that $\displaystyle\int_{K} u_{n}^{p+1}dx\to \int_{K}  u^{p+1}dx$  for any compact set $K$  in $\R^N$ containing the origin.
Furthermore,  $\displaystyle\int_{\R^N \setminus K}u_{n}^{p+1} dx$ is very small  by \eqref{strauss} and hence we have strong convergence. Moreover, $\displaystyle\int_{\R^N}u_{n}^{p+1}dx=1 $  implies
$\displaystyle\int_{\R^N }u^{p+1} dx=1.$\\

Now we show that $\mathcal{K}= F(u).$

We note that $u\mapsto ||u||^2$ is weakly lower semicontinuous. Therefore using Fatou's lemma we can write

\begin{eqnarray*}\mathcal{K} &=&   \lim_{n \to \infty } \bigg[\frac{1}{2}\int_{\R^N} |\nabla u_{n}|^2 dx -   \frac{\mu}{2} \int_{\R^N} \frac{| u_{n}|^2}{|x|^2} dx+  \frac{1}{q+1}\int_{\R^N}u_{n}^{q+1} dx\bigg] \\
&=&\lim_{n \to \infty } \bigg[\frac{1}{2}||u_n||^2+\frac{1}{q+1}\int_{\R^N}u_{n}^{q+1} dx\bigg]\\
&\geq & \frac{1}{2}||u||^2+  \frac{1}{q+1}\int_{\R^N}u^{q+1} dx\bigg]\\
&\geq & F(u)   .\end{eqnarray*}
This proves $F(u)=\mathcal{K}$. Moreover, using the Schwartz symmetrisation method via. Polya-Szego inequality, it is easy to check that $u$ is radially symmetric and radially decreasing. Applying the Lagrange multiplier rule, we obtain $u$ satisfies
\be\no - \De u- \mu \frac{u}{|x|^2}+ u^q=  \lambda u^p, \ee
for some $\lambda>0$. This in turn implies
\be\no - \De u- \mu \frac{u}{|x|^2}=  \lambda u^p - u^q  ~~~~\text{ in } ~~~ \R^N.\ee
\end{proof}

\section{ Classification of singularity near $0$}
\subsection{Lower and upper estimate of solution}

In this subsection,  we study the asymptotic behavior of solutions (whenever exists) at origin of Eq.\eqref{a4}.

Following Lemma~\ref{l:4-1} and \ref{l:4-2} are crucially used to prove Theorem \ref{t:low-est}.
\begin{lemma}\lab{l:4-1}
Let $q<\f{2+\nu}{\nu}$ and $\nu\in (0, \f{N-2}{2})$. Then there exists $l>0$ (can be chosen small)  such that the following problem
\be\lab{12-4-3}
\left\{\begin{aligned}
&-\text{div}(|x|^{-2\nu}\na w)+|x|^{-(q+1)\nu}w^q =0 \quad\text{in}\quad B_l(0)\\
&w>0 \quad\text{in}\quad B_l(0)\\
&w\in H^{1}(B_l(0), |x|^{-2\nu})\cap L^{q+1}(B_l(0), |x|^{-(q+1)\nu}),
\end{aligned}
\right.
\ee
has a continuous radial solution $w_1$ such that $w_1(0)=1$.
\end{lemma}
\begin{proof}
To prove this lemma, it is enough to show that the following ODE has a unique solution $w_1$ in $(0,l)$ for some $l>0$ and $w_1$ is a solution of Eq.\eqref{12-4-3},
\begin{equation}\lab{21-4-9}
\left\{\begin{aligned}
w''+\f{N-1-2\nu}{r}w'(r) &=r^{-(q-1)\nu}w^q \quad\text{in}\quad (0,1)\\
w &>0 \quad\text{in}\quad (0,1)\\
w(0) &=1
\end{aligned}
\right.
\end{equation}
We can write a solution of the above ODE as
\be\lab{21-4-10}
w(r)=1+\int_{0}^{r}s^{2\nu+1-N}\int_{0}^{s}t^{N-1-(q+1)\nu}w^q(t)dt ds.\ee
Since $q<\f{2+\nu}{\nu}$, using Banach fixed point theorem, it is easy to check that solution of the integral equation \eqref{21-4-10} exists and unique in $(0, l)$ for some $l>0$. From  \eqref{21-4-10}, it follows $w$ is continuous in $[0,l]$ and
$$w'(r)=r^{2\nu+1-N}\int_0^s t^{N-1-(q+1)\nu}w^q(t)dt \quad\text{for}\quad r>0.  $$ Therefore by a straight forward computation it follows
\be\lab{22-4-1}
\int_{0}^l w'(r)^2 r^{N-1-2\nu}dr<\infty \quad\text{and}\quad \int_0^l w^{q+1}(r)r^{N-1-(q+1)\nu}<\infty
\ee
as $q<\f{2+\nu}{\nu}$ and $\nu<\f{N-2}{2}$. Define $w_1(x):=w(r)$, where $r=|x|$.\\

{\bf Claim:} $w_1$ is a weak solution of Eq.\eqref{12-4-3}.\\

Indeed by \eqref{22-4-1},  $w_1\in H^{1}(B_l(0), |x|^{-2\nu})\cap L^{q+1}(B_l(0), |x|^{-(q+1)\nu})$. Choose $0<\eta<l$ and define  $\chi_{\eta}\in C^{\infty}_0(B_l(0))$ such that $\chi_{\eta}=1$ for $|x|\leq\f{\eta}{2}$,  $\chi_{\eta}=0$ for $|x|>\eta$ and $|\na\chi_{\eta}|\leq \f{4}{\eta}$. Let $\phi\in C^{\infty}_0(B_l(0))$ be arbitrarily chosen. Set $D_{\eta}:=B_l(0)\setminus B_{\f{\eta}{2}}(0)$.
Therefore,
\bea\lab{22-4-3}
&&\displaystyle\int_{B_l(0)}|x|^{-2\nu}\na w_1\na\phi dx+\int_{B_l(0)}|x|^{-(q+1)\nu}w_1^q\phi dx \no\\
&&=\displaystyle\lim_{\eta\to 0}\int_{B_l(0)}\chi_{\eta}|x|^{-2\nu}\na w_1\na\phi dx+ \lim_{\eta\to 0}\int_{B_l(0)}\chi_{\eta}|x|^{-(q+1)\nu}w_1^q\phi dx\no\\
&&+\displaystyle\lim_{\eta\to 0}\int_{B_l(0)}(1-\chi_{\eta})(|x|^{-2\nu}\na w_1\na\phi+|x|^{-(q+1)\nu}w_1^q\phi)dx\no\\
&&=-\displaystyle\lim_{\eta\to 0}\int_{D_{\eta}}\big(\na(1-\chi_{\eta})
\na w_1\big) |x|^{-2\nu}\phi \no\\
&&-\displaystyle\lim_{\eta\to 0}\int_{D_{\eta}}(1-\chi_{\eta})\left(\text{div}(|x|^{-2\nu}\na w_1)-|x|^{-(q+1)\nu}w_1^q\right)\phi\, dx
\eea
Since $w_1$ is a solution of the ODE \eqref{21-4-9} in $(0,l)$ and it is $C^1$ away from $0$, it easily follows that $w_1$ is a $C^1$ solution of Eq.\eqref{12-4-3} in $D_{\eta}$, for every $\eta>0$. Thus the last integral in \eqref{22-4-3} equals $0$.
Furthermore,
$$|\displaystyle\lim_{\eta\to 0}\int_{D_{\eta}}\na(1-\chi_{\eta})
\na |x|^{-2\nu} w_1\phi|\leq \lim_{\eta\to 0} C\eta^{N}\cdot\eta^{-1-2\nu+2\nu+1-(q+1)\nu}=0.$$
Hence \eqref{22-4-3} yields $$\displaystyle\int_{B_l(0)}|x|^{-2\nu}\na w_1\na\phi\, dx+\int_{B_l(0)}|x|^{-(q+1)\nu}w_1^q\phi\, dx=0,$$
which in turn proves the claim. This completes the proof of the lemma.
\end{proof}

\begin{remark}
It is easy to see that \eqref{21-4-10} is related to a 2nd order ODE and solving this ODE requires two initial/boundary conditions. In our case it is natural to have initial  values on $u(0)$ and $u'(0)$. But it is not hard to see from \eqref{21-4-10} that $u'(0)$ is not defined for $\f{1+\nu}{\nu}\leq q<\f{2+\nu}{\nu}$. Therefore a standard ODE technique does not give existence of solution here. Moreover, as the solution of the integral equation \eqref{21-4-10} is not differentiable at $0$, it does not directly follow that $w$ is a solution of the given PDE \eqref{12-4-3}.
\end{remark}

\begin{lemma}\lab{l:4-2}
Let $m>0$ ,  $q<\f{2+\nu}{\nu}$ and $\nu\in (0, \f{N-2}{2})$. Then for some
$\de\in(0, 1)$,
there exists a radial continuous solution $w_{\de}$ of Eq. \eqref{12-4-3} in $B_{l}(0)$,
where $l$ is as in Lemma \ref{l:4-1}, with the property that $w_{\de}(0)=\de$ and $w_{\de}|_{\pa B_l(0)}<m$ and $\int_0^l |w'_{\de}(r)|^2r^{N-1-2\nu}dr<\infty$.
\end{lemma}
\begin{proof}
Given $\de>0$, let $w_{\de}$ be the solution of
\begin{equation}\label{12-4-5}
\left\{\begin{aligned}
&u''+\f{N-1-2\nu}{r}u'(r)=r^{-(q-1)\nu}u^q \quad\text{in}\quad (0, l_{\de}),\\
&u>0 \quad\text{in}\quad (0,l_{\de}),\\
&u(0) =\de,
\end{aligned}
\right.
\end{equation}
where $[0, l_{\de})$ is maximum neighbourhood of $0$ where the solution exists. Due to local existence, we have $l_\de>0$ (see for instance, Lemma \ref{l:4-1}).  Moreover, we can write the solution as
\be\lab{12-4-8}
w_{\de}(r)=\de+\int_{0}^{r}s^{2\nu+1-N}\int_{0}^{s}t^{N-1-(q+1)\nu}w_{\de}^q(t)dt ds.
\ee
\paragraph{\bf Claim 1}: If $0<\de_1<\de_2\leq 1$, then $w_{\de_1}\leq w_{\de_2}\leq w_1$ in $[0, l]$, where $w_1$ and $l$ are as in Lemma \ref{l:4-1}.\\
To see the claim, let $0<\de_1<\de_2\leq 1$. Since $w_{\de_1}(0)<w_{\de_2}(0)$, there exits $r_0>0$ such that $w_{\de_1}< w_{\de_2}$ in $[0, r_0]$. Define
$$S:= \{s\in[0,l]: w_{\de_1}(s)>w_{\de_2}(s)\}.$$
If $S=\emptyset$, then we are done. Suppose $S\not=\emptyset$. We define
$$\tilde r_0:=\inf S.$$
Clearly $\tilde r_0>0$. We show that   $\tilde r_0\not<l$.
Indeed, from \eqref{12-4-8}, we have
$$w_{\de_1}'(r)-w_{\de_2}'(r)=r^{2\nu+1-N}\int_{0}^{r}t^{N-1-(q-1)\nu}[w^q_{\de_1}(t)-w^q_{\de_2}(t)]dt.$$
Therefore $(w_{\de_1}-w_{\de_2})'(r)< 0$ for $r\in[0, \tilde{r}_0]$.
This implies $w_{\de_1}(\tilde{r}_0)< w_{\de_2}(\tilde{r}_0)$,
which is a contradiction to the definition of $\tilde r_0$. Hence the claim follows.

\paragraph{\bf Claim 2}: $w_{\de}\to 0$ uniformly in $[0,l]$, as $\de\to 0$.\\
Note that $\lim_{\de\downarrow 0}w_{\de}$ exists, since $w_{\de}>0$ and Claim 1 holds. Let $w:=\lim_{\de\downarrow 0}w_{\de}$. Using monotone convergence theorem, we pass the limit in \eqref{12-4-8} to obtain $$w(r)=\int_{0}^{r}s^{2\nu+1-N}\int_{0}^{s}t^{N-1-(q+1)\nu}w^q(t)dt ds.$$  Solution of this integral equation uniquely exists in $(0,l)$ (see for instance Lemma \ref{l:4-1}). Therefore $w=0$. Hence the claim follows by Dini's theorem. \\
Combining Claim 1 and Claim 2, the lemma follows.
\end{proof}

\vspace{3mm}

\begin{proof}[\bf Proof of Theorem \ref{t:low-est}]
Define, $v=|x|^{\nu}u$. Then it follows from \cite[Theorem 1.1]{Han} that $v\in H^1_0(\Omega, |x|^{-2\nu} )$ and $v$ satisfies the following equation:
\begin{equation}\label{12-4-1}
\left\{\begin{aligned}
-\text{div}(|x|^{-2\nu}\na v)&=&|x|^{-(p+1)\nu}v^p-|x|^{-(q+1)\nu}v^q \quad\text{in}\quad\Om,\\
v&>&0 \quad\text{in}\quad\Om,\\
v&\in& H^{1}_0(\Om, |x|^{-2\nu} )\cap L^{q+1}(\Om, |x|^{-(q+1)\nu}).
\end{aligned}
\right.
\end{equation}
By elliptic regularity theory $v\in C^2(\Omega\setminus\{0\})\cap C^1(\bar\Omega\setminus\{0\})$ (see \cite{GT}, \cite{Han}). It is easy to see that  $v$ is a super-solution of the following problem
\begin{eqnarray}\lab{12-4-2}
-\text{div}(|x|^{-2\nu}\na w)+|x|^{-(q+1)\nu}w^q &=&0 \quad\text{in}\quad B_l(0),\no\\
w&=&m \quad\text{on}\quad\pa B_l(0),\no\\
w&>&0 \quad\text{in}\quad B_l(0),\no\\
w&\in& H^{1}(B_l(0), |x|^{-2\nu} )\cap L^{q+1}(B_l(0), |x|^{-(q+1)\nu}),
\end{eqnarray}
where  $l>0$ is as in Lemma \ref{l:4-1} and $0<m<m_l=\min_{|x|=l}v$.\\
{\bf Claim}: If $w$ is any solution of \eqref{12-4-2}, then  $v\geq w$ in $B_l(0)$.\\
To see the claim, we note that $(v-w)$ satisfies
$$-\text{div}(|x|^{-2\nu}\na(v-w))\geq -|x|^{-(q+1)\nu}A(x)(v-w)\quad\text{in}\quad B_l(0),$$
where $0\leq A(x):=\f{v^q(x)-w^q(x)}{v(x)-w(x)}\leq q\,\text{max}[v(x), w(x)]^{q-1}$.
Moreover, $w\leq v$ on $\pa B_l(0)$. Thus taking $(v-w)^-$ as the test function we obtain
$$\int_{B_l(0)}|x|^{-2\nu}|\na(v-w)^-|^2dx+\int_{B_l(0)}|x|^{-(q+1)\nu}A(x)|(v-w)^-|^2dx\leq 0,$$ which implies $v\geq w$ in $B_l(0)$.\\
By Lemma \ref{l:4-2}, it follows that  Eq. \eqref{12-4-2} admits a solution $w_{\de}$ with $w_{\de}(0)=\de>0$. As a result $\lim_{|x|\to 0}v(x)\geq \de$, which in turn implies $$u(x)\geq c|x|^{-\nu}, \quad x\in B_{r_0}(0)\setminus \{0\},$$ for some $r_0>0$ small.
\end{proof}

\vspace{3mm}

\begin{proof}[\bf Proof of Theorem \ref{t:up-est}]
We prove this theorem in the spirit of \cite{Han}.\\
Define,
\begin{equation}\label{v-f}
v(x)=|x|^{\nu} u(x) \quad\text{and}\quad f(x, u)=u^p-u^q.
\end{equation}

Then Eq.\eqref{a4} reduces to
\begin{equation}\label{eq-1}
-div(|x|^{-2\nu}\na v)=|x|^{-(p+1)\nu}v^p-|x|^{-(q+1\nu)}v^q\quad \forall \   \ x\in \Omega\setminus\{0\}.
\end{equation}
By elliptic regularity theory $v\in C^2(\Omega\setminus\{0\})\cap C^1(\bar\Omega\setminus\{0\})$ (see \cite{GT}, \cite{Han}). Let $\rho>0$ small enough such that $B_{\rho}(0)\Subset\Omega$. For $s, l>1$, we choose the test function $\va$ as follows:
$$\va=\eta^2 v v_l^{2(s-1)}\in H^1_0(\Omega, |x|^{-2\nu}dx),$$
$$v_l=\text{min}\{v, l\}, \ \ \eta\in C^{\infty}_0(B_{\rho}(0)),$$
with the properties $0\leq\eta\leq 1$, $\eta=1$ in $B_r(0)$, $r<\rho$ and $|\na\eta|\leq\frac{4}{\rho-r}$. Using this test function $\va$, we obtain from \eqref{eq-1},
\begin{equation}\label{eq-2}
\Iom |x|^{-2\nu}\na v\na\va dx=\Iom \big(|x|^{-(p+1)\nu}v^p-|x|^{-(q+1\nu)}v^q\big)\va dx.
\end{equation}
Substituting the function $f$, RHS of \eqref{eq-2} can be simplified as below
\begin{equation}\label{eq-3}
RHS=\Iom |x|^{-(p+1)\nu}\eta^2v^{p+1}v_l^{2(s-1)}dx-\Iom |x|^{-(q+1)\nu}\eta^2v^{q+1}v_l^{2(s-1)}dx.
\end{equation}
After doing a standard computation, the LHS of \eqref{eq-2} can be rewritten as:
\begin{equation}\label{eq-4}
\Iom |x|^{-2\nu}\times \displaystyle\left(2\eta v v_l^{2(s-1)}\na\eta\na v+\eta^2v_l^{2(s-1)}|\na v|^2+2(s-1)\eta^2v_l^{2(s-1)}|\na v_l|^2\right)dx.
\end{equation}
Using Cauchy-Schwartz inequality, for any $\epsilon>0$ we have,
\begin{eqnarray}\label{eq-5}
|2\Iom |x|^{-2\nu}\eta v v_l^{2(s-1)}\na\eta\na v dx|&\leq&\epsilon\Iom |x|^{-2\nu}\eta^2v_l^{2(s-1)}|\na v|^2 dx\no\\
&+&C(\epsilon)\Iom |x|^{-2\nu}|\na\eta|^2|v|^2v_l^{2(s-1)} dx.
\end{eqnarray}
Combining  \eqref{eq-2}-- \eqref{eq-5} we obtain,
\begin{eqnarray}\label{eq-6}
&\displaystyle\Iom |x|^{-2\nu}\left(\eta^2v_l^{2(s-1)}|\na v|^2+2(s-1)\eta^2v_l^{2(s-1)}|\na v_l|^2\right)dx \no\\
&\leq C\displaystyle\Iom |x|^{-2\nu}|\na \eta|^2v^2v_l^{2(s-1)}dx\no\\
&+\displaystyle \Iom |x|^{-(p+1)\nu}\eta^2v^{p+1}v_l^{2(s-1)}dx \no\\
&-\displaystyle\Iom |x|^{-(q+1)\nu}\eta^2v^{q+1}v_l^{2(s-1)}dx.
\end{eqnarray}
We recall here Caffarelli-Kohn-Nirenberg inequality (see \cite{CKN}):
\begin{equation}\label{CKN}
\displaystyle\left(\Iom |x|^{-br}|w|^r dx\right)^\frac{2}{r}\leq C_{a,b}\Iom |x|^{-2a}|\na w|^2 dx \quad\forall\ \ w\in H^1_0(\Omega, |x|^{-2a}dx),
\end{equation}
where $-\infty<a<\frac{N-2}{2}$, $a\leq b\leq a+1$, $r=\frac{2N}{N-2+2(b-a)}$ and $C_{a,b}$ is a positive constant.\\
Let $w=\eta v v_l^{s-1}$ and $a=b=\nu<\frac{N-2}{2}$ in \eqref{CKN}. Then $r=2^*$. Consequently we get from \eqref{CKN},
\begin{equation}\label{eq-7}
\displaystyle\left(\Iom |x|^{-2^*{\nu}}|\eta v v_l^{s-1}|^{2^*}dx\right)^\frac{2}{2^*}\leq C_{a,a}\Iom |x|^{-2{\nu}}|\na(\eta v v_l^{s-1})|^2 dx.
\end{equation}
Using \eqref{eq-6}, we simplify the RHS of \eqref{eq-7} as in \cite{Han}, i.e.,
\begin{eqnarray}\label{eq-8}
 RHS &\leq& 2C_{a,a}\displaystyle\Iom |x|^{-2{\nu}}\no\\
 &\times& \left(|\na\eta|^2v^2v_l^{2(s-1)}+\eta^2v_l^{2(s-1)}|\na v|^2+(s-1)^2\eta^2v_l^{2(s-1)}|\na v_l|^2\right)dx\no\\
&\leq& Cs\displaystyle\Iom |x|^{-2{\nu}}|\na\eta|^2v^2v_l^{2(s-1)}
+Cs\displaystyle \Iom |x|^{-(p+1){\nu}}\eta^2v^{p+1}v_l^{2(s-1)}dx \no\\
&-&\displaystyle\Iom |x|^{-(q+1){\nu}}\eta^2v^{q+1}v_l^{2(s-1)}dx\no\\
&\leq& Cs\displaystyle\Iom |x|^{-2{\nu}}|\na\eta|^2v^2v_l^{2(s-1)}
+Cs\displaystyle \Iom |x|^{-(p+1){\nu}}\eta^2v^{p+1}v_l^{2(s-1)}dx.
\end{eqnarray}
For $p\geq 2^*-1$, choose $t>1$ as follows:
\begin{equation}\label{eq:p-t}
\frac{N}{2}<t<\frac{q+1}{p-1}.
\end{equation}
Note that for $p=2^*-1$ the interval $(\f{N}{2}, \f{q+1}{p-1})$ is always a nonempty set. On the other hand, $(\f{N}{2}, \f{q+1}{p-1})\not=\emptyset$, since $q>(p-1)\f{N}{2}-1$. From \eqref{eq:p-t} we have,
$$(p-1)t<q+1 \quad\text{and}\quad 2<\frac{2t}{t-1}<2^*.$$ Consequently
\begin{eqnarray}\label{eq-9}
\displaystyle \Iom |x|^{-(p+1)\nu}\eta^2v^{p+1}v_l^{2(s-1)}dx &=&\Iom \eta^2u^{p+1}v_l^{2(s-1)} dx\no\\
&=&\Iom |\eta v v_l^{s-1}|^2u^{p-1}|x|^{-2{\nu}}dx\no\\
&\leq& |u|_{L^{(p-1)t}(\Omega)}^{p-1}||x|^{-{\nu}}\eta v v_l^{s-1}|^2_{L^\frac{2t}{t-1}(\Omega)}\no\\
&\leq&C\displaystyle\big(\epsilon||x|^{-{\nu}}\eta v v_l^{s-1}|_{L^{2^*}(\Omega)} \no\\
&+&C(N,t)\epsilon^{-\frac{N}{2t-N}}||x|^{-{\nu}}\eta v v_l^{s-1}|_{L^2(\Omega)}\big)^2\no\\
&\leq& C\epsilon^2\displaystyle\left(\Iom|x|^{-2^*{\nu}}|\eta v v_l^{s-1}|^{2^*}dx\right)^\frac{2}{2^*}\no\\
&+&C\epsilon^{-\frac{2N}{2t-N}}\displaystyle\Iom |x|^{-2\nu}|\eta v v_l^{s-1}|^2dx.
\end{eqnarray}
Plugging \eqref{eq-9} into \eqref{eq-8} and then  \eqref{eq-8} into \eqref{eq-7}, we have
\begin{eqnarray}\label{eq-10}
\displaystyle\left(\Iom |x|^{-2^*{\nu}}|\eta v v_l^{s-1}|^{2^*}dx\right)^\frac{2}{2^*} &\leq& Cs\Iom |x|^{-2{\nu}}|\na\eta|^2v^2v_l^{2(s-1)}dx\no\\
&+& Cs\epsilon^2\displaystyle\left(\Iom|x|^{-2^*{\nu}}|\eta v v_l^{s-1}|^{2^*}dx\right)^\frac{2}{2^*}\no\\
&+&Cs\epsilon^{-\frac{2N}{2t-N}}\displaystyle\Iom |x|^{-2{\nu}}|\eta v v_l^{s-1}|^2dx.
\end{eqnarray}
By choosing $\epsilon=\frac{1}{\sqrt{2Cs}}$, we obtain from \eqref{eq-10}
\begin{eqnarray}\label{eq-11}
\displaystyle\left(\Iom |x|^{-2^*{\nu}}|\eta v v_l^{s-1}|^{2^*}dx\right)^\frac{2}{2^*} &\leq& Cs\Iom |x|^{-2{\nu}}|\na\eta|^2v^2v_l^{2(s-1)}dx\no\\
&+&Cs^\frac{2t}{2t-N}\displaystyle\Iom |x|^{-2{\nu}}|\eta v v_l^{s-1}|^2dx\no\\
&\leq& Cs^{\al}\displaystyle\Iom |x|^{-2{\nu}}(\eta^2+|\na\eta|^2)v^2 v_l^{2(s-1)}dx,
\end{eqnarray}
where $\al=\frac{2t}{2t-N}$. Moreover, it is not difficult to check that \begin{equation*}
\Iom |x|^{-2^*{\nu}}|\eta v v_l^{s-1}|^{2^*}dx\geq \displaystyle\Iom |x|^{-2^*{\nu}}|\eta|^{2^*}v^2v_l^{2^*s-2}dx.
\end{equation*}
Consequently as in \cite{Han}, we have
\begin{eqnarray}\label{eq-12}
\displaystyle\left(\Iom |x|^{-2^*{\nu}}|\eta|^{2^*}v^2v_l^{2^*s-2}dx\right)^\frac{2}{2^*}
&\leq& Cs^{\al}\displaystyle\Iom |x|^{-2{\nu}}(\eta^2+|\na\eta|^2)v^2 v_l^{2(s-1)}dx\no\\
&\leq& Cs^{\al}\displaystyle\Iom |x|^{-2^{*}{\nu}}(\eta^2+|\na\eta|^2)v^2 v_l^{2(s-1)}dx.
\end{eqnarray}
Substituting $\eta$ and $\na\eta$ we deduce
\begin{equation}\label{eq-13}
\displaystyle\left(\int_{B_r(0)} |x|^{-2^*{\nu}}v^2v_l^{2^*s-2}dx\right)^\frac{2}{2^*}
\leq\frac{Cs^{\al}}{(\rho-r)^2}\int_{B_{\rho}(0)}|x|^{-2^*{\nu}}v^2v_l^{2s-2}dx.
\end{equation}
Set $s^*$ and $s_j$ as follows:
$$\frac{N}{N-2}<s^*<\frac{q+1}{2} \quad\text{and}\quad s_j=s^*\displaystyle\left(\frac{2^*}{2}\right)^j, \quad j=1,2,\cdots. $$
If we take $s=s_j$ in \eqref{eq-13}, a straight forward computation yields:
\begin{equation}\label{eq-14}
\displaystyle\left(\int_{B_r(0)} |x|^{-2^*{\nu}}v^2v_l^{2s_{j+1}-2}dx\right)^\frac{1}{2s_{j+1}}
\leq\displaystyle\left(\frac{Cs^{\al}}{(\rho-r)^2}\right)^\frac{1}{2s_j}\left(\int_{B_{\rho}(0)}|x|^{-2^*{\nu}}v^2v_l^{2s_j-2}dx\right)^\frac{1}{2s_j}.
\end{equation}
Choose $\rho_0>0$ such that $B_{2\rho_0}\Subset\Omega$ and $r_j=\rho_0(1+\rho_0^j), \ \ j=1,2, \cdots.$ By taking $\rho=r_j$, $r=r_{j+1}$ in \eqref{eq-14} and following the calculation of \cite{Han} we find:
\begin{eqnarray}\label{eq-15}
\displaystyle\left(\int_{B_{r_{j+1}}(0)} |x|^{-2^*{\nu}}v^2v_l^{2s_{j+1}-2}dx\right)^\frac{1}{2s_{j+1}}
&\leq&\displaystyle\left(\frac{C}{(1-\rho_0)\rho_0}\right)^{\sum_{j=0}^{\infty}\frac{1}{2s_j}-\sum_{j=0}^{\infty}\frac{j}{2s_j}}\times\no\\
&&\displaystyle\prod_{j=0}^{\infty}s_j^\frac{\al}{2s_j}\left(\int_{B_{r_0}(0)} |x|^{-2^*{\nu}}v^2v_l^{2s^*-2}dx\right)^\frac{1}{2s^*}.
\end{eqnarray}
By standard computation it follows that (see \cite{Han})
$$\displaystyle\sum_{j=0}^{\infty}\frac{1}{2s_j}\leq C, \quad \sum_{j=0}^{\infty}\frac{j}{2s_j}\leq C \quad\text{and}\quad \prod_{j=0}^{\infty}s_j^\frac{\al}{2s_j}\leq C. $$
Since $2^*<2s^*<q+1$, after a straight forward computation as in \cite{Han}, we obtain
\begin{equation*}
\int_{B_{r_0}(0)} |x|^{-2^*{\nu}}v^2v_l^{2s^*-2}dx
\leq(\text{diam}\,\Omega)^{(2s^*-2^*){\nu}}\Iom u^{2s^*}dx\leq C.
\end{equation*}
As a result, from \eqref{eq-15} we have
\begin{equation}\label{eq-16}
\displaystyle\left(\int_{B_{r_{j+1}}(0)} |x|^{-2^*{\nu}}v^2v_l^{2s_{j+1}-2}dx\right)^\frac{1}{2s_{j+1}}\leq C.
\end{equation}
Moreover,
\begin{eqnarray}\label{eq-17}
\text{LHS of \eqref{eq-16}}&\geq& \displaystyle\left(\int_{B_{r_{j+1}}(0)} |x|^{-2^*{\nu}}v_l^{2s_{j+1}}dx\right)^\frac{1}{2s_{j+1}}\no\\
&\geq&(\text{diam}\, \Omega)^\frac{-2^*{\nu}}{2s_{j+1}}|v_l|_{L^{2s_{j+1}}(B_{\rho_0}(0))}
\end{eqnarray}
Combining \eqref{eq-17} with \eqref{eq-16}, we obtain
$$|v_l|_{L^{2s_{j+1}}(B_{\rho_0}(0))}\leq C(\text{diam}\,\Omega)^\frac{2^*{\nu}}{2s_{j+1}}.$$
Note that $s_{j+1}\to\infty$ as $j\to\infty$. Hence $|v_l|_{L^{\infty}(B_{\rho_0}(0))}\leq C$. Finally letting $l\to\infty$ we have $|v|_{L^{\infty}(B_{\rho_0}(0))}\leq C$, which in turn implies
$$u(x)\leq C|x|^{-\nu} \quad\forall\ \ x\in B_{\rho_0}(0)\setminus\{0\}.$$
\end{proof}

\vspace{3mm}

\begin{proof}[\bf Proof of Theorem \ref{t:up-est-2}]
We use an idea from \cite{GLS}.
If $u$ is a positive solution of Eq. \eqref{a4}, then $u$ satisfies
$$-\De u-\mu\f{u}{|x|^2}=-(1+o(1))u^q, \quad\text{in}\quad B_{R}(0),$$ for some $R>0$ small. Using the transformation $v=|x|^{\nu}u$, we get $v$ satisfies
\be\lab{16-4-3'}
-\text{div}(|x|^{-2\nu}\na v)=-(1+o(1))|x|^{-(q+1)\nu}v^q, \quad\text{in}\quad  B_{R}(0).
\ee
Therefore we can write
\be\lab{16-4-3}
-\text{div}(|x|^{-2\nu}\na v)=-A|x|^{-(q+1)\nu}v^q, \quad\text{in}\quad  B_{R}(0),
\ee
where $1-\de<A<1+\de$, for some $\de>0$.\\

{\bf Claim:} $v(x)\leq C|x|^{\nu-\f{2}{q-1}} \quad\text{in}\quad  B_{\f{2R}{3}}(0)\setminus\{0\}$, for some $C=C(N,q,p,R,\mu)>0$.\\

\vspace{2mm}

To see the claim, for $0<r<R$, set $$y=\f{x}{r} \quad\text{and}\quad w(y)=r^{-\nu+\f{2}{q-1}}v(x).$$
Then $w$ satisfies Eq. \eqref{16-4-3} in $B_1(0)$.\\
Now define $$W(y):=c\displaystyle\left[\left(\f{9}{16}-|y|^2\right)\left(|y|^2-\f{1}{16}\right)\right]^{-\ba},$$
where $\ba>\f{2}{q-1}$ and $c>0$ will be chosen later. Clearly $$W=\infty \quad\text{on}\quad \pa \big(B_{\f{3}{4}}(0)\setminus B_{\f{1}{4}}(0)\big).$$
We show that $\ba$ and $c$ in the definition of $W$ can be chosen such that $$-\text{div}(|x|^{-2\nu}\na W)\geq -A|x|^{-(q+1)\nu}W^q, \quad\text{in}\quad  B_{\f{3}{4}}(0)\setminus B_{\f{1}{4}}(0).$$
Since $W$ is radial, it is enough to show that
\be\lab{16-4-4}
W''+\f{N-1-2\nu}{r}W'\leq A r^{-(q-1)\nu}W^q,   \quad \f{1}{4}<r<\f{3}{4}.
\ee
By a direct computation, when
$ \f{1}{4}<r<\f{3}{4}$, we obtain
\Bea
W''+\f{N-1-2\nu}{r}W' &=&-2W\ba\big[-\big(\f{9}{16}-r^2\big)^{-1}+\big(r^2-\f{1}{16}\big)^{-1}\big](N+2r^2-2\nu)\no\\
&+&4r^2W\ba\big[\big(\f{9}{16}-r^2\big)^{-2}+\big(r^2-\f{1}{16}\big)^{-2}\big]\no\\
&\leq&CW\ba\big[\big(\f{9}{16}-r^2\big)^{-2}+\big(r^2-\f{1}{16}\big)^{-2}\big]\no\\
&\leq&CW\ba\big[\big(\f{9}{16}-r^2\big)^{-2}\big(r^2-\f{1}{16}\big)^{-2}\big]\no\\
&=&C\ba W^{1+\f{2}{\ba}}\no
\Eea
Since $\ba>\f{2}{q-1}$ implies $1+\f{2}{\ba}<q$, \eqref{16-4-4} follows.
Therefore we obtain,
$$-\text{div}(|x|^{-2\nu}\na(W-w))\geq -A|x|^{-(q+1)\nu}B(x)(W-w)\quad\text{in}\quad B_{\f{3}{4}}(0)\setminus B_{\f{1}{4}}(0),$$
where $0\leq B(x):=\f{W^q(x)-w^q(x)}{W(x)-w(x)}\leq q\,\text{max}[W(x), w(x)]^{q-1}$.
Moreover, $(W-w)^-=0$ on $\pa(B_{\f{3}{4}}(0)\setminus B_{\f{1}{4}}(0))$. Thus taking $(W-w)^-$ as the test function we obtain
$$\int_{B_{\f{3}{4}}(0)\setminus B_{\f{1}{4}}(0)}|x|^{-2\nu}|\na(W-w)^-|^2dx+\int_{B_{\f{3}{4}}(0)\setminus B_{\f{1}{4}}(0)}A|x|^{-(q+1)\nu}B(x)|(W-w)^-|^2dx\leq 0,$$ which implies $w\leq W$ in $B_{\f{3}{4}}(0)\setminus B_{\f{1}{4}}(0)$.\\
In particular,
$$w(y)\leq \max_{\f{1}{3}<|y|<\f{2}{3}}W(y) \quad\text{in}\quad B_{\f{2}{3}}(0)\setminus B_{\f{1}{3}}(0),$$
which yields
$$\max_{\f{r}{3}<|x|<\f{2r}{3}}v(x)\leq C|x|^{\nu-\f{2}{q-1}} .$$
Since $0<r<R$ was arbitrarily chosen, the claim follows .

Hence $u(x)\leq C|x|^{-\f{2}{q-1}}$ in $B_{\f{2R}{3}}(0)\setminus\{0\}$. From Theorem \ref{t:up-est}, it also follows that $u(x)\leq C|x|^{-\nu}$. Since $q>\f{2+\nu}{\nu}$ implies $|x|^{-\f{2}{q-1}}\leq C|x|^{-\nu}$, the theorem follows.
\end{proof}

\vspace{3mm}

\begin{proof}[{\bf Proof of Theorem \ref{t:low-est-2}}]
If $u$ is a positive solution of Eq. \eqref{a4}, then as in the proof of Theorem \ref{t:up-est-2}, $u$ satisfies
$$-\De u-\mu\f{u}{|x|^2}=-(1+o(1))u^q, \quad\text{in}\quad B_{R}(0),$$ for some $R>0$ small. Using the transformation $v=|x|^{\nu}u$, we get $v$ satisfies
\be
-\text{div}(|x|^{-2\nu}\na v)=-(1+o(1))|x|^{-(q+1)\nu}v^q, \quad\text{in}\quad  B_{R}(0)\no.
\ee
Given $\de>0$, we can write
\be\lab{21-4-1}
-\text{div}(|x|^{-2\nu}\na v)=-A|x|^{-(q+1)\nu}v^q, \quad\text{in}\quad  B_{R}(0),
\ee
where $1-\de<A<1+\de$. Define
\be\lab{21-4-2}
V(x):= c|x|^{\nu-\f{2}{q-1}},
\ee
where $c<\min\{c_1, c_2\}$, $$c_1:= R^{-\nu+\f{2}{q-1}}\min_{|x|=R} v,$$
and $c_2$ is defined in \eqref{21-4-8}.
 Therefore, it is easy to see
\be\lab{21-4-3}
v\geq V \quad\text{on}\quad \pa B_R(0).
\ee

\noi {\bf Claim:} $-\text{div}(|x|^{-2\nu}\na V)\leq -A|x|^{-(q+1)\nu}V^q$ in   $B_{R}(0)$.\\
To prove the claim, we note that since $V$ is radial, it is enough to show that
$$V''+\f{N-1-2\nu}{r}V'\geq Ar^{-(q-1)\nu}V^q \quad r\in(0, R).$$
Using the Emden-Fowler transformation
\be\lab{21-4-4}
y(t)=\al^{\nu}V(r), \quad t=(\f{\al}{r})^{\al},
\ee
where $\al=N-2-2\nu$, it is equivalent to prove that
\be\lab{21-4-5}
y''(t)\geq At^\f{-(2\al+2)+(q-1)\nu}{\al}y^q(t), \quad t>(\f{\al}{R})^{\al}.
\ee
Using \eqref{21-4-2} in \eqref{21-4-4}, it is not difficult to see that $y(t)=c\al^{-\f{2}{q-1}}t^{-\f{1}{\al}(\nu-\f{2}{q-1})}$.
Consequently, by a straight forward computation we obtain,
\be\lab{21-4-6}
y''(t)=c\al^{-\f{q+1}{q-1}}\displaystyle\left(\nu-\f{2}{q-1}\right)\left[\big(\nu-\f{2}{q-1}\big)\f{1}{\al}+1\right]t^{-\big(\nu-\f{2}{q-1}\big)\f{1}{\al}-2}.
\ee
On the other hand, by direct computation it follows
\be\lab{21-4-7}
At^\f{-(2\al+2)+(q-1)\nu}{\al}y^q(t)=A(c\al^{-\f{2}{q-1}})^q t^{-\big(\nu-\f{2}{q-1}\big)\f{1}{\al}-2}.
\ee
Define
\be\lab{21-4-8}
c_2:=\displaystyle\left(\f{1}{A}\al \left(\nu-\f{2}{q-1}\right)\left[\big(\nu-\f{2}{q-1}\big)\f{1}{\al}+1\right]\right)^\f{1}{q-1}.
\ee
Note that $q>\f{2+\nu}{\nu}$ implies $\nu-\f{2}{q-1}>0$. Therefore,
since $c<\min\{c_1, c_2\}$,  comparing \eqref{21-4-6} and \eqref{21-4-7}, we
conclude \eqref{21-4-5} holds true. Hence the claim follows.

We also note that  both $v$ and $V$ are bounded in $B_R(0)$ (for $V$ it follows from Theorem \ref{t:up-est-2}). Therefore combining the Claim and \eqref{21-4-3} and using comparison principle as in previous theorem, we obtain $v\geq V$ in $B_R(0)$. This in turn implies, $u(x)\geq c|x|^{-\f{2}{q-1}}$ in $B_R(0)\setminus\{0\}$ which completes the proof.
\end{proof}

\vspace{3mm}

\subsection{The Critical Case $q=\f{2+\nu}{\nu}$}
\begin{proof}[{\bf Proof of Theorem \ref{t:critical}}]
Let $u$ be any radial solution of Eq.\eqref{a4} with $q=\f{2+\nu}{\nu}$. Note that this implies $\nu=\f{2}{q-1}$. Then as in  the proof of previous theorem, $v=r^{\nu}u$ satisfies
\be\lab{27-4-1'}
v''+\f{N-1-2\nu}{r}v'=Ar^{-2}v^q, \quad r\geq R,
\ee
where $1-\de<A<1+\de$, for some $\de>0$ (see \eqref{21-4-1}).
Using the Emden-Fowler transformation
\be\lab{27-4-1}
y(t)=\al^{\nu}v(r), \quad t=(\f{\al}{r})^{\al},
\ee
where $\al=N-2-2\nu$, \eqref{27-4-1'} reduces to
\be\lab{27-4-2}
t^2y''(t)-Ay^q(t)=0, \quad t>(\f{\al}{R})^{\al}.
\ee
{\bf Claim:} $y(t)\to 0$ as $t\to\infty$.\\
To see the claim, we note that for large $t$, $y'$ is increasing and nonnegative. From Theorem \ref{t:up-est}, we have $v$ is bounded near $0$. Therefore $y$ is bounded near infinity. Using this fact, it is easy to check that $\lim_{t\to\infty}y'(t)=0$. Consequently, $y'(t)\leq 0$ for large $t$ which implies $y$ is decreasing for large $t$. Hence
$\lim_{t\to\infty}y(t)=c<+\infty$.  If $c\neq 0$
$$y'(\theta)= -\int_{\theta}^{\infty} \frac{y^q}{s^2} ds,  \text{  implies  }  y(T)=y(t)+ \int_{T}^{t} \int_{\theta}^{\infty} \frac{y^q}{s^2} ds d\theta .$$
Hence we have
$$ y(T)\geq y(t)+ (\frac{c}{2})^q \log \frac{t}{T}. $$
Since $y$ is bounded, taking the limit $t\to\infty$ in the above expression yields a contradiction. Hence $c=0$ and the claim follows.

Setting
 \be\lab{28-4-1}
 t=e^s \quad\text{and}\quad x(s)=y(t),\ee  \eqref{27-4-2} yields
\be\lab{27-4-3}
x''(s)-x'(s)-Ax^q(s)=0 \quad s\geq R',
\ee
where $R'=\log\f{\al}{R}$. We are only interested in the solutions of \eqref{27-4-3},  $x(s) \to 0$ as $s \to \infty.$ Following an argument along the same line of \cite[Lemma 3.2]{Veron}, it can be shown that
\be\lab{27-4-4}
x(s)=\displaystyle\left(\f{1}{(q-1)s}\right)^\f{1}{q-1}\left(1+\f{q}{(q-1)^2}\f{\log s}{s}(1+o(1))\right).
\ee
Using  \eqref{27-4-1} and \eqref{28-4-1} and the fact that $\nu=\f{2}{q-1}$, we obtain
\be\lab{27-4-5}
v(r)=\displaystyle\left(\f{\al}{q-1}\right)^\f{\nu}{2}\left(\log\f{\al}{r}\right)^{-\f{\nu}{2}}\left[1+\f{q}{(q-1)^2}\f{\log(\al\log\f{\al}{r})}{\al\log\f{\al}{r}}(1+o(1))\right].
\ee
Therefore
$$u(r)=\displaystyle\left(\f{\al}{q-1}\right)^\f{\nu}{2} r^{-\nu}(-\log r)^{-\f{\nu}{2}}\left(1-
\f{\log\al}{\log r}\right)^{-\f{\nu}{2}}\left[1+\f{q}{(q-1)^2}\f{\log(\al\log\f{\al}{r})}{\al\log\f{\al}{r}}(1+o(1))\right].$$
Hence it is easy to see that
$$\lim_{|x|\to 0}\f{u(x)}{|x|^{-\nu}\big|\log |x|\big|^{-\f{\nu}{2}}}=\displaystyle\left(\f{\al\nu}{2}\right)^\f{\nu}{2}.$$

\end{proof}

\subsection{Gradient estimate}

In this subsection we establish gradient estimate of any solution of Eq. \eqref{a4} near origin. More precise we prove Theorem \ref{t:grad}. Towards this goal, first we need the following two lemmas.
\begin{lemma}\lab{l:grad}
Let $\rho_0, \rho$ be as in Theorem \ref{t:up-est} and Theorem \ref{t:up-est-2}, respectively, $u$ be a weak solution of Eq.  \eqref{a4} and $p, q$ be as in Theorem \ref{t:grad}. Then there exists $\mu^*=\mu^*(N,q)>0$ and  a constant $C$ depending on $N, p, q, \mu$ such that  $u$ satisfies
\begin{center}
$\displaystyle\fint_{B_{\f{|x|}{4}}(x)}|\na u(x)|^2dx\leq \left\{\begin{array}{lll}
C|x|^{-2(\nu+1)} \quad\text{if}\quad \mu\in[0,\mu^*),\\
C|x|^{-2(\f{q+1}{q-1})} \quad\text{if}\quad \mu\in [\mu^*,\bar\mu),
\end{array}
\right.$
\end{center}
for $0<|x|<\f{1}{2}\min\{\rho_0, \rho\}$.
\end{lemma}
\begin{proof}
Define $\tilde\rho_0:=\f{1}{2}\min\{\rho_0, \rho\}$.  Fix $x\in\Rn$ such that $0<|x|<\tilde\rho_0$. Let $B=B_{\f{|x|}{4}}(x)$ and $2B= B_{\f{|x|}{2}}(x)$. Choose $\eta\in C^{\infty}_0(2B)$ such that $0\leq\eta\leq 1$, $\eta\equiv 1$ on $B$ and $|\na \eta |<\f{8}{|x|}$. Define $\va:=\eta^2u$. Using this test function $\va$, we obtain from Eq. \eqref{a4}
$$\displaystyle\int_{2B}\na u\na\va dx=\int_{2B}\left(\f{\mu u^2\eta^2}{|x|^2}+u^{p+1}\eta^2-u^{q+1}\eta^2\right)dx.$$
Moreover, by a straight forward computation it follows
$$\displaystyle\int_{2B}\na u\na\va dx\geq \f{1}{2}\int_{B}|\na u|^2dx-C\int_{2B}u^2|\na\eta|^2dx.$$
Therefore we have
\be\lab{20-4-1}
\displaystyle\int_{B}|\na u|^2dx\leq C\int_{2B}\left(u^2|\na\eta|^2+\f{\mu u^2\eta^2}{|x|^2}+u^{p+1}\eta^2-u^{q+1}\eta^2\right)dx.
\ee
 Define
\be\lab{mu-star}
\mu^*=\displaystyle\left(\f{N-2}{2}\right)^2-\left(\f{N-2}{2}-\f{2}{q-1}\right)^2.
\ee
We observe that $ \mu<\mu^*\Longleftrightarrow q<\f{2+\nu}{\nu}$.\\

{\bf Case 1:} $q<\f{2+\nu}{\nu}$.\\

Applying Theorem \ref{t:up-est} in \eqref{20-4-1}, we obtain
\be\lab{17-1}
\displaystyle\int_{B}|\na u|^2dx\leq C\left(|x|^{-2\nu+N}+|x|^{-2\nu-2+N}+|x|^{-(p+1)\nu+N}+|x|^{-(q+1)\nu+N}\right),
\ee
for every $x$ satisfying $0<|x|<\tilde\rho_0$ . Therefore from \eqref{17-1}, we have
\be\lab{17-2}
\displaystyle\int_{B}|\na u|^2dx\leq C |x|^{-(2\nu+2)+N} \quad\text{if}\quad\mu\in(0,\mu^*).\ee

{\bf Case 2:} $q\geq\f{2+\nu}{\nu}$.\\

In this case we have  $\mu\geq\mu^*$. Applying Theorem \ref{t:up-est-2} in \eqref{20-4-1}, we obtain
\bea\lab{17-3}
\displaystyle\int_{B}|\na u|^2dx &\leq& C\left(|x|^{-\f{4}{q-1}+N}+|x|^{-2(\f{q+1}{q-1})+N}+|x|^{-2(\f{p+1}{q-1})+N}+|x|^{-2(\f{q+1}{q-1})+N}\right)\no\\
&\leq& C|x|^{-2(\f{q+1}{q-1})+N},
\eea
for $0<|x|<\tilde\rho$.

Combining \eqref{17-2} and \eqref{17-3}, the lemma follows.
\end{proof}
The next lemma is due to Xiang, see \cite[Proposition 2.1]{Xiang}.
\begin{lemma}\lab{l:Xiang}
Let $\Om$ be a domain in $\Rn$, $f\in L^{\infty}_{loc}(\Om)$ and $u\in H^1_{loc}(\Om)$ be a weak solution of the equation $$-\De u=f \quad\text{in}\quad\Om.$$ Then for any $B_{2R}(x_0)\subseteq\Om$, it holds
$$\sup_{B_{\f{R}{2}}(x_0)}|\na u|\leq C\displaystyle\left(\fint_{B_{R}(x_0)}|\na u(x)|^2dx\right)^\f{1}{2}+CR|f|_{L^{\infty}(B_R(x_0))}.$$
\end{lemma}

\vspace{3mm}

\begin{proof}[\bf Proof of Theorem \ref{t:grad}] Let $u$ be a weak solution of \eqref{a4} and $\tilde\rho$ be as in Lemma \ref{l:grad}. Then we can write $-\De u=f(x)$, where
$$f(x)=\mu\f{u}{|x|^2}+u^p-u^q.$$

{\bf Case 1:} $q<\f{2+\nu}{\nu}$.\\
In this case by Theorem \ref{t:up-est}, it follows $|f(x)|\leq C(|x|^{-\nu-2}+|x|^{-\nu p}+|x|^{-\nu q})$. Since $q<\f{2+\nu}{\nu}\Longleftrightarrow \mu<\mu^*$, we have
\begin{equation}\lab{17-4}
|f(x)|\leq C|x|^{-\nu-2} \quad\text{if}\quad \mu\in[0,\mu^*),
\end{equation}
for $0<|x|<\tilde\rho$.

{\bf Case 2:} $q\geq\f{2+\nu}{\nu}$.\\
In this case By Theorem \ref{t:up-est-2}, we obtain
\be\lab{17-4'}
|f(x)| \leq C(|x|^{-\f{2q}{q-1}}+|x|^{-\f{2p}{q-1}})\leq C|x|^{-\f{2q}{q-1}}, \quad \mu\in [\mu^*,\bar\mu)
\ee
for $0<|x|<\tilde\rho$.

 Consequently, in both Case 1 and Case 2,  $f\in L^{\infty}_{loc}(B_{\rho_0}(0)\setminus\{0\})$. As a result, for any $x\in B_{\tilde\rho}(0)\setminus\{0\}$, we apply Lemma \ref{l:Xiang}
on the domain $B_{\f{|x|}{2}}(x)$ to obtain that
$$\sup_{B_{\f{|x|}{8}}(x)}|\na u|\leq C\displaystyle\left(\fint_{B_{\f{|x|}{4}}(x)}|\na u(x)|^2dx\right)^\f{1}{2}+C|x||f|_{L^{\infty}(B_{\f{|x|}{4}}(x))}.$$
Combining  \eqref{17-4}, \eqref{17-4'} and Lemma \ref{l:grad}, we obtain from the above expression
\begin{center}
$\displaystyle\sup_{B_{\f{|x|}{8}}(x)}|\na u|\leq \left\{\begin{array}{lll}
C|x|^{-(\nu+1)} \quad\text{if}\quad \mu\in[0,\mu^*),\\
C|x|^{-(\f{q+1}{q-1})} \quad\text{if}\quad \mu\in [\mu^*,\bar\mu),
\end{array}
\right.$
\end{center}
for every $x$ satisfying $0<|x|<\tilde\rho$. This completes the theorem.
\end{proof}

\section{Holder continuity and Green function estimates}
\begin{lemma} \label{ghil}Let $R>0$ be given a small number. Then Green function defined in \eqref{green}  satisfies
\be\label{sanj1} \sup_{r/2<|x-y|< r } G(x, y) \leq  C\int_{r}^{R}\frac{ t^2}{w({B_{t}(x)})} \frac{dt}{t}\ee
where $w({B_{t}(x)})= \displaystyle\int_{|x-y|< t} |y|^{-2\nu}dy$  with $N-2\nu-2>0$ and $r\in (0, \f{R}{2})$ and $dist(x,\pa\Om)>R$, $dist(y,\pa\Om)>R$.
In fact, we have \be\label{sanj2} G(x, y) \leq  \frac{ C r^2}{w({B_{r}(x)})}.\ee
\end{lemma}

\begin{proof}  This is a modification of the theorem by Chanillo-Wheeden \cite[pg. 311]{CW}.
Note that the term in \eqref{sanj1} contribute when $t$ is close to $r.$ Define $f(y)=|y|^{-2\nu}$
 As  $f$ is a Muckenhoupt weight it satisfies the doubling property:
\be\label{doub} \int_{|x-y|< \frac{r}{2}} f(y)dy\leq \int_{|x-y|<r} f(y)dy\leq C \int_{|x-y|<\frac{r}{2}} f(y)dy.\ee
Using this fact we can cut the RHS  dia-dically;  we are left simply with the term near $r$ and the above expression reduces to
\be\no \sup_{r/2<|x-y|< r }  G(x, y) \leq  C  \sum_{k\geq 0}^{m}\int_{2^{k}r}^{2^{k+1}r} \frac{ t^2}{w({B_{t}(x)})} \frac{dt}{t}.\ee
where $2^{m+1} r=R.$\\
Let \be I= \sum_{k\geq 0} \frac{2^{2k}r^2}{\int_{{B_{2^{k}r}(x)}} f(y)dy}.\ee
It follows from Lemma \ref{l:17-4} (see Appendix B) that
  \be  \label{gxx2} \int_{{B_{2^{k}r}(x)}} f(y)dy\geq C 2^{k(N-2 \nu)} \int_{{B_{r}(x)}} f(y)dy\ee
where $C>0$ is independent of $x, k$ and $r.$
Therefore
\be \no I \leq \frac{C r^2}{\int_{{B_{r}(x)}} f(y)dy} \sum_{k\geq 0} 2^{k(2+ 2\nu -N)} \leq \frac{C r^2}{\int_{{B_{r}(x)}} f(y)dy}= \frac{C r^2}{w({B_{r}(x)})}.\ee
Hence the lemma follows.
\end{proof}

\begin{lemma} \label{p} The Green function satisfies the following estimate
\be\label{GXY}  G(x, y) \leq C \bigg(\frac{|x|^{2\nu}+ |x-y|^{2\nu}}{|x-y|^{N-2}}\bigg)\ee
for any $x\neq y$ and $x, y \in \Om$ for some $C>0$ depending on $\Om.$ Moreover, $G(x, .)$ is continuous whenever $x\neq y.$
\end{lemma}

\begin{proof} Since $\f{r}{2}<|x-y|<r$, we can write $r= O(|x-y|)$. Then from \eqref{sanj2},  we have
\be \label{sanj3} G(x, y) \leq  \frac{ C |x-y|^2}{w({B_{r}(x)})}\ee
 Now we estimate the denominator $D= w({B_{r}(x)})$ in the two cases. \\
 Case 1: $|x-y|\leq\f{1}{4}|x|$.\\
 In this case we note that $$D=\int_{|x-y|<r}|y|^{-2\nu} dy=\int_{|y|<r}|x-y|^{-2\nu}dy\geq \om_N\displaystyle\left(\f{1}{4}|x|\right)^{-2\nu}r^N\geq C|x|^{-2\nu}|x-y|^N,$$
 where $\om_N$ is volume of unit ball in $\Rn$. Therefore
 $G(x, y)\leq  C\f{|x|^{2\nu}}{|x-y|^{N-2}}$.\\
 Case 2: $|x-y|>\f{1}{4}|x|$.\\
  In this case we can write $|y|\leq|x-y|+|x|\leq 5|x-y|$. Therefore
  $$D=\int_{|x-y|<r}|y|^{-2\nu} dy\geq \om_N (5|x-y|)^{-2\nu}r^N=C|x-y|^{-2\nu+N}.$$
Thus  $G(x, y)\leq C \f{|x-y|^{2\nu}}{|x-y|^{N-2}}$.
 Combining case 1 and case 2, the lemma follows.
\end{proof}

\begin{lemma} \label{regur} Consider the problem
 \begin{equation}
  \label{420}
  \left\{
    \begin{aligned}
     -div (|x|^{-2\nu} \nabla v)&= |x|^{-(p+1)\nu } v^{p}  - |x|^{-(q+1)\nu } v^{q} &&\text{in } \Om;\\
v &= 0 &&\text{on } \pa \Om,
    \end{aligned}
  \right.
\end{equation}
where $ q<\frac{2+\nu}{\nu}.$
Then $v$ is  H\"older continuous at the origin with H\"older exponent $\theta= 2+2\nu-(q+1)\nu$.
\end{lemma}

\begin{proof}  In order to prove this result we use the information on the Green function.
We know that $v$ is bounded at the origin.  Also note that near the origin $ |x|^{-(q+1)\nu}  v^{q} $ is the dominating term. Then
 \be v(x)= \int_{\Om} G(x, z) [|z|^{-(p+1)\nu } v^{p}(z)- |z|^{-(q+1)\nu}  v^{q} (z)] dz. \ee
So we have
\be \no v(x)- v(y)= \int_{\Om}[ G(x, z)- G(y, z)][|z|^{-(p+1)\nu } v^{p}(z) - |z|^{-(q+1)\nu}  v^{q} (z)]  dz. \ee
By the self adjoint-ness of the Green function, for fixed $z$ we consider
\be \no G(x, z)- G(y, z) = G_{z}(x)- G_{z}(y).\ee
Consider the ball of radius $|x-y|= \rho$ centered at $x$. We take $x=0$. then using the fact that $v$ is bounded and $|y|=\rho$,  we obtain
\bea |v(y)- v(0)|&\leq & C (A+B), \no\eea where
\be A= \int_{|z|\leq \rho} |G_{z}(0)- G_{z}(y)| |z|^{-(q+1)\nu}dz\no\ee and
\be B=\int_{|z|\geq \rho} |G_{z}(0)- G_{z}(y)| |z|^{-(q+1)\nu}dz. \no  \ee

\underline{Case 1} If $z$ lies outside the ball of radius $|y|= \rho$.  Then we can think of $z$ as the pole  and we are far away. That is, we can consider a ball $\mathcal{B}$ of radius $\frac{|z|}{4}$ centered at $0$ such that double the ball does not meet $z$. Then we can use the Moser--Harnack inequality as in Chanillo--Wheeden \cite{CW1} to obtain
\be \no  |G_{z}(0)- G_{z}(y)|\leq C \bigg( \frac{|y|}{|z|}\bigg)^{\de} \bigg( \frac{|z|^2}{\int_{B_{\rho}(0)} |\xi|^{-2\nu}}\bigg)= C \bigg( \frac{|y|}{|z|}\bigg)^{\de}\bigg ( \frac{|z|^2}{|y|^{N-2\nu}}\bigg),\ee where $\de>0.$
This is a bound from the Harnack inequality as the Green's function is non-negative. Therefore,
\bea B &\leq  &  |y|^{\de-N+2\nu}  \int_{|z|\geq \rho} \frac{|z|^2}{|z|^{(q+1)\nu+ \de}} dz\no= O(|y|^{\theta}).\eea

\underline{Case 2} In this case,  $z$ lies in the ball of radius $|y|=\rho.$ Here we use the crude bound on the Green function from Lemma \ref{p}
\be \no  | G_{z}(0)- G_{z}(y)| \leq |G_{z}(0)|+ |G_{z}(y)| \leq  \frac{C}{|z|^{N-2\nu-2}} + |G_{z}(y)|.\ee
Hence   we have
\be \no C \int_{|z|\leq \rho}   \frac{|z|^{-(q+1) \nu}}{|z|^{N-2\nu-2}} dz = O(|y|^{2+2\nu-(q+1)\nu})= O( |y|^{\theta}).\ee
Now we estimate $G_{z}(y).$  For this we divide the domain into two parts.  Whenever $|z|\leq \frac{|y|}{2}$, we have $B_{|y|/4} (y)  \subset B_{|z-y|} (y) $ and from \eqref{sanj2}
\be\no  |G_{z}(y)|\leq C\bigg( \frac{|z-y|^2}{\int_{B_{|z-y|}(y)} |\xi|^{-2\nu}}\bigg).\ee
As a result we have
\be\no  |G_{z}(y)|= O(|y|^{2+ 2\nu -N}).\ee Hence we obtain
\be \no \int_{|z|\leq \frac{|y|}{2}}  |G_{z}(y)| |z|^{-(q+1)\nu}  |v|^{q} dz \leq c |y|^{\theta}.\ee
We are now left to check what happens in the region $ \frac{|y|}{2}\leq |z|\leq |y|=\rho.$
So in this region we have
\be \label{gcw1}\int_{ \frac{|y|}{2}\leq |z| \leq |y|  }  |G_{z}(y)| |z|^{-(q+1)\nu }  |v|^{q} dz \leq  C|y|^{-(q+1)\nu }\int_{ \frac{|y|}{2}\leq |z| \leq |y|  }  |G_{z}(y)|  dz .\ee
Now suppose that $|z-y|\leq \frac{|y|}{2},$ This implies if $\xi\in B_{|z-y|}(y)$, then we have $|\xi-y| \leq |z-y|\leq \frac{|y|}{2}.$ Consequently, $||\xi|-|y||\leq \frac{|y|}{2}$ which in fact implies that $\frac{|y|}{2}\leq |\xi|\leq \frac{3}{2}|y|$. Therefore
\be\no  \int_{B_{|z-y|}(y)} |\xi|^{-2\nu} d \xi=O(|y|^{-2\nu} |z-y|^N ).\ee
Using the above expression, we obtain
\be\no  |G_{z}(y)|\leq C\bigg( \frac{|z-y|^2}{\int_{B_{|z-y|}(y)} |\xi|^{-2\nu}}\bigg)=O(|y|^{2\nu} |z-y|^{2-N} ) .\ee
When $|z-y|\geq \frac{|y|}{2}$, then since $|z|\leq |y|$ we  have
\be\no  |G_{z}(y)|= O(|y|^{2+2\nu-N}) .\ee
Hence from \eqref{gcw1}, we obtain
\bea \no \int_{ \frac{|y|}{2}\leq |z| \leq |y|  }  |G_{z}(y)| |z|^{-(q+1)\nu }  |v|^{q} dz &\leq&  |y|^{-(q+1)\nu } \int_{ \{\frac{|y|}{2}\leq |z| \leq |y|\} \cap \{|z-y|\leq \frac{|y|}{2}\}  }  G_{z}(y)dz \no\\&+&  |y|^{-(q+1)\nu } \int_{ \{\frac{|y|}{2}\leq |z| \leq |y|\} \cap \{|z-y|\geq \frac{|y|}{2}\}  }  G_{z}(y) dz \no \\&\leq & C  |y|^{2\nu-(q+1)\nu } \int_{ \{\frac{|y|}{2}\leq |z| \leq |y|\}  }  \frac{1}{|z-y|^{N-2}} dz\no\\
&+& C |y|^{2+2\nu-(q+1)\nu}\no \\&=& O(|y|^{\theta})\no.\eea
\end{proof}

\begin{lemma}\label{wpi} (Weighted Pohozaev Identity) Let $\Om$ be a smooth bounded domain and $v\in C^{1}({\overline{\Om} \setminus \{0\}})\cap C^{0}(\Om)$ be a positive solution of
\begin{equation}
  \label{pd1}
  \left\{
    \begin{aligned}
-\nabla (|x|^{-2\nu} \nabla v) + \eps  |x|^{-(q+1)\nu}  v^q &= |x|^{-(p+1)\nu}  v^{p}\quad\text{in }\quad \Om,\\
v &\in H^1(\Om, |x|^{-2\nu})\cap L^{q+1}(\Om, |x|^{-(q+1)}),
    \end{aligned}
  \right.
\end{equation}
with $2^*-1\leq p<q$, $\nu\in(0, \f{N-2}{2})$. Then $v$ satisfies
\bea \label{poh1}\frac{1}{2}\int_{\pa \Om}&&   |x|^{-2\nu}  \langle x, n \rangle |\nabla v|^2 dS + \bigg(\frac{N-2-2\nu}{2}\bigg)\int_{ \pa \Om} |x|^{-2\nu}  v \frac{\pa v}{\pa n} dS \no\\&- &  \frac{\eps}{q+1}\int_{\pa \Om}   |x|^{-(q+1)\nu}  \langle x, n \rangle v^{q+1} dS + \bigg(\frac{N}{q+1}-\frac{N-2}{2}\bigg)  \eps \int_{\Om}|x|^{-(q+1 )\nu}  v^{q+1} dx \no\\&=& \bigg(\frac{N}{p+1}-\frac{N-2}{2}\bigg)\int_{\Om}  |x|^{-(p+1)\nu} v^{p+1}dx\no \\&-& \frac{1}{p+1}\int_{\pa \Om}  |x|^{-(p+1)\nu}  \langle   x, n \rangle  v^{p+1}dS.
\eea
Moreover, if $v=0$ on $\pa\Om$, then
\bea \label{poh2}\frac{1}{2} \int_{\pa \Om}  |x|^{-2\nu}\langle x, n \rangle |\nabla v|^2 dS &=& \bigg(\frac{N}{p+1}- \frac{N-2}{2}\bigg)\int_{\Om}  |x|^{-(p+1)\nu}v^{p+1} dx \no\\&+ & \eps  \bigg( \frac{N-2}{2}- \frac{N}{q+1}\bigg) \int_{\Om}  |x|^{-(q+1)\nu} v^{q+1} dx .
\eea
\end{lemma}

\begin{proof}
This follows from Proposition \ref{p:Poho}. Note that, $v$ is a solution of \eqref{pd1} implies $u(x)=|x|^{-\nu}v(x)$ is a solution of \eqref{eq:poho-1}.  Therefore substituting $u(x)=|x|^{-\nu}v(x)$ in \eqref{poh1'} we obtain,
\begin{eqnarray*}
&&\frac{1}{2}\int_{\pa \Om}   |x|^{-2\nu}  \langle x, n \rangle |\nabla v|^2 dS + \bigg(\frac{N-2-2\nu}{2}\bigg)\int_{ \pa \Om} |x|^{-2\nu}  v \frac{\pa v}{\pa n} dS \no\\
&+&\frac{\nu^2-\nu(N-2)+\mu}{2}\int_{\pa\Om}|x|^{-2\nu-2}v^2\langle x, n\rangle dS\no\\
&=&  \frac{\eps}{q+1}\int_{\pa \Om}   |x|^{-(q+1)\nu}  \langle x, n \rangle v^{q+1} dS - \eps\bigg(\frac{N}{q+1}-\frac{N-2}{2}\bigg) \int_{\Om}|x|^{-(q+1 )\nu}  v^{q+1} dx \no\\
&+& \bigg(\frac{N}{p+1}-\frac{N-2}{2}\bigg)\int_{\Om}  |x|^{-(p+1)\nu} v^{p+1}dx- \frac{1}{p+1}\int_{\pa \Om}  |x|^{-(p+1)\nu}  \langle   x, n \rangle  v^{p+1}dS.
\end{eqnarray*}
Since $\nu^2-\nu(N-2)+\mu=0$, the above expression reduces to \eqref{poh1}. Consequently $v=0$ on $\pa\Om$ implies \eqref{poh2}.
\end{proof}

\begin{lemma} \label{dub1} Let $\nu\in (0, \f{N-2}{2})$. Then the Green function $G(x,0)$ satisfies
\be\label{creg} \int_{\pa \Om}|x|^{-2\nu} \langle x, n\rangle |\nabla  G(x, 0)|^2 dS= (N-2-2\nu)|R(0)|.\ee
\end{lemma}

\begin{proof} We apply Pohozaev identity \eqref{poh1} to \eqref{green}  on $\Om \setminus B_{r}(0)$, for $r$ sufficiently small. Then we have
\bea\lab{8-4-3}   \int_{\pa \Om} |x|^{-2\nu} \langle x, n\rangle |\nabla  G(x, 0)|^2 dS &=& \int_{\pa B_{r}} |x|^{-2\nu} \langle x, n\rangle |\nabla  G(x, 0)|^2 dS\no\\ &+& (N-2-2\nu)\int_{ \pa B_{r}} |x|^{-2\nu}  G(x, 0) \frac{\pa  G(x, 0)}{\pa n} \no \\&=&  r  \int_{\pa B_{r}} |x|^{-2\nu} |\nabla  G(x, 0)|^2 dS\no\\
&+&(N-2-2\nu)\int_{ \pa B_{r}} |x|^{-2\nu}  G(x, 0) \frac{\pa  G(x, 0)}{\pa n}\eea
Moreover, from \eqref{eq:gr-H} and \eqref{eq:gr-F} we have,
\be\lab{8-4-1}
G(x, 0)= H(x,0)+\f{1}{(N-2\nu-2)\om_N|x-y|^{N-2\nu-2}}\ee and hence
\be\lab{8-4-2} \nabla G(x, 0)= -\frac{1}{\om_{N}}|x|^{(2\nu-N)}x+ \nabla H(x, 0).\ee
Substituting $G(x,0)$ and $\na G(x,0)$ in \eqref{8-4-3}, we take the limit $r\to 0$. After simplifying the terms, we obtain
\bea\lab{12-4-7}
\lim_{r\to 0}\text{RHS of \eqref{8-4-3}}&=&\lim_{r\to 0} r^{-2\nu-1}\int_{\pa B_r}|x\cdot \na H(x,0)|^2 dS\no\\
&-&\f{r^{1-N}}{\om_N}\int_{\pa B_r}\langle x\cdot \na H(x,0)\rangle dS \no\\
&-&\f{(N-2-2\nu)}{\om_N}r^{-N+1}\int_{\pa B_r} H(x,0)dS \no\\
&+&(N-2-2\nu)r^{-2\nu-1}\int_{\pa B_r}H(x,0) \langle x\cdot \na H(x,0)\rangle dS
\eea

 Note that, as $H(x, 0)$ is Holder continuous at origin \cite{CW}, it follows  $|x\cdot \nabla H(x, 0)|\to 0$ on $\pa B_r$ as $r\to 0$. Therefore a straight forward computation yields
$$\text{RHS of \eqref{8-4-3}}=-\f{(N-2-2\nu)}{\om_N r^{N-1}}\int_{\pa B_r}H(x,0)dS .$$
Using the mean value theorem,
\be R(0)= H(0,0)= \frac{1}{\om_{N}r^{N-1}} \int_{\pa B_{r}} H(x, 0) dS.\ee
Hence the lemma follows.
\end{proof}

\section{ Symmetry and decay properties of entire problem}
In this section using moving plane method, we give the proof of Theorem \ref{t:moving pl}.
\begin{proof}[\bf Proof of Theorem \ref{t:moving pl}]
It is enough to show that $u$ is symmetric with respect to each coordinate axis. For $\al>0$, we define $$\Omega_{\al}=\{x\in\Rn : x_1>\al \},$$ and for $x\in\Omega_{\al}$, let $x_{\al}$ denote the it's reflection to the hyperplane $x_1=\al$, that is $x_{\al}=(2\al-x_1, x_2, \cdots, x_n)$. Set
$$u_{\al}(x):= u(x_{\al}), \quad x\in\Omega_{\al} \quad\text{and}\quad w_{\al}=u_{\al}-u .$$ We note that $w_{\al}$ is smooth away from the point $(2\al, 0, \cdots, 0)$ and $w_{\al}=0$ on $\pa\Omega_{\al}$. It is easy to check that $w_{\al}\in D^{1,2}(\Omega_{\al})$.

{\bf Claim 1:} $w_{\al}\geq 0$ in $\Omega_{\al}$, if $\al>0$ is large enough.

To see the claim, we note that $|x_{\al}|<|x|$ if $\al>0$. By a straight forward computation it follows that $w_{\al}$ satisfies the following equation
\begin{equation}\label{eq:rad1}
-\Delta w_{\al}-\mu\frac{w_{\al}}{|x|^2}\geq A_1(x)w_{\al}-A_2(x)w_{\al} \quad\text{in}\quad\Omega_{\al},
\end{equation}
where $$0\leq A_1(x):=\la\frac{u_{\al}^p-u^p}{u_{\al}-u}\leq \la p \big[\text{max}\{u_{\al}(x), u(x)\}\big]^{p-1}$$ and
$$0\leq A_2(x):=\frac{u_{\al}^q-u^q}{u_{\al}-u}\leq q \big[\text{max}\{u_{\al}(x), u(x)\}\big]^{q-1}.$$
Multiplying \eqref{eq:rad1} by $w_{\al}^{-}$ and integrating by parts over $\Om_{\al}$, we obtain
\begin{eqnarray}\label{eq:rad2}
\int_{\Omega_{\al}}|\na w_{\al}^-|^2 dx-\mu\frac{|w_{\al}^-|^2}{|x|^2}dx &\leq& \int_{\Omega_{\al}}(A_1(x)-A_2(x))|w_{\al}^{-}|^2 dx\no\\
&\leq& \int_{\Omega_{\al}}A_1(x)|w_{\al}^{-}|^2 dx\no\\
&\leq& \displaystyle\left(\int_{\Omega_{\al}}|w_{\al}^{-}|^{2^*}dx\right)^\frac{N-2}{N}\left(\int_{\Omega_{\al}\cap\{w_{\al}<0\}}A_1^\frac{N}{2}dx\right)^\frac{2}{N}.
\end{eqnarray}
As $\mu<(\frac{N-2}{2})^2$,  it is not difficult to check that $\displaystyle\left(\int_{\Omega_{\al}}|\na w_{\al}^{-}|^2-\mu\frac{|w_{\al}^{-}|^2}{|x|^2}\right)^\frac{1}{2}$ is an equivalent norm to $D^{1,2}(\Rn)$. Therefore there exists a positive constant $C_1$ such that $C_1\displaystyle\int_{\Omega_{\al}}|\na w_{\al}^{-}|^2\leq\displaystyle \int_{\Omega_{\al}}|\na w_{\al}^{-}|^2-\mu\frac{|w_{\al}^{-}|^2}{|x|^2}$. Applying this estimate along with Sobolev inequality, we have from \eqref{eq:rad2}:
\begin{equation}\label{eq:rad3}
C_1 \mathcal{S} \displaystyle\left(\int_{\Omega_{\al}}|w_{\al}^{-}|^{2^*}dx\right)^\frac{2}{2^*}\leq \left(\int_{\Omega_{\al}}|w_{\al}^{-}|^{2^*}dx\right)^\frac{N-2}{N}\left(\int_{\Omega_{\al}\cap\{w_{\al}<0\}}A_1^\frac{N}{2}dx\right)^\frac{2}{N},
\end{equation}
where $\mathcal{S}$ is the Sobolev constant. On the other hand, $u_{\al}<u$ on $\{w_{\al}<0\}$ implies
$$\int_{\Omega_{\al}\cap\{w_{\al}<0\}}A_1^\frac{N}{2}dx\leq C \int_{\Omega_{\al}\cap\{w_{\al}<0\}} u^{(p-1)\frac{N}{2}}.$$
We know that $u\in L^{2^*}(\Rn)\cap L^{q+1}(\Rn)$. As by the given assumption $$q>(p-1)\frac{N}{2}-1 \quad\text{and}\quad p\geq 2^*-1,$$ using interpolation theory we can show that $u\in L^{(p-1)\frac{N}{2}}$. Consequently $$\int_{\Omega_{\al}\cap\{w_{\al}<0\}} u^{(p-1)\frac{N}{2}}\to 0\quad\text{if}\quad \al \quad\text{is large enough}.$$ Hence from \eqref{eq:rad3} we conclude $w_{\al}^{-}=0$ in $\Omega_{\al}$ if $\al$ is large enough. This proves the claim.

Let $$\al_0=\text{inf}\{\al>0: u_{\al'}\geq u \ \text{in}\ \Omega_{\al'} \ \forall \ \al'>\al\}.$$
{\bf Claim 2:} $\al_0=0$.

We will prove this claim by method of contradiction. Let $\al_0>0$.  Define $w_{\al_0}=u_{\al_0}-u$. Then $w_{\al_0}\geq 0$ in $\Omega_{\al_0}$ and \\
$-\Delta w_{\al_0}+A_2(x)w_{\al_0}=\mu\frac{w_{\al_0}}{|x|^2}+A_1(x)w_{\al_0}\geq 0$ in $\Omega_{\al_0}$ and away from the point $(2\al_0, 0, \cdots, 0)$. As $A_2\geq 0$, by maximum principle we have $w_{\al_0}>0$ in this region.

Let $\epsilon>0$. We choose $R>0$ and $\delta_0>0$ such that
\begin{equation}\label{eq:rad4}
\int_{|x|>R}u^{(p-1)\frac{N}{2}}dx<\frac{\epsilon}{2}
\end{equation}
and  \begin{equation}\label{eq:rad5}
\int_{\al_0-\delta_0<x_1<\al_0+\delta_0}u^{(p-1)\frac{N}{2}}dx+\int_{2\al_0-\delta_0<x_1<2\al_0+\delta_0}u^{(p-1)\frac{N}{2}}dx<\frac{\epsilon}{2}.
\end{equation}
Define $$K:=\{x\in\Omega: \al_0+\delta_0\leq x_1\leq 2\al_0-\delta_0 \quad{or}\ x_1\geq 2\al_0+\delta_0\}\cap \{|x|\leq R\}.$$
Then $K$ is a compact set and $w_{\al_0}>0$ in $K$. Choose $\delta_1\in (0,\delta_0)$ such that $w_{\al_0-\delta}>0$ in $K \quad\forall\ \delta\in(0,\delta_1)$. Define $\al_1:=\al_0-\delta$. Next we will show that $u_{\al_1}\geq u$ in $\Omega_{\al_1}$ and this will contradict the definition of $\al_0$. Towards this goal, we define $w_{\al_1}:=u_{\al_1}-u$. We proceed as in the case of \eqref{eq:rad3} to get
$$C_1\mathcal{S} \displaystyle\left(\int_{\Omega_{\al_1}}|w_{\al_1}^{-}|^{2^*}dx\right)^\frac{2}{2^*}\leq \left(\int_{\Omega_{\al_1}}|w_{\al_1}^{-}|^{2^*}dx\right)^\frac{2}{2^*}\left(\int_{\Omega_{\al_1}\cap\{w_{\al_1}<0\}}A_1^\frac{N}{2}dx\right)^\frac{2}{N}.$$
By the choice of $\al_1$, we have $w_{\al_1}>0$ in $K$ and thus by \eqref{eq:rad4}  and \eqref{eq:rad5}, we conclude that
$$\int_{\Omega_{\al_1}\cap\{w_{\al_1}<0\}}A_1^\frac{N}{2}dx\leq\la p \int_{\Omega_{\al_1}\cap\{w_{\al_1}<0\}}u^{(p-1)\frac{N}{2}}dx<\epsilon.$$ As $\epsilon>0$ is arbitrarily chosen, we can conclude that $w_{\al_1}^{-}=0$, which contradicts the definition of $\al_0$. Hence the claim follows.

\noi Consequently we have $$u(-x_1,x_2,\cdots, x_n)\geq u(x_1,x_2,\cdots, x_n)  \quad\forall\ x_1>0.$$ Now repeating the same arguments for $\tilde{u}(x)=u(-x_1,x_2,\cdots, x_n)$, we can prove that $$u(-x_1,x_2,\cdots, x_n)\leq u(x_1,x_2,\cdots, x_n)  \quad\forall\ x_1>0.$$ Hence $$u(-x_1,x_2,\cdots, x_n)= u(x_1,x_2,\cdots, x_n)  \quad\forall\ x_1>0.$$ As a result symmetry follows since we can do the moving plane argument in any direction instead of $x_1$ direction.
\end{proof}

\begin{remark}
Doing some simple modifications to the proof of Theorem \ref{t:moving pl}, it can be shown that $u$ is a radially symmetric solution of \eqref{eq:a3'}, when $\Om=B_R(0)$, for any $R>0$.  Therefore $v$ is a radially symmetric solution of \eqref{eq:a3}, when $\Om=B_R(0)$, for any $R>0$.
\end{remark}

\iffalse
We can rewrite $\mathcal{K}$ in Theorem \ref{1.2} as
\be\label{a2''} \mathcal{K}=\inf_N \hat{F}(u, \Rn)\ee

where
\begin{equation}\label{F(u)}
\hat{F}(u, \Omega)=\frac{1}{2}\displaystyle\int_{\Om} |x|^{-2\nu}|\nabla u|^2 dx+\frac{1}{q+1}\displaystyle\int_{\Om} |x|^{-(q+1)\nu}u^{q+1} dx,
\end{equation}
\begin{equation}\label{N}
N= \bigg\{u\in D^{1, 2}(\R^N, |x|^{-2\nu}) \cap L^{q+1}(\R^N, |x|^{-(q+1)\nu}): \ \displaystyle\int_{\Rn} |x|^{-(p+1)\nu}u^{p+1}=1\bigg\}
\end{equation}
and $q>p> 2^*-1.$  Then
\fi

 From Theorem \ref{1.2}, we know $\mathcal{K}$ is achieved by a radial function $v\in D^{1, 2}(\R^N, |x|^{-2\nu}) \cap L^{q+1}(\R^N, |x|^{-(q+1)\nu})$.   Furthermore, there exists a constant $\lambda>0$ such that $v$ satisfies the following problem:
\begin{equation}
  \label{div1}
\left\{\begin{aligned}
      -div (|x|^{-2\nu}\na v)&=\lambda |x|^{-(p+1)\nu} v^{p} -|x|^{-(q+1)\nu} v^q \quad\text{in }\quad \R^N,\\ v&>0   \quad\text{in }\quad \R^N,\\
      v &\in D^{1,2}(\Rn, |x|^{-2\nu})\cap L^{q+1}(\Rn, |x|^{-(q+1)\nu}).
         \end{aligned}
  \right.
  \end{equation}

\begin{lemma}\label{rem1}
Define $\al=N-2-2\nu$. Let $2^*-1< p$ and $q> (p-1)\frac{N}{2}-1$, $0<\nu<\f{N-2}{2}$.  Suppose $v$ is a solution of \eqref{div1}. Then $v(x)=v(|x|)=v(r)$, where $|x|=r$. Moreover, $$\displaystyle \int_{0}^{\infty} (v(r))^{2^*} r^{N-1-2^*\nu} dr < + \infty.$$
Furthermore, if  $v(r)\leq Cr^{-\al}$ for  $r>>1$ then $v \sim r^{-\al}$ as $r \to \infty$.
 \end{lemma}

\begin{proof} If $v$ is any solution of \eqref{div1}, then $u=|x|^{-\nu}v$ is a solution of \eqref{a3}. By Theorem \ref{t:moving pl},  $u$ is radially symmetric. Therefore  $v(x)=v(|x|)=v(r)$  and $v$ satisfies the following ode
\begin{equation}
  \label{eq:a4}
  \left\{
    \begin{aligned}  v_{rr}&+\frac{(N-2\nu-1)}{r} v_{ r}+\lambda r^{-(p-1)\nu} v^{p}-   r^{-(q-1)\nu} v^{q} =0 &&\text{ in  }  (0, \infty),\\
 v(r) &> 0 &&\text{ in  }  (0, \infty), \\
& \int_{0}^{\infty} (v'(r))^2 r^{N-1-2\nu} dr < +\infty .
\end{aligned}
  \right.
\end{equation}
Applying Sobolev embedding theorem it follows
$\displaystyle \int_{0}^{\infty} (v(r))^{2^*} r^{N-1-2^*\nu} dr < + \infty$. To prove the last assertion, we first show that
$v(r)\geq C r^{-\al}$  for some $C>0$ and $r\gg1.$  \\
To see this, let $w=|x|^{-\al}$. Then $div(|x|^{-2\nu}\na w)=0$ in $\Rn\setminus B_R(0)$. Hence by standard method using comparison principle it follows that
$$v\geq Cw=C|x|^{-\al}\quad \text{in}\quad \Rn\setminus B_R(0).$$
Hence there exist $C_{1}, C_2>0$ such that
\be \label{ulb} C_1 \leq v(r)r^{\al}  \leq C_2  \text{ for }  r>>1. \ee
But from \eqref{eq:a4} we have
$$- (r^{N-1-2\nu} v_{r})_{r}= v^p r^{N-1-(p+1)\nu} (1+ o(1))   \quad \text{for}\quad r>>1. $$
Since $v$ satisfies \eqref{ulb} and $p>2^*-1$, the RHS of the above expression is integrable in $(s, \infty)$ and positive. This implies that
$$\lim_{r \to \infty } v_{r}  r ^{(N-1-2\nu)} = -c. $$ for some  $c>0.$ This in fact implies that $v_{r}\sim - r^{-(N-1-2\nu)}.$  Integrating this expression from $(s,\infty)$ we obtain, $$\lim_{r \to \infty} r^{N-2-2\nu} v = a \in (0, + \infty).$$
\end{proof}

\section{ Proof of Theorem \ref{main1}}
\subsection{Auxiliary results}
Define \be \label{hatF}{\hat F}(w)=  \displaystyle\frac{1}{2}\int |x|^{-2\nu}  |\nabla w|^2 dx+  \frac{1}{q+1}\int |x|^{-(q+1)\nu}  w^{q+1} dx ,\ee
where $\nu\in(0,\f{N-2}{2})$, $q>p\geq 2^*-1$.
For $\rho>0$, set
$$N_{\rho}= \bigg\{ w\in H^{1}_{0}(\rho\Om, |x|^{-2\nu})\cap L^{q+1}(\rho\Om, |x|^{-(q+1)\nu}):
 \  \int_{\rho\Om} |x|^{-(p+1)\nu} w^{p+1}dx=1 \bigg\}.$$
Define $$S_{\rho}:=\inf_{w\in N_{\rho}} \hat{F}(w).$$

\begin{theorem}\label{A.1}
Let $p=2^{*}-1$. Then  $S_{\rho}\to \frac{\mathcal{S}}{2}$ as $\rho\to\infty$, where  $\mathcal{S}$ is as defined in \eqref{go}.
\end{theorem}
\begin{proof}
{\bf Step 1}: $\lim_{\rho\to\infty}S_{\rho}\leq \frac{\mathcal{S}}{2}$\\
To see this,  let $U(x)$ be as in \eqref{ent-U}. We know from \cite{CZ} that, $U$ is an extremal of $\mathcal{S}$, with $\displaystyle\int_{\Rn}|x|^{2^*\nu}U^{2^*}(x)dx=1$ and $U$ is a ground state  solution of \eqref{ent} .  It is easy to check that $\mu^{-\frac{\al}{2}}U(\frac{x}{\mu})$ is also a solution of \eqref{ent},  for any $\mu>0$, where $\al=N-2-2\nu$.

 Set $\rho :=\mu^2$. Define $$U_{\rho}(x):=\mu^{-\frac{\al}{2}}U(\frac{x}{\mu}) \quad\text{and}\quad
\phi_{\rho}(x)=\phi(\frac{x}{\rho}),$$ where $\phi\in C^{\infty}_0(\Rn)$ such that supp $\phi\in\Om$, $\phi=1$ in $\frac{\Om}{2}$, $0\leq\phi\leq 1$ and $|\na\phi|\leq\frac{2}{d}$ and $d=\text{diam}(\Om)$. We set $$v_{\rho}(x):=U_{\rho}(x)\phi_{\rho}(x) \quad\text{and}\quad
\hat{v}_{\rho}:=\frac{v_{\rho}}{||x|^{-\nu}v_{\rho}|_{L^{2^*}(\rho\Om)}}.$$ Then $\hat{v}_{\rho}\in N_{\rho}$.
\begin{eqnarray}\label{ap-1}
\displaystyle\lim_{\rho\to\infty}\int_{\rho\Om}|x|^{-2^*\nu}v_{\rho}^{2^*}dx &=&\lim_{\rho\to\infty}\mu^{-\frac{\al}{2}2^*}\int_{\Rn}|x|^{-2^*\nu}U^{2^*}(\frac{x}{\mu})\phi^{2^*}(\frac{x}{\rho})dx\no\\
&=&\lim_{\rho\to\infty}\rho^{-\frac{\al}{2}2^*+N-2^*\nu}\int_{\Rn}|x|^{-2^*\nu}U^{2^*}(x)\phi^{2^*}(\frac{x}{\sqrt{\rho}})dx\no\\
&=&\lim_{\rho\to\infty}\int_{\Rn}|x|^{-2^*\nu}U^{2^*}(x)\phi^{2^*}(\frac{x}{\sqrt{\rho}})dx\no\\
&=& \int_{\Rn}|x|^{-2^*\nu}U^{2^*}(x)dx=1.
\end{eqnarray}
Similarly we see that
\begin{equation*}
\displaystyle\int_{\rho\Om}|x|^{-(q+1)\nu}\hat{v}_{\rho}^{q+1}dx =
\frac{\rho^{-\frac{\al}{2}(q+1)+N-(q+1)\nu}}{||x|^{-\nu}v_{\rho}|_{L^{2^*}(\rho\Om)}^{q+1}}\int_{\Rn}|x|^{-(q+1)\nu}U^{q+1}(x)\phi^{q+1}(\frac{x}{\sqrt{\rho}})dx.
\end{equation*}
As before $$\displaystyle\lim_{\rho\to\infty}\int_{\Rn}|x|^{-(q+1)\nu}U^{q+1}(x)\phi^{q+1}(\frac{x}{\sqrt{\rho}})dx= \int_{\Rn}|x|^{-(q+1)\nu}U^{q+1}(x)dx.$$  Moreover  $q>2^*-1$ implies $-\frac{\al}{2}(q+1)+N-(q+1)\nu<0$.  Hence by \eqref{ap-1}, we have
\begin{equation}\label{ap-2}
\displaystyle\lim_{\rho\to\infty}\int_{\rho\Om}|x|^{-(q+1)\nu}\hat{v}_{\rho}^{q+1}dx=0.
\end{equation}
\begin{equation}\label{ap-3}
\displaystyle \int_{\rho\Om}|x|^{-2\nu}|\na\hat{v}_{\rho}|^2dx =\frac{\displaystyle\int_{\rho\Om}|x|^{-2\nu}|\na v_{\rho}|^2dx}{||x|^{-\nu}v_{\rho}|^2_{L^{2^*}(\rho\Om)}}; \quad\text{and}\quad  \int_{\rho\Om}|x|^{-2\nu}|\na v_{\rho}|^2dx= I_{\rho}^1+I_{\rho}^2+I_{\rho}^3,
\end{equation}
where $$\displaystyle I_{\rho}^1=\mu^{-(\al+2)}\int_{\Rn}|x|^{-2\nu}|\na U(\frac{x}{\mu})|^2\phi^2(\frac{x}{\rho});$$
$$\displaystyle I_{\rho}^2=\mu^{-(\al+4)}\int_{\rho\Om\setminus\rho\frac{\Om}{2}}|x|^{-2\nu}U^2(\frac{x}{\mu})|\na\phi(\frac{x}{\rho})|^2dx;$$
$$I_{\rho}^3=2\mu^{-(\al+3)}\displaystyle\int_{\rho\Om\setminus\rho\frac{\Om}{2}}|x|^{-2\nu}U(\frac{x}{\mu})\phi(\frac{x}{\rho})\na U(\frac{x}{\mu})\na\phi(\frac{x}{\rho})dx.$$
By straight forward computation we see that
\begin{equation}\label{ap-4}
\lim_{\rho\to\infty}I_{\rho}^1=\lim_{\rho\to\infty}\int_{\Rn}|x|^{-2\nu}|\na U(x)|^2\phi^2(\frac{x}{\sqrt{\rho}})dx=\int_{\Rn}|x|^{-2\nu}|\na U(x)|^2 dx=\mathcal{S}.
\end{equation}
\begin{eqnarray}\label{ap-5}
\lim_{\rho\to\infty}I_{\rho}^2 &\leq& \lim_{\rho\to\infty} \frac{4}{d^2}\mu^{-(\al+4)}\int_{\rho\Om\setminus\rho\frac{\Om}{2}}|x|^{-2\nu}U^2(\frac{x}{\mu})dx\no\\
&=&\lim_{\rho\to\infty}\frac{4}{d^2}\mu^{-(\al+4)-2\nu+N}\int_{\sqrt{\rho}\Om\setminus\sqrt{\rho}\frac{\Om}{2}}|x|^{-2\nu}U^2(x)dx\no\\
&\leq&\lim_{\rho\to\infty}\frac{4}{d^2}\mu^{-(\al+4)-2\nu+N}\bigg(\int_{\Rn}|x|^{-2*\nu}U^{2^*}(x)dx\bigg)^\frac{2}{2^*}|\sqrt{\rho}(\Om\setminus\frac{\Om}{2})|^{1-\frac{2}{2^*}}\no\\
&\leq&\lim_{\mu\to\infty} C\mu^{-(\al+4)-2\nu+N+(1-\frac{2}{2^*})}=\lim_{\mu\to\infty}C\mu^{-1-\frac{2}{2^*}}=0.
\end{eqnarray}
Similarly,
\begin{eqnarray*}
\lim_{\rho\to\infty}I_{\rho}^3 &\leq\lim_{\rho\to\infty}\mu^{-(\al+3)}\bigg(\int_{\Rn}|x|^{-2\nu}|\na U(\frac{x}{\mu})|^2\phi^2(\frac{x}{\rho})dx\bigg)^\frac{1}{2}\times\\
&\bigg(\int_{\rho\Om\setminus\rho\frac{\Om}{2}}|x|^{-2\nu}|\na\phi(\frac{x}{\rho})|^2\phi^2(\frac{x}{\mu})dx\bigg)^\frac{1}{2}.
\end{eqnarray*}
Consequently,
\begin{equation}\label{ap-6}
\lim_{\rho\to\infty}I_{\rho}^3\leq \lim_{\mu\to\infty}  C\mu^{-\frac{2}{2^*}}=0.
\end{equation}
Combining \eqref{ap-2}, \eqref{ap-4}, \eqref{ap-5}, \eqref{ap-6} and \eqref{ap-1} we obtain
$$S_{\rho}\leq F(\hat{v_{\rho}}) \quad\text{and}\quad  F(\hat{v_{\rho}})\to\frac{\mathcal{S}}{2} \quad\text{as}\quad \rho\to\infty.$$
Hence \be\label{ap-7} \lim_{\rho\to\infty}S_{\rho}\leq \frac{\mathcal{S}}{2}.\ee

\vspace{2mm} {\bf Step 2}: $\frac{\mathcal{S}}{2}\leq \lim_{\rho\to\infty}S_{\rho}$.\\
This is standard to prove. Therefore we just give here a sketch of the proof.
Let $\eps>0$.  Then there exists $u_{\rho,\eps}\in N_{\rho}$ such that
\be\label{ap-8}\hat{F}(u_{\rho,\eps})< S_{\rho}+\eps.\ee Extend $u_{\rho,\eps}$ by $0$ outside $\rho\Om$ and we denote it by $u_{\rho,\eps}$ too. Let $\eta(x)=C\exp(\frac{1}{|x|^2-1})$ if $|x|<1$ and $0$ otherwise. Set $\eta_{\de}(x)=\de^{-N}\eta(\frac{x}{\de})$.

Define $u_{\rho,\eps}^{\de}:=u_{\rho,\eps}*\eta_{\de}$ and $v_{\rho,\eps}^{\de}=\frac{u_{\rho,\eps}^{\de}}{|u_{\rho,\eps}^{\de}|_{L^{2^*}(\Rn)}}.$ Thus $v_{\rho,\eps}^{\de}\in C^{\infty}_0(\Rn)\cap N$, where
$$N :=\bigg\{ w\in D^{1,2}(\Rn, |x|^{-2\nu}dx): w\in L^{q+1}(\Rn, |x|^{-(q+1)\nu}dx),
 \  \int_{\Rn} |x|^{-2^*\nu} w^{2^*}dx=1 \bigg\}.$$
Moreover,
$$v_{\rho,\eps}^{\de}\to u_{\rho,\eps} \quad\text{in}\quad D^{1,2}(\Rn, |x|^{-2\nu}dx)\cap L^{q+1}(\Rn,|x|^{-(q+1)}dx) \quad\text{as}\quad\de\to 0.$$
Hence $$\frac{\mathcal{S}}{2}\leq\hat{F}(v_{\rho,\eps}^{\de})\to \hat{F}(u_{\rho,\eps}) \quad\text{as}\quad\de\to 0.$$ Combining this with \eqref{ap-8}, we conclude $\frac{\mathcal{S}}{2}<S_{\rho}+\eps$. As $\eps>0$ is arbitrary, this proves Step 2.

Combining Step 1 and Step 2, theorem follows.
\end{proof}

\begin{theorem}\label{A.2}
Let $p>2^{*}-1$. Then $S_{\rho}\to \mathcal{K}$ as $\rho\to\infty$, where $\mathcal{K}$ is as defined in \eqref{a22}.
\end{theorem}
\begin{proof}
Let $w\in D^{1,2}(\Rn, |x|^{-2\nu}dx)\cap L^{q+1}(\Rn, |x|^{-(q+1)\nu}dx)$
 be a minimizer of $\mathcal{K}$ (which exists by Theorem \ref{1.2}) with $\displaystyle\int_{\Rn}|x|^{-(p+1)\nu}w^{p+1}dx=1$.  Define $\phi_{\rho}$ as in Step 1 of the proof of Theorem \ref{A.1}. Set $w_{\rho}=w\phi_{\rho}$ and $\hat{w_{\rho}}=\frac{w_{\rho}}{||x|^{-\nu}w_{\rho}|_{L^{p+1}(\Rn)}}$. Then $\hat{w_{\rho}}\in N_{\rho}$ and consequently $S_{\rho}\leq F(\hat{w_{\rho}})$. Proceeding the same way as in  Step 1 of Theorem \ref{A.1}, we  obtain  $F(\hat{w_{\rho}})\to \mathcal{K}$ as $\rho\to\infty$. Hence $\lim_{\rho\to\infty}S_{\rho}\leq \mathcal{K}$. To get the other sided inequality we use the same idea as in Step 2 of Theorem \ref{A.2}. This completes the proof.
\end{proof}

\subsection{Asymptotic Behavior} For $v\in H^1_0(\Om, |x|^{-2\nu})\cap L^{q+1}(\Om, |x|^{-(q+1)\nu})$, we recall the definition of the functional $F(.,\Om)$ from \eqref{a22} for $p>2^*-1$ and $S(.,\Om)$ from \eqref{sep-17-1} for $p=2^*-1$:
\be F(v, \Om)= \frac{1}{2}\frac{\displaystyle\int_{\Om} |x|^{-2\nu}|\nabla v|^2 dx}{\displaystyle\int_{\Om}|x|^{-(p+1)\nu} v^{p+1} dx}+\frac{1}{q+1}\frac{\displaystyle\int_{\Om}|x|^{-(q+1)\nu}v^{q+1} dx}{\displaystyle\left(\int_{\Om}|x|^{-(p+1)\nu} v^{p+1} dx\right)^l},\no \ee
where
\be
l=\frac{2(q+1)-N(p-1)}{2(p+1)-N(p-1)}, \quad q>p> 2^*-1.\no\ee
 $$S(v)=\frac{\displaystyle\int_{\Om} |x|^{-2\nu}|\nabla v|^2 dx}{\displaystyle\left(\int_{\Om}|x|^{-(p+1)\nu} v^{p+1} dx\right)^\f{2}{p+1}}, \quad p=2^*-1.$$
Using the transform
\begin{equation}\label{v-w_eps}
v(x)= \eps^{-\frac{2+2\nu-(p+1)\nu}{2(q-p)}} w(\eps^{-\frac{p-1}{2(q-p)}}x),
\end{equation} Eq. \eqref{eq:a3} reduces to
\begin{equation}
  \label{eq:a3a}
\left\{\begin{aligned}
      -div(|x|^{-2\nu} \na w) &=|x|^{-(p+1)\nu} w^p - |x|^{-(q+1)\nu} w^q  &&\text{in } \Om_{\eps},\\  w &>0 && \text{ in } \Om_{\eps}, \\
      w(x) & =0 &&\text{ on } \pa \Om_{\eps},
          \end{aligned}
  \right.
\end{equation}
where $\Om_{\eps}= \frac{\Om}{\eps^{\frac{p-1}{2(q-p)}}}.$  Clearly $\Om_{\eps} \mapsto \R^N$ as $\eps \to 0.$

\begin{proposition}\label{p:main1}
Let $2^*-1\leq p<q$ and  $\nu\in (0,\f{N-2}{2})$. Then there exists $\eps_{0}>0$ such that  for  all $\eps\in (0, \eps_{0})$,  the problem
\begin{equation}
  \label{eq:eps_la}
\left\{\begin{aligned}
      -div(|x|^{-2\nu} \na v)&=\la_{\eps}|x|^{-(p+1)\nu} v^p -\eps |x|^{-(q+1)\nu} v^q &&\text{in } \Om,\\ v&>0 && \text{ in } \Om, \\
      v(x) & =0 &&\text{ on } \pa \Om,
          \end{aligned}
  \right.
\end{equation}
admits a solution $v_{\eps}$, with the property that
$$A<\la_{\eps}<B,$$ for some constants $A,B>0$, independent of $n$. In addition
\begin{itemize}
\item[(i)] if $p>2^*-1$, then $F(v_\eps)\to \mathcal{K}$ and $\displaystyle\int_{\Om} |x|^{-(p+1)\nu} v_{\eps}^{p+1}dx\to 0$ as $\eps\to 0$;
\item[(ii)]if $p=2^*-1$, then $S(v_{\eps})\to \mathcal{S}$ as $\eps\to 0$ and  $\displaystyle\int_{\Om} |x|^{-(p+1)\nu} v_{\eps}^{p+1}dx=1$,
\end{itemize}
where $\mathcal{K}$ and $\mathcal{S}$ are defined as in \eqref{14-K} and \eqref{go} respectively.
\end{proposition}

\begin{proof} Let $\Om_{\eps}=\frac{\Om}{\eps^\frac{p-1}{2(q-p)}}$.
We are going to work on the manifold $$N_{\eps}= \bigg\{ w\in H^{1}_{0}(\Om_{\eps}, |x|^{-2\nu}) \cap L^{q+1}(\Om_{\eps}, |x|^{-(q+1)\nu}): \int_{\Om_{\eps} } |x|^{-(p+1)\nu} w^{p+1}=1 \bigg\}.$$
Then $F $ on $N_{\eps}$ reduces to $\hat{F}$ (defined as in Subsection 6.1)
$$F (w)=  \frac{1}{2}\int_{\Om_{\eps}} |x|^{-2\nu}  |\nabla w|^2 dx+  \frac{1}{q+1}\int_{\Om_{\eps}}|x|^{-(q+1)\nu}  w^{q+1} dx =\hat{F}(w) .$$
For every $p\geq 2^*-1$, let
\be S_{\eps}= \inf_{N_{\eps}} \hat{F}(w)=\inf_{N_{\eps}} F(w).\ee
Let $\{w_{n,\eps}\} $ be a minimising sequence in $N_{\eps}$ such that
$$\hat{F}(w_{n,\eps}) \to S_{\eps} \text{ with } \int_{\Om_{\eps}}|x|^{-(p+1)\nu} w_{n,\eps}^{p+1} dx =1.$$
Thus $\{w_{n, \eps}\}$ is bounded in $H^{1}_{0}(\Om_{\eps}, |x|^{-2\nu}) \cap L^{q+1}(\Om_{\eps}, |x|^{-(q+1)\nu}).$ Hence $w_{n,\eps}\rightharpoonup w_{\eps}$ in $H^{1}_{0}(\Om_{\eps},|x|^{-2\nu})$ and $w_{n,\eps} \to w_{\eps}$ in $L^2(\Om_{\eps}, |x|^{-2\nu})$. As a result,  $w_{n,\eps}\to w_{\eps}$ point-wise almost everywhere. By the interpolation inequality,  we have $w_{n,\eps} \to w$ on $ L^{p+1}(\Om_{\eps}, |x|^{-(p+1)\nu}).$
Consequently $\displaystyle\int_{\Om_{\eps} }|x|^{-2\nu}w_{\eps}^{p+1} dx=1.$\\

Now we show that $S_{\eps}= \hat{F}(w_{\eps}).$ Clearly
$S_{\eps}\leq \hat{F}(w_{\eps})$. Furthermore,
applying Fatou's Lemma and the fact that $w\mapsto||w||^2_{H_0^1(\Om_{\eps}, |x|^{-2\nu}dx)}$ is weakly lower semicontinuous, we have
\begin{eqnarray*}S_{\eps}  &=&   \lim_{n \to \infty } \bigg[\frac{1}{2}\int_{\Om_{\eps}} |x|^{-2\nu}|\nabla w_{n,\eps}|^2 dx +  \frac{1}{q+1}\int_{\Om_{\eps}}|x|^{-(q+1)\nu} w_{n,\eps}^{q+1} dx\bigg] \\&\geq &  \bigg[\frac{1}{2}\int_{\Om_{\eps}}  |x|^{-2\nu} |\nabla w_{\eps}|^2 dx+  \frac{1 }{q+1}\int_{\Om_{\eps}}  |x|^{-(q+1)\nu}w_{\eps}^{q+1} dx\bigg]\\ &\geq & \hat{F}(w_{\eps}) .\end{eqnarray*}
Hence $S_{\eps}$ is achieved by $w_{\eps}$.\\
Using the Lagrange multiplier rule, we obtain $w_{\eps}$ satisfies
\be\lab{7-4-1} -\text{div} (|x|^{-2\nu} \nabla w_{\eps})=  \lambda_{\eps}|x|^{-(p+1)\nu} w_{\eps}^p - |x|^{-(q+1)\nu} w_{\eps}^q  ~~~~\text{ in } ~~~ \Om_{\eps},\ee
where $ \lambda_{\eps}= \lambda(\eps).$
Moreover,
$$\int_{\Om_{\eps}} |x|^{-2\nu} |\nabla w_{\eps}|^2dx=  \lambda_{\eps} \int_{\Om_{\eps}}|x|^{-(p+1)\nu} w_{\eps}^{p+1}dx -\int_{\Om_{\eps}} |x|^{-(q+1)\nu} w_{\eps}^{q+1}dx,$$ which implies that
$$ \lambda_{\eps}=\int_{\Om_{\eps}} |x|^{-2\nu} |\nabla w_{\eps}|^2dx+ \int_{\Om_{\eps}} |x|^{-(q+1)\nu} w_{\eps}^{q+1}dx.$$ This fact along with $\hat{F}(w_{\eps})=S_{\eps}$ implies
$$2S_{\eps}<\la_{\eps}<(q+1)S_{\eps}.$$
In Theorem \ref{A.1} and \ref{A.2}, if we take $\rho=\eps^{-\f{p-1}{2(q-p)}}$, then $N_{\rho}$ and $S_{\rho}$ of those theorems reduces to $N_{\eps}$ and $S_{\eps}$ defined as above. Therefore taking the limit $\eps\to 0$, it follows from  Theorem \ref{A.1} and \ref{A.2} that
\begin{equation}\label{S-eps-KS}
S_{\eps}\to \mathcal{K}  \quad{if}\  p>2^*-1 \quad\text{and}\quad S_{\eps}\to \frac{\mathcal{S}}{2} \quad{if}\  p=2^*-1.
\end{equation}
Hence there exist constants $\eps_0>0$ and $A, B>0$ such that
$$A<\la_{\eps}<B \quad\forall\quad\eps\in(0,\eps_0). $$
Using the transformation \eqref{v-w_eps}, we obtain from \eqref{7-4-1} that $v_{\eps}$ is a solution of \eqref{eq:eps_la}.
Moreover $\displaystyle\int_{\Om_{\eps}}|x|^{-(p+1)\nu}w_{\eps}^{p+1}dx=1$ implies
$\displaystyle\int_{\Om}|x|^{-(p+1)\nu}v_{\eps}^{p+1}dx=\eps^{\frac{p(N-2)-(N+2)}{2(q-p)}}.$
Hence $$\displaystyle\int_{\Om}|x|^{-(p+1)\nu}v_{\eps}^{p+1}dx=1 \quad\text{when}\quad p=2^*-1$$ and  $$\displaystyle\int_{\Om}|x|^{-(p+1)\nu}v_{\eps}^{p+1}dx\to 0 \quad\text{as}\ \eps\to 0 \quad\text{when}\quad p>2^*-1.$$
By a straight forward computation we see that $$F(w_{\eps})=\hat{F}(w_{\eps})=F(v_{\eps}),  \quad\text{when}\quad p>2^*-1$$
where $F$ and $\hat{F}$ are as in \eqref{a22} and \eqref{hatF} respectively. This along with  \eqref{S-eps-KS} and the fact that $F(w_{\eps})=S_{\eps}$ implies
$$F(v_{\eps})\to\mathcal{K} \quad\text{if}\ p>2^*-1$$
Moreover when $p=2^*-1$,
$$\mathcal{S}\leq S(v_{\eps})\leq 2\hat{F}(v_{\eps}, \Om)=2\hat{F}(w_{\eps},\Om_{\eps})=2S_{\eps}\To\mathcal{S}.$$ 
Hence 
 $$S(v_{\eps})\to\mathcal{S} \quad\text{if}\ p=2^*-1.$$ This completes the proof.
\end{proof}

\vspace{4mm}

\begin{proof} [{\bf Proof of Theorem \ref{main1}}] Let $v_{\eps}$ and $\la_{\eps}$ be as in Proposition \ref{p:main1}. Setting $u_{\eps}=\la_{\eps}^\frac{1}{p-1}v_{\eps}$, we find $u_{\eps}$ satisfies
$$-div(|x|^{-2\nu}\na u_{\eps})=|x|^{-(p+1)\nu}u_{\eps}^{p}-\eps\la_{\eps}^{-\frac{q-1}{p-1}}|x|^{-(q+1)\nu}u_{\eps}^{q} \quad\text{in}\quad\Omega.$$ Using the bounds on $\la_{\eps}$ from   Proposition \ref{p:main1}, we can conclude that there exist solutions $u_n$ of Problem \eqref{eq:a3} along a sequence $\{\eps_n\}$ of values of $\eps$ which tends to zero as $n$ tends to infinity. By setting $\la_n:=\la_{\eps_n}^{-\frac{1}{p-1}}$,  theorem follows from Proposition \ref{p:main1}.
\end{proof}

\section{The case $p=2^*-1$ and proof of Theorem \ref{main2}}
\begin{lemma}\label{bub} Let $v_{\eps}$ be as in Theorem \ref{main2}. Then $\|v_{\eps}\|_{\infty}\rightarrow +\infty$ as $\eps \rightarrow 0.$ \end{lemma}
\begin{proof}
We have
\be \label{zig1} \displaystyle\int_{\Om} |x|^{-2^{\star}\nu}v_{\eps}^{2^{\star}} dx=c,\ee where $c\in(A, B)$.
If possible, let $\|v_{\eps}\|_{\infty}$ be uniformly bounded. Hence by the Schauder estimate $v_{\eps}\rightarrow v$ in $C^{2}_{loc}(\Om\setminus \{0\}),$ where $v$ satisfies
\begin{equation}
  \label{b1}
  \left\{
    \begin{aligned}
     -\nabla (|x|^{-2\nu} \nabla v)&= |x|^{-2^{\star}\nu} v^{2^*-1}\; &&\text{in } \Om,\\
 v &\nequiv 0 &&\text{in }  \Om,  \\
v &= 0 &&\text{on } \pa \Om.
    \end{aligned}
  \right.
\end{equation}
Moreover, by the dominated convergence theorem we have
\be \label{zig11} \displaystyle A<\int_{\Om} |x|^{-2^{\star}\nu}v^{2^{\star}} dx<B.\ee
As $A>0$, the above expression implies  $v$ is nontrivial in a star-shaped domain which is a contradiction.\end{proof}
Define
\be\lab{ga-eps}
\gamma_{\eps}:=\|v_{\eps}\|_{\infty}^{-\f{2}{\al}}.
\ee
Therefore   $\|v_{\eps}\|_{\infty}= \gamma_{\eps}^{-\frac{\al}{2}}$ and $\gamma_{\eps}\rightarrow 0$ as $\eps\rightarrow 0.$ Define
\begin{equation}\label{z-eps}
z_{\eps}(x)=\gamma_{\eps}^{\frac{\al}{2}}v_{\eps} (\gamma_{\eps} x).
\end{equation}
Then $\|z_{\eps}\|_{\infty}=1$ and satisfies
\begin{equation}
  \label{1.41}
  \left\{
    \begin{aligned}  - \nabla (|x|^{-2\nu} \nabla z_{\eps})& = |x|^{-2^{\star}\nu} z_{\eps}^{2^*-1}-  \eps \gamma_{\eps}^{\f{(N+2)-q(N-2)}{2}}  |x|^{-(q+1)\nu} z_{\eps}^{q}  &&\text{ in }    \Om_{\eps}, \\
  z_{\eps}&> 0 &&\text{ in }  \Om_{\eps},\\ z_{\eps}& =0 &&\text{ in }    \pa \Om_{\eps},
 \end{aligned}
  \right.
\end{equation}
where $\Om_{\eps}= {\gamma_{\eps}^{-1}}\Om .$

\begin{lemma}\label{l:Z}
Suppose $z_{\eps}$ is as in \eqref{z-eps}, $0<\nu<\f{N-2}{4}$, $\f{N+2}{N-2}<q<\f{1+\nu}{\nu}$ and \eqref{sup} holds. Then
\begin{itemize}
\item[(i)] $\lim_{\eps\to 0}\eps \gamma_{\eps}^{\f{(N+2)-q(N-2)}{2}} =0$\\
\item[(ii)] There exists $Z\in D^{1,2}(\Rn, |x|^{-2\nu} )$ such that  $z_{\eps} \rightarrow Z$ in $C^{2}_{loc}(\R^N\setminus \{0\})\cap L^{\infty}(\Rn)$ as $\eps \rightarrow 0$.\\
\item[(iii)] $Z$ satisfies Eq. \eqref{ent} and given  by \eqref{ent-U}.
\end{itemize}
\end{lemma}
\begin{remark}
The upper bound of $q$ comes from the fact that limit of $\eps \gamma_{\eps}^{\f{(N+2)-q(N-2)}{2}}$ can be $\infty$ as $q$ is supercritical. To exclude this option we need to put this restriction on $q$. Note that,  when $q$   is critical or subcritical, the above limit is always $0$. Therefore in the subcritical case no extra restriction on the upper bound of $q$ appears.
\end{remark}

\begin{proof}
Being defined as in \eqref{z-eps}, $z_{\eps}$ satisfies Eq.\eqref{1.41}. Let $\phi\in C^{\infty}_0(\Rn)$. Thus $\phi\in C^{\infty}_0(\Omega_{\eps})$ for $\eps$ small. Taking $\phi$ as the test function, from Eq.\eqref{1.41} we have
\begin{equation}
\displaystyle\int_{\Om_{\eps}}|x|^{-2\nu}\na z_{\eps}\na\phi=\int_{\Om_{\eps}}|x|^{-2^*\nu}z_{\eps}^{2^{*}-1}\phi-\eps\ga_{\eps}^{\f{(N+2)-q(N-2)}{2}}\int_{\Om_{\eps}}|x|^{-(q+1)\nu}z_{\eps}^q\phi.
\end{equation}

{\bf Case 1}: $\eps\ga_{\eps}^{\f{(N+2)-q(N-2)}{2}}$ is bounded.\\
Therefore there exists $c\geq 0$ such that $\eps\ga_{\eps}^{\f{(N+2)-q(N-2)}{2}}\to c$ (along a subsequence).  Furthermore, by the elliptic regularity theory it follows that $z_{\eps}\to Z$ in $C^2_{loc}(\Rn\setminus\{0\})$.

Suppose $c>0$.
Since $z_{\eps}\to Z$ a.e and $||z_{\eps}||_{L^{\infty}}=1$, by dominated convergence theorem it follows
\be\lab{27-3-1}
\lim_{\eps\to 0}\eps\ga_{\eps}^{(N+2)-q(N-2)}\displaystyle\int_{\Om_{\eps}}|x|^{-(q+1)\nu}z_{\eps}^q\phi= c\int_{\Rn}|x|^{-(q+1)\nu}Z^q\phi;\ee
\be\lab{27-3-2}
\lim_{\eps\to 0}\displaystyle\int_{\Om_{\eps}}|x|^{-2^*\nu}z_{\eps}^{2^{*}-1}\phi=\int_{\Rn}|x|^{-2^*\nu}Z^{2^{*}-1}\phi.\ee

{\bf Claim}: $\||x|^{-\nu}\na z_{\eps}\|_{L^2(\Om_{\eps})}$ is uniformly bounded with respect to $\eps$.\\
Assuming the claim,
\be\lab{27-3-3}\displaystyle\lim_{\eps\to 0}\int_{\Om_{\eps}}|x|^{-2\nu}\na z_{\eps}\na\phi=\int_{\Rn}|x|^{-2\nu}\na Z\na\phi,\ee

follows from Vitali's convergence theorem, since $\na z_{\eps}\to\na Z$ a.e. in $\Rn$.

To prove the claim, we see
\begin{eqnarray}\lab{20-4-4}
\displaystyle\int_{\Om_{\eps}}|x|^{-2\nu}|\na z_{\eps}|^2 &=&\int_{\Om_{\eps}}|x|^{-2^*\nu}z_{\eps}^{2^*}-\eps\ga_{\eps}^{\f{(N+2)-q(N-2)}{2}}\int_{\Om_{\eps}}|x|^{-(q+1)\nu}z_{\eps}^{q+1}\no\\
&\leq&\int_{\Om}|x|^{-2^*\nu}v_{\eps}^{2^*}=1.
\end{eqnarray}
Combining \eqref{27-3-1}-\eqref{27-3-3}, we have
\be\lab{2-4-16}
-\text{div}(|x|^{-2\nu}\na Z)=|x|^{-2^*\nu}Z^{2^*-1}-c|x|^{-(q+1)\nu}Z^q \quad\text{in}\quad\Rn.\ee
Moreover, by Fatou's lemma
\bea\lab{20-4-3}
\displaystyle c\int_{\Rn}|x|^{-(q+1)\nu}Z^{q+1}dx&\leq&\lim \inf_{\eps\to 0}\eps\ga_{\eps}^{\f{(N+2)-q(N-2)}{2}}\int_{\Om_{\eps}}|x|^{-(q+1)\nu}z_{\eps}^{q+1}dx\no\\
&=&\lim\inf_{\eps\to 0}\big[\int_{\Om_{\eps}}|x|^{-2^*\nu}z_{\eps}^{2^*}dx-\int_{\Om_{\eps}}|x|^{-2\nu}|\na z_{\eps}|^2dx\big]\no\\
&\leq& 1.
\eea
Since $c>0$ and $z_{\eps}\to Z$ in $C^2_{loc}(\Rn\setminus\{0\})$, from \eqref{20-4-4} and \eqref{20-4-3}, it follows  that $Z\in D^{1,2}(\R^N,  |x|^{-2\nu}) \cap L^{q+1}(\R^N, |x|^{-(q+1)\nu})$. Therefore using Pohozaev identity (see \eqref{poh2}), we have
$$c\bigg( \frac{N-2}{2}- \frac{N}{q+1}\bigg) \int_{\Rn}  |x|^{-(q+1)\nu} Z^{q+1} dx=0,$$ which is a contradiction as $|Z|_{L^{\infty}}=1$. Therefore, $c=0$. Consequently,  \eqref{2-4-16} yields $Z$ satisfies
\eqref{ent}.\\
{\bf Case 2}: $\displaystyle{\lim_{\eps\to 0}}\eps\ga_{\eps}^{\f{(N+2)-q(N-2)}{2}}=\infty$.\\
Set, $\la_{\eps}:= \eps\ga_{\eps}^{\f{(N+2)-q(N-2)}{2}}$ and define
$\tilde z_{\eps}(x):= z_{\eps}(\f{x}{\la_{\eps}^m})$, where $m=\f{1}{2+(q-1)\nu}$.\\ A straight forward computation yields, for any $\psi\in C^{\infty}_0(\Rn)$, we have
\bea\lab{2-4-1}
\int_{\la_{\eps}^m\Om_{\eps}}|x|^{-2\nu}\na\tilde z_{\eps}(x)\na\psi(x)dx &=&\la_{\eps}^\f{2^*\nu-2\nu-2}{2+(q-1)\nu}\int_{\la_{\eps}^m\Om_{\eps}}|x|^{-2^*\nu}\tilde z_{\eps}^{2^*-1}(x)\psi(x)dx \no\\
&-&\int_{\la_{\eps}^m\Om_{\eps}}|x|^{-(q+1)\nu}\tilde z_{\eps}^{q}(x)\psi(x)dx.
\eea
Since $\la_{\eps}\to\infty$ as $\eps\to 0$, we obtain $\la_{\eps}^m\Om_{\eps}\to\Rn$ and $\la_{\eps}^\f{2^*\nu-2\nu-2}{2+(q-1)\nu}\to 0$. Using elliptic regularity theory we can argue as before that there exists $\tilde Z$ such that $\tilde z_{\eps}\to \tilde Z$ in $C^2_{loc}(\Rn\setminus\{0\})$. Moreover, $||z_{\eps}||_{L^{\infty}}=1$ implies $||\tilde z_{\eps}||_{L^{\infty}}=1$. Therefore arguing as in Case 1, we  can prove that $\tilde Z$ satisfies the following equation:
\be\lab{2-4-2}
-\text{div}(|x|^{-2\nu}\na \tilde Z)+|x|^{-(q+1)\nu}\tilde Z^q=0 \quad\text{in}\quad\Rn.
\ee
From Theorem \ref{t:ap-B1} (see Appendix A), it follows that $\tilde Z=0$. This is a contradiction as $||\tilde z_{\eps}||_{L^{\infty}}=1$ implies $||\tilde Z||_{L^{\infty}}=1$. Hence Case 2 can not occur. Therefore from Case 1 we conclude (i) holds
and $z_{\eps}\to Z$ in $C^{2}_{loc}(\Rn\setminus\{0\})$.

Since $Z$ satisfies \eqref{ent}, $Z$ must be of the form $\mu^{-\frac{\al}{2}}U(\frac{x}{\mu})$, where $U$ is as in \eqref{ent-U} for some $\mu>0$. By \eqref{sup}, it follows that max $z_{\eps}=z_{\eps}(0)=1$. This implies $Z(0)=1$ and $0\leq Z\leq 1$. From this it follows $z_{\eps}\to Z$ in $L^{\infty}_{loc}(\Rn)$ and $\mu=\bigg(\al\sqrt{\frac{N}{N-2}}\bigg)^\frac{N-2}{\al}$. From this, direct calculation yields that $Z(x)=\bigg(1+\frac{|x|^\frac{2\alpha}{N-2}}{\frac{N{\alpha}^2}{N-2}}\bigg)^{-\frac{N-2}{2}}$.
\end{proof}

We know the local behavior of $z_{\eps}.$ Now we need to check the behavior of $z_{\eps}$ near $\infty.$ Hence define the Kelvin transform of  $z_{\eps}$ as \be \hat{z}_{\eps}(x)=|x|^{-\al} z_{\eps}\bigg(\frac{x}{|x|^2}\bigg) \text{ in } \Om_{\eps}\setminus \{0\}.\ee
 Then from \eqref{1.41},  $\hat{z}_{\eps}$ satisfies
\begin{equation}
  \label{1.34}
  \left\{
    \begin{aligned}  -div(|x|^{-2\nu} \nabla\hat{z}_{\eps})& = |x|^{-2^{\star}\nu}\hat{z}_{\eps}^{2^*-1}-\eps \gamma_{\eps}^{\f{(N+2)-q(N-2)}{2}}   |x|^{-(q+1)\nu +\al(q- 2^*+1)}   \hat{z}_{\eps} ^{q}  \text{ in  }  \Om^{\star}_{\eps}\\
    \hat{z}_{\eps}&=0 \text{ on  }  \pa \Om^{\star}_{\eps}.
\end{aligned}
  \right.
\end{equation}
where $\Om^{\star}_{\eps}$ is the image $\Om_{\eps}$ under the Kelvin transform. Hence the behavior of $z_{\eps}$ near $\infty$ amounts to behavior of  $\hat{z}_{\eps}$ near $0.$

\begin{lemma}\label{lb}
There exist $R>0$ and $C>0$ independent of $\eps>0$ such that any solution of \eqref{1.34} satisfy
\be\label{zig2} \|\hat{z}_{\eps}\|_{L^{\infty}(B_{r})}\leq C \bigg(\int_{B_{R}} |x|^{-2^{\star}\nu} \hat{z}_{\eps}^{2^{\star}} dx\bigg)^{\frac{1}{2^{\star}}}.\ee
\end{lemma}

\begin{proof} The proof of the above lemma follows along the same line of arguments as in Theorem \ref{t:up-est} (i) with a suitable modification and we skip the proof.  \end{proof}

\begin{remark}
There exists $C>0$ independent of $\eps>0$ such that $z_{\eps} \leq C Z(x) $ for all $x\in \Om_{\eps}.$
For this, note that $\|z_{\eps}\|_{\infty}=1$, this implies that $z_{\eps} \leq C Z(x) $ locally.
From \eqref{zig11} we have
\be\no  A< \int_{\Om_{\eps}} |x|^{-2^*\nu} z_{\eps}^{2^*} dx<B  .\ee But this implies that
\be\no    \int_{B_{R}} |x|^{-2^{\star}\nu} \hat{z}_{\eps}^{2^{\star}} dx\leq \int_{\Om_{\eps}} |x|^{-2^*\nu} z_{\eps}^{2^*} dx  <B \ee
and since at infinity $Z$ decays as $|x|^{-\al}$, we have $z_{\eps} \leq C Z(x) $ near infinity. Hence, we have $z_{\eps} \leq C Z(x) $ for all $x\in \Om_{\eps}.$
As a conclusion, from \eqref{z-eps} we obtain that there exists $C>0$ independent of $\eps$ such that \be\label{upps} v_{\eps}(x) \leq C \gamma_{\eps}^{-\frac{\al}{2}}Z\bigg(\frac{x}{\gamma_{\eps}}\bigg) .\ee
\end{remark}

\noi Define $w_{\eps}(x)=\|v_{\eps}\|_{\infty}v_{\eps}(x)= \gamma_{\eps}^{-\frac{\al}{2}}v_{\eps}(x).$ Then $w_{\eps}$ satisfies
\begin{equation}
  \label{1.37}
  \left\{
    \begin{aligned}  -div (|x|^{-2\nu} \nabla w_{\eps})& = \gamma_{\eps}^{-\frac{\al}{2}}  |x|^{-2^{\star}\nu}v_{\eps}^{2^*-1}-\eps   \gamma_{\eps}^{-\frac{\al}{2}}  |x|^{-(q+1)\nu} v_{\eps}^q \quad\text{ in  } \quad\Om\\
    w_{\eps}&=0 \quad\text{ on  } \quad \pa \Om.
\end{aligned}
  \right.
\end{equation}

\begin{lemma} \label{p8} Let $\nu$ and $q$ be as in Lemma \ref{l:Z} and $w_{\eps}$ be as in \eqref{1.37}. Then  there exists a constant $C>0$ such that $$\|w_{\eps}\|_{L^\infty(K)}+ \|\nabla w_{\eps}\|_{L^\infty(K)}\leq C,$$
for every compact subset $K$ of $\Om \setminus \{0\}.$\end{lemma}

\begin{proof}
Using the Green kernel's representation and Lemma \ref{p}, we have
\begin{eqnarray*}|  w_{\eps}(x)|&=&   \gamma_{\eps}^{-\frac{\al}{2}} \bigg |\int_{\Om} G(x,y)[ |y|^{-2^{\star}\nu}v_{\eps}^{2^*-1} - \eps |y|^{-(q+1)\nu}v_{\eps}^{q} ]dy\bigg|
\no \\&\leq& C \gamma_{\eps}^{-\frac{\al}{2}}   |x|^{2\nu}\int_{\Om}|x-y|^{2-N}|y|^{-2^{\star}\nu}v_{\eps}^{2^*-1} dy\\&+&  C  \gamma_{\eps}^{-\frac{\al}{2}} \int_{\Om}|x-y|^{2+2\nu-N}|y|^{-2^{\star}\nu}v_{\eps}^{2^*-1} dy \\&+& C \eps  \gamma_{\eps}^{-\frac{\al}{2}} |x|^{2\nu}\int_{\Om}|x-y|^{2-N}|y|^{-(q+1)\nu}v_{\eps}^{q} dy\\&+&  C \eps \gamma_{\eps}^{-\frac{\al}{2}} \int_{\Om}|x-y|^{2+2\nu-N}|y|^{-(q+1)\nu}v_{\eps}^{q} dy\\
&=:&I_1+I_2+I_3+I_4.
\end{eqnarray*}
Moreover,
\begin{eqnarray*}  I_1 &:=& C|x|^{2\nu}\gamma_{\eps}^{-\frac{\al}{2}}\int_{\Om}|x-y|^{2-N}|y|^{-2^{\star}\nu}v_{\eps}^{2^*-1} dy\\
&=&C|x|^{2\nu}  \gamma_{\eps}^{-\frac{\al}{2}} \int_{\Om \cap B_{\frac{|x|}{2}}(0)}|x-y|^{2-N}|y|^{-2^{\star}\nu}v_{\eps}^{2^*-1} dy \\&+&C|x|^{2\nu}  \gamma_{\eps}^{-\frac{\al}{2}} \int_{\Om_{}\setminus B_{\frac{|x|}{2}}(0)} |x-y|^{2-N}|y|^{-2^{\star}\nu}v_{\eps}^{2^*-1} dy\\
&=:& I_{11}+I_{12}.
\end{eqnarray*}
Using \eqref{upps} along with the facts that $Z(x)\thicksim |x|^{-\al}$ at infinity and $\ga_{\eps}\to 0$,   we find
\be\lab{24-3-1}  \gamma_{\eps}^{-\frac{\al}{2}} |y|^{-2^{\star}\nu}v_{\eps}^{2^*-1}(y) \leq \frac{C}{|y|^{(N+2)- \frac{4\nu}{(N-2)}}} \quad\text{in}\quad \Om\setminus
B_{\frac{|x|}{2}}(0),
\ee
\be\lab{24-3-2}  \gamma_{\eps}^{-\frac{\al}{2}} |y|^{-(q+1)\nu}v_{\eps}^{q}(y) \leq \frac{C\ga_{\eps}^{(q-1)\f{\al}{2}}}{|y|^{(N-2)q-\nu(q-1)}}  \quad\text{in}\quad \Om\setminus
B_{\frac{|x|}{2}}(0).
\ee
 Hence
\bea  I_{12} &\leq &  C|x|^{2\nu} \int_{\Om\setminus B_{\frac{|x|}{2}}(0)}\frac{1}{|x-y|^{N-2}|y|^{(N+2)- \frac{4\nu}{N-2}}} dy\no\\& \leq&  \frac{C|x|^{2\nu}}{|x|^{(N+2)- \frac{4\nu}{N-2}}} \int_{\Om\setminus B_{\frac{|x|}{2}}(0)} |x-y|^{2-N} dy \no . \eea
When $y\in \Om\cap B_{\frac{|x|}{2}}(0)$, we have
$|x-y|\geq |x|-|y|\geq \frac{1}{2}|x|$. Therefore using \eqref{upps}, we get
\bea\lab{24-3-3} I_{11}&\leq& \frac{C|x|^{2\nu}  \gamma_{\eps}^{-\frac{\al}{2}}}{|x|^{N-2}}  \int_{\Om\cap B_{\frac{|x|}{2}}(0)}|y|^{-2^{\star}\nu}v_{\eps}^{2^*-1}(y)dy \no \\&\leq& \frac{C\ga_{\eps}^{-\f{2^*\al}{2}}}{|x|^{\al}}  \int_{\Om\cap B_{\frac{|x|}{2}}(0)}|y|^{-2^{\star}\nu}Z\displaystyle\left(\f{y}{\ga_{\eps}}\right)^{2^*-1} dy\no \\&\leq& \frac{C}{|x|^{\al}}\ga_{\eps}^{-\f{2^*\al}{2}-2^*\nu+N}
\int_{\Rn}|y|^{-2^{\star}\nu}Z(y)^{2^*-1}dy\no\\
&=&\displaystyle\frac{C}{|x|^{\al}} \int_{\Rn}|y|^{-2^{\star}\nu}Z(y)^{2^*-1}dy \no\\
&=& \displaystyle\frac{C}{|x|^{\al}}\om_N\al^{N-1}\left(\f{N}{N-2}\right)^\f{N-2}{2},
\eea
where the last integral can be computed as in  \eqref{ap:ga-0'} in Lemma \ref{A.3}.
Similarly $I_2$ can be written as
\bea
I_2 &=& C  \ga_{\eps}^{-\frac{\al}{2}} \int_{\Om \cap B_{\frac{|x|}{2}}(0)}|x-y|^{2+\nu-N}|y|^{-2^{\star}\nu}v_{\eps}^{2^*-1} dy\no\\
&+& C\ga_{\eps}^{-\frac{\al}{2}} \int_{\Om_{}\setminus B_{\frac{|x|}{2}}(0)} |x-y|^{2+\nu-N}|y|^{-2^{\star}\nu}v_{\eps}^{2^*-1} dy\no\\
&=:&I_{21}+I_{22}.
\eea
Proceeding similarly as we did for $I_{12}$ and $I_{11}$, we have
\bea  I_{22}&\leq &  C \int_{\Om\setminus B_{\frac{|x|}{2}}(0)}\frac{1}{|x-y|^{N-2\nu-2}|y|^{(N+2)- \frac{4\nu}{N-2}}} dy\no\\& \leq&  \frac{C}{|x|^{(N+2)- \frac{4\nu}{N-2}}} \int_{\Om\setminus B_{\frac{|x|}{2}}(0)} |x-y|^{2+2\nu -N} dy; \no \\
 I_{21} &\leq& \displaystyle\frac{C}{|x|^{\al}}\om_N\al^{N-1}\left(\f{N}{N-2}\right)^\f{N-2}{2}. \eea
Similarly to compute $I_3$ and $I_4$, we break those integral into two parts, namely in $\Om\cap B_{\f{|x|}{2}}(0)$ and $\Om\setminus B_{\f{|x|}{2}}(0)$. Using \eqref{24-3-2}, integral in $\Om\setminus B_{\f{|x|}{2}}(0)$ can be computed as before.
Proceeding as in \eqref{24-3-3}, we have
\bea
&& \eps  \gamma_{\eps}^{-\frac{\al}{2}} |x|^{2\nu}\int_{\Om\cap B_{\f{|x|}{2}}(0)}|x-y|^{2-N}|y|^{-(q+1)\nu}v_{\eps}^{q} dy \no\\
&&\leq
\frac{C\eps \gamma_{\eps}^{-\frac{\al}{2}}}{|x|^{\al}}  \int_{\Om\cap B_{\frac{|x|}{2}}(0)}|y|^{-(q+1)\nu}v_{\eps}^{q}(y)dy\no\\
&&\leq \frac{C\eps \gamma_{\eps}^{\f{(N+2)-q(N-2)}{2}}}{|x|^{\al}}\int_{\Rn}|y|^{-(q+1)\nu}Z^q(y)dy\no
\eea
By a straight forward computation using the expression of $Z$ from Lemma \ref{l:Z}, it can shown that $\displaystyle\int_{\Rn}|y|^{-(q+1)\nu}Z^q(y)dy<\infty$.
Moreover, as $\eps \gamma_{\eps}^{\f{(N+2)-q(N-2)}{2}}\to 0$ (see Lemma \ref{l:Z}), we can conclude that for any compact set $K\subset \Om \setminus\{0\},$
we have $\|w_{\eps}\|_{L^\infty(K)}\leq C$ and  by the regularity $\|\nabla w_{\eps}\|_{L^\infty(K)}\leq C.$
\end{proof}
\begin{lemma}
\label{p7.1} Let $\nu,q, w_{\eps}$ be as in Lemma \ref{p8}. Then there exists $\gamma_{0}>0$  such that
\be\label{p7.1'}\lim_{\eps \rightarrow 0}w_{\eps}(x)=\gamma_{0} G(x, 0)~\text{in } C^{1}_{loc}(\overline{\Om} \backslash \{0\}).\ee
\end{lemma}
\begin{proof}
Define $$f_{\eps}:=\gamma_{\eps}^{-\frac{\al}{2}}  |x|^{-2^{\star}\nu}v_{\eps}^{2^*-1}-\eps   \gamma_{\eps}^{-\frac{\al}{2}}  |x|^{-(q+1)\nu} v_{\eps}^q.$$
Choose  $R>0$ such that $\Om'=\Om\setminus \overline{B_{R}}(0)$ is connected.
Then $ |w_{\eps}|+ |\nabla w_{\eps}|\leq C$ for all $x\in \Om'$. Let
$x'\in \pa \Om \cap \pa \Om',$ then $|w_\eps(x)- w_{\eps}(x')|\leq C$ for
all $x\in \Om'.$ But this implies  $w_{\eps}$ is uniformly bounded in
$\Om'\cap \overline{\Om}.$  By the standard regularity, we have
$w_{\eps}\rightarrow w$ as $\eps \to 0$ in
$C^{1}_{loc}(\overline{\Om}\setminus 0)$. If $K\subseteq\bar\Om\setminus\{0\}$, then for any $x\in K$ and $r>0$ small, using the fact  $\gamma_{\eps}\to 0$, we have
\bea\lab{27-3-4}
w_{\eps}(x)&=&\int_{\Om} G(x,y) f_{\eps}(y)dy\no \\&=&  \int_{B_{r}(0)} G(x,y)  f_{\eps}(y) dy+ \int_{\Om\setminus B_{r}(0)} G(x,y)  f_{\eps}(y)dy \no \\&=&   \gamma_{\eps}^{-\frac{\al}{2}} \int_{B_{r}(0)} G(x,y) |y|^{-2^{\star}\nu}v_{\eps}^{2^*-1}(y)dy\no\\&+&
\gamma_{\eps}^{-\frac{\al}{2}} \int_{\Om\setminus B_{r}(0)} G(x,y) |y|^{-2^{\star}\nu}v_{\eps}^{2^*-1}(y)dy\no\\&+&
 \eps \gamma_{\eps}^{-\frac{\al}{2}} \int_{B_{r}(0)} G(x,y) |y|^{-(q+1)\nu}v_{\eps}^{q}(y)dy\no\\&+&
 \eps \gamma_{\eps}^{-\frac{\al}{2}} \int_{\Om\setminus B_{r}(0)} G(x,y) |y|^{-(q+1)\nu}v_{\eps}^{q}(y)dy
 \eea
{\bf Claim:}(i) $\eps \gamma_{\eps}^{-\frac{\al}{2}} \displaystyle\int_{\Om\setminus B_{r}(0)} G(x,y) |y|^{-(q+1)\nu}v_{\eps}^{q}(y)dy=o(1)$ and \\ (ii) $\gamma_{\eps}^{-\frac{\al}{2}} \int_{\Om\setminus B_{r}(0)} G(x,y) |y|^{-2^{\star}\nu}v_{\eps}^{2^*-1}(y)dy=o(1).$ \\

To see this,
\bea
&&\lim_{\eps\to 0}\eps \gamma_{\eps}^{-\frac{\al}{2}} \displaystyle\int_{\Om\setminus B_{r}(0)} G(x,y) |y|^{-(q+1)\nu}v_{\eps}^{q}(y)dy\no\\
&&
\leq\lim_{\eps\to 0} C\eps\ga_{\eps}^{-\f{\al}{2}(q+1)}\int_{\Om\setminus B_r(0)}\big[|x|^{2\nu}|x-y|^{2-N}+|x-y|^{-\al}\big]|y|^{-(q+1)\nu}Z^q(\f{y}{\ga_{\eps}})dy\no\\
&&\leq\lim_{\eps\to 0}C\eps\ga_{\eps}^{-\f{\al}{2}(q+1)}\int_{\Om\setminus B_r(0)}\big[|x|^{2\nu}|x-y|^{2-N}+|x-y|^{-\al}\big]|y|^{-(q+1)\nu}|\f{y}{\ga_{\eps}}|^{-\al q}dy\no\\
&&\leq\lim_{\eps\to 0}C\eps\ga_{\eps}^{\f{\al}{2}(q-1)}\int_{\Om\setminus B_r(0)}\big[|x|^{2\nu}|x-y|^{2-N}+|x-y|^{-\al}\big]|y|^{-(q+1)\nu-\al q}dy\no\\
&&= o(1) \no.
\eea
Similarly (ii) follows. Therefore  from \eqref{27-3-4}, we obtain
\bea \lim_{\eps\to 0} w_{\eps}(x)&=&
\gamma_{\eps}^{-\frac{\al}{2}} \int_{B_{r}(0)} G(x,y) |y|^{-2^{\star}\nu}v_{\eps}^{2^*-1}(y)dy\no\\&+&
 \eps \gamma_{\eps}^{-\frac{\al}{2}} \int_{B_{r}(0)} G(x,y) |y|^{-(q+1)\nu}v_{\eps}^{q}(y)dy+o(1).\no
\eea
Furthermore $G(x,.)$ is continuous in
$\overline{\Om}\setminus\{x\},$ we obtain
\be w_{\eps}(x)= \gamma_{\eps}^{-\frac{\al}{2}} G(x,0)  \int_{B_{r}(0)} |y|^{-2^{\star}\nu}v_{\eps}^{2^*-1}dy+ L+ o(1).\ee
where
\Bea L =\eps \gamma_{\eps}^{-\frac{\al}{2}} G(x,0)  \int_{B_{r}(0)} |y|^{-(q+1)\nu}v_{\eps}^{q}dy. \Eea
By doing a straight forward computation using \eqref{upps}, it follows
$$L\leq \eps\ga_{\eps}^{\f{(N+2)-q(N-2)}{2}}\displaystyle\int_{\Rn}|y|^{-(q+1)\nu}Z^q(y)dy.$$

Consequently, a direct computation using Lemma \ref{l:Z} yields $L=o(1)$. \\
Define
$$ \gamma_{0} :=\lim_{r \to 0}\lim_{\eps\to 0} \gamma_{\eps}^{-\frac{\al}{2}} \int_{B_{r}(0)}  |y|^{-2^{\star}\nu}v_{\eps}^{2^*-1}dy. $$
Then
\be \lim_{\eps \to 0} w_{\eps}(x)= \gamma_{0} G(x, 0).\ee
Moreover from Lemma \ref{A.3} we get $\ga_0=\om_{N}(N-2-2\nu)^{N-1}\bigg(\frac{N}{N-2}\bigg)^\frac{N-2}{2}.$

Using the same  procedure as above we can show that $$  w_{\eps} \to \gamma_{0}G(x, 0) \text{ in }  C^{1}_{loc}(\overline{\Om}\setminus 0).
$$
\end{proof}

\begin{lemma}\label{A.3}
Let  $v_{\eps}$  be as in Theorem \ref{main2} and $\ga_{\eps}$ be as defined in \eqref{ga-eps}.  Define $\ga_0:=\displaystyle{\lim_{r \to 0}\lim_{\eps\to 0}\gamma_{\eps}^{-\frac{\al}{2}} \displaystyle\int_{B_{r}(0)}  |y|^{-2^{\star}\nu}v_{\eps}^{2^*-1}dy. }$ Then
\begin{equation}\lab{ap:ga-0}
\ga_0=\om_{N}(N-2-2\nu)^{N-1}\bigg(\frac{N}{N-2}\bigg)^\frac{N-2}{2}.
\end{equation}
\end{lemma}
\begin{proof}
We define $I_{\eps, r}:=\gamma_{\eps}^{-\frac{\al}{2}} \displaystyle\int_{B_{r}(0)}  |y|^{-2^{\star}\nu}v_{\eps}^{2^*-1}dy$.
Since $v_{\eps}$ and $z_{\eps}$ are related by \eqref{z-eps}, we have $v_{\eps}(x)=\ga_{\eps}^{-\frac{\al}{2}}z_{\eps}\big(\frac{x}{\ga_{\eps}}\big)$. Thus
\begin{equation}
I_{\eps,r}=\ga_{\eps}^{-\frac{\al}{2}-2^*\nu-\frac{\al}{2}(2^*-1)+N}\int_{\frac{B_r(0)}{\ga_{\eps}}}|x|^{-2^*\nu}z_{\eps}^{2^*-1}(x)dx=\int_{\frac{B_r(0)}{\ga_{\eps}}}|x|^{-2^*\nu}z_{\eps}^{2^*-1}(x)dx
\end{equation}
Since $\eps\to 0$ implies $\ga_{\eps}\to 0$, we have
\be\lab{ap:ga-0'}
\ga_0=\lim_{r\to 0}\lim_{\eps\to 0}I_{\eps, r}=\int_{\Rn}|x|^{-2^*\nu}Z^{2^*-1}dx,
\ee
 where $Z$ is as in Lemma \ref{l:Z}. Therefore by doing a straight forward computation, we obtain
$$\ga_0=\frac{\om_N N\al}{2}\bigg(\frac{N\al^2}{N-2}\bigg)^\frac{N-2}{2} B\bigg(\frac{N}{2}, 1\bigg),$$
where $B(a, b)=\displaystyle\int_{0}^{\infty}t^{a-1}(1+t)^{-a-b}dt$ is the Beta function. Recall that $B(a, b)=\f{\Ga(a)\Ga(b)}{\Ga(a+b)}$.
Thus $B\big(\frac{N}{2}, 1\big)=\frac{\Ga(\frac{N}{2})}{\Gamma(\frac{N}{2}+1)}=\frac{\Ga(\frac{N}{2})}{\f{N}{2}\Gamma(\frac{N}{2})}=\frac{2}{N}$, the lemma follows.
\end{proof}

\vspace{2mm}

\begin{proof}[{\bf Proof of Theorem \ref{main2} }] From \eqref{poh2} we have
$$\frac{1}{2} \int_{\pa \Om}  |x|^{-2\nu}\langle x, n \rangle |\nabla v_{\eps}|^2 dS=  \eps  \bigg( \frac{N-2}{2}- \frac{N}{q+1}\bigg) \int_{\Om}  |x|^{-(q+1)\nu} v_{\eps}^{q+1} dx .
$$
Using $w_{\eps}=||v_{\eps}||_{\infty}v_{\eps}$ in the above expression, we have
\Bea && \int_{\pa \Om} |x|^{-2\nu}|\nabla w_{\eps}|^2\langle x, n\rangle dS= 2\eps\bigg( \frac{N-2}{2}- \frac{N}{q+1}\bigg)  \|v_{\eps}\|_{\infty}^{2}  \int_{\Om}  |x|^{-(q+1)\nu} v_{\eps}^{q+1} dx \no \\&=& 2\eps \bigg( \frac{N-2}{2}- \frac{N}{q+1}\bigg)  \|v_{\eps}\|_{\infty}^{\frac{q(N-2)-(N+2)+2\al}{\al}}  \int_{\Om_{\eps}} |x|^{-(q+1)\nu} z_{\eps}^{q+1} dx .\Eea
Since $z_{\eps}\to Z$ a,e and $z_{\eps}\leq CZ$, by the dominated convergence theorem it follows $\displaystyle\int_{\Om_{\eps}} |x|^{-(q+1)\nu} z_{\eps}^{q+1} dx\to \int_{\Rn} |x|^{-(q+1)\nu} Z^{q+1} dx$.
Therefore taking limit $\eps\to 0$ and using \eqref{p7.1'} and Lemma \ref{dub1}, we obtain
\be\label{6-i} \lim_{\eps \rightarrow 0 }\eps\|v_{\eps}\|_{\infty}^{\frac{q(N-2)-(N+2)+2\al}{\al}}    = \frac{(N-2-2\nu) \gamma_{0}^2|R(0)|}{2\bigg( \frac{N-2}{2}- \frac{N}{q+1}\bigg)  \displaystyle\int_{\R^N}|x|^{-(q+1)\nu}Z ^{q+1}dx }.\ee
From Lemma \ref{l:Z}, we know $Z(x)=\bigg(1+\frac{|x|^\frac{2\alpha}{N-2}}{\frac{N{\alpha}^2}{N-2}}\bigg)^{-\frac{N-2}{2}}$.  Therefore a straight forward calculation yields
\begin{eqnarray}\label{6-ii}
\displaystyle\int_{\R^N}|x|^{-(q+1)\nu}Z ^{q+1}dx &=&\frac{\om_N N\al}{2}\left(\frac{N\al^2}{N-2}\right)^{\big(N-(q+1)\nu\big)\frac{N-2}{2\al}-1}\no\\
&\times& B\left(\frac{N-2}{2\al}(N-(q+1)\nu), \frac{N-2}{2\al}\{q(N-2-\nu)-(2+\nu)\}\right),
\end{eqnarray}
where $B(a, b)=\displaystyle\int_0^{\infty}t^{a-1}(1+t)^{-a-b}dt$ and $\om_N$ is the surface measure of unit sphere in $\Rn$.
From Lemma \ref{A.3}, it is known that
$\ga_0=\om_N\al^{N-1}\bigg(\frac{N}{N-2}\bigg)^\frac{N-2}{2}$.
Substituting the value of  $\ga_0$, \  $\al$, and $\displaystyle\int_{\R^N}|x|^{-(q+1)\nu}Z ^{q+1}dx$  in \eqref{6-i}, we have RHS of \eqref{6-i} as
\begin{eqnarray}
RHS &=\frac{\al(q+1)|R(0)|}{\{(N-2)q-(N+2)\}}\cdot \frac{\om_N^2\al^{2(N-1)}\big(\frac{N}{N-2}\big)^{N-2}}{\om_N\frac{N\al}{2}\big(\frac{N\al^2}{N-2}\big)^{\big(N-(q+1)\nu\big)\big(\frac{N-2}{2\al}\big)-1}}\times\no\\
&\bigg[B\bigg(\frac{N-2}{2\al}(N-(q+1)\nu), \frac{N-2}{2\al}\{q(N-2-\nu)-(2+\nu)\}\bigg)\bigg]^{-1}\no\\
&=\frac{\om_N|R(0)|}{C_{q,N}}\frac{(N-2-2\nu)^\frac{\big(N-(q+1)\nu\big)(N-2)-4N\nu}{\al}(N-2)^\frac{\big(N-(q+1)\nu\big)(N-2)-2\al(N-1)}{2\al}}{N^\frac{\big(N-(q+1)\nu-2\al\big)(N-2)}{2\al}}\no\\
&\times\bigg[B\bigg(\frac{N-2}{2\al}(N-(q+1)\nu), \frac{N-2}{2\al}\{q(N-2-\nu)-(2+\nu)\}\bigg)\bigg]^{-1},\no
\end{eqnarray}
where $C_{q,N}$ is as in \eqref{C_qn}.

Furthermore, \eqref{nsd2} follows from \eqref{p7.1'}.
\end{proof}

\appendix
\section{}
In this section we consider the following problem:
\bea\lab{2-4-3}
-\text{div}(|x|^{-2\nu}\na u)+|x|^{-(q+1)\nu}u^q &=& 0 \quad\text{in}\quad\Rn,\no \\
u &\geq& 0 \quad\text{in}\quad \Rn,
\eea
where $\f{N+2}{N-2}<q<\f{1+\nu}{\nu}$ and $0<\nu<\f{N-2}{4}$.
\begin{definition}
If $u$ is a solution of \eqref{2-4-3} in a domain $\Om$ such that $u(x)\to\infty$ as $x\to\pa\Om$, then $u$ is called a large solution.
\end{definition}
\begin{theorem}\lab{t:ap-B1}
Suppose $u$ is any solution of Eq.\eqref{2-4-3}. Then $u=0$.
\end{theorem}
\begin{proof}
Let $U_1$ be a large solution of \eqref{2-4-3} in $B_1(0)$ such that $U_1\in L^{\infty}_{loc}(B_1(0))$ (existence of such solution follows from Theorem \ref{t:large-sol}). Define
$$U_R(x):=R^{\nu-\f{2}{q-1}}U_1(\f{x}{R}).$$
By a straight forward computation, it follows $U_R$ is a large solution of \eqref{2-4-3} in $B_R(0)$. Moreover, $U_R\to 0$ as $R\to\infty$. If $u$ is any solution of \eqref{2-4-3} in $\Rn$, then $u$ is a solution of \eqref{2-4-3} in $B_R(0)$ as well. Consequently,
$$-\text{div}(|x|^{-2\nu}\na (u-U_R))+|x|^{-(q+1)\nu}(u^q-U_R^q) = 0 \quad\text{in}\quad B_R(0).$$
Clearly $u\leq U_R$ on $\pa B_R(0)$.
Therefore by taking $(u-U_R)^+$ as a test function for the above equation we obtain
$$\int_{B_R(0)}|x|^{-2\nu}|\na(u-U_R)^+|^2+\int_{B_R(0)}|x|^{-(q+1)\nu}\f{(u^q-U^q_R)}{u-U_R}|(u-U_R)^+|^2=0,$$ which in turn implies $(u-U_R)^+=0$ in $B_R(0)$. Thus $u\leq U_R$ in $B_R(0)$. Hence by taking limit $R\to\infty$, we have $u\leq 0$ in $\Rn$. Since $u$ is a nonnegative solution, we get $u=0$ in $\Rn$.
\end{proof}
\begin{theorem}\lab{t:large-sol}
There exists a large solution $u$ to the equation \eqref{2-4-3} in $B_1(0)$. Moreover, $u\in L^{\infty}_{loc}(B_1(0))$.
\end{theorem}
We essentially follow the classical method of Veron-Vazquez \cite[Lemma 2.1]{VV} to prove this result.
\begin{proof}
We will show that there exists a radial large solution. Towards this goal, let us consider the following ode
\begin{equation}
\left\{\begin{aligned}
&u''+\f{N-1-2\nu}{r}u'(r)=r^{-(q-1)\nu}u^q \quad\text{in}\quad (0,1)\\
&u>0 \quad\text{in}\quad (0,1)\\
&u(0) =1\quad u'(0)=0.
\end{aligned}
\right.
\end{equation}
Then we can write the solution as
$$u(r)=1+\int_{0}^{r}s^{2\nu+1-N}\int_{0}^{s}t^{N-1-(q+1)\nu}u^q(t)dt ds.$$
Since $q<\f{1+\nu}{\nu}$ implies $q<\f{2+\nu}{\nu}<\f{N-\nu}{\nu}$, by the standard existence of ode theory, it follows that solution $u(r)$ exists in a neighborhood of $0$. Moreover, $q<\f{1+\nu}{\nu}$ implies $u'(0)=0$ and $u$ is $C^1$ up to the blow-up time. \\

{\bf Claim:} There exists a solution $u$ of the following ode
$$u''+\f{N-1-2\nu}{r}u'(r)=r^{-(q-1)\nu}u^q $$ in $[0, r^*)$ such that $\lim_{r\uparrow r^*}u(r)=+\infty$ for some $r^*>0$.\\

To see the claim, we use generalised Emden--Fowler transform $t=(\frac{\al}{r})^{\al}$ and $y (t)= \al^{-\nu} u(r)$, where $\al=N-2-2\nu$.   Therefore we obtain
\be\lab{ap:y-t}
y''(t)= t^\frac{-(2\al+2)+(q-1)\nu}{\al}y^{q} \quad\text{in}\quad  R< t<+\infty.
\ee
Existence of $u(r)$ in the neighbourhood of $0$ implies, Eq. \eqref{ap:y-t} has a solution $y(t)$ in $(R, \infty)$ for some large $R>0$ with $y'(\infty)=0$, $y(\infty)>0$.  To prove the claim, it is equivalent to show that there exists a  solution $y$ of \eqref{ap:y-t} in $(t^*, \infty)$ such that $\lim_{t\downarrow t^*}y(t)=\infty$ for some $t^*\in (0,\infty)$. Suppose this is not true, then
$y(t)$ can be continued as a solution of \eqref{ap:y-t} to the left of $\infty$ till $0$, i.e $y(t)$ can be defined on $(0,\infty)$.\\
Set $f(t):=t^\frac{-(2\al+2)+(q-1)\nu}{\al}$ and let $0<R<R'<\infty$. As $f$ is continuous and positive we get there exists $m, M>0$ such that $0<m\leq f(t)\leq M$ for $t\in [R, R']$. Now consider the ode
\be\lab{ap:v-t} v''(t)=M v^q(t) \quad\text{in}\quad (R, R'); \quad v>0  \quad\text{in}\quad (R, R').\ee
Rename the nonlinear term in \eqref{ap:v-t} as $h(v)$, that is $h(t):=Mt^q$. Then
$$H(t):=\int_{0}^t h(s)ds,\quad \psi(a):=\int_{a}^{\infty}\f{ds}{\sqrt{H(s)}}<\infty,$$ for any $a>0$. Therefore, applying Vazquez's classical a-priori estimates \cite{V} (also see \cite[Lemma 2.1]{VV}) we find a large solution $v(t)$ of \eqref{ap:v-t}. That is, $\lim_{t\downarrow R}v(t)=\infty=\lim_{t\uparrow R'}v(t)$. Using comparison principle it is easy to check that any solution $y$ of \eqref{ap:y-t} satisfies
\be\lab{ap:yv}y(t)\leq v(t) \quad\text{in}\quad (R, R').\ee
From \eqref{ap:v-t}, it also follows that $v$ is a convex function. If $y$ is a solution of \eqref{ap:y-t} in $(T,\infty)$ for some large $T$ with the initial value
$y(\infty)>\min_{R<t<R'} v(t)$ and $y'(\infty)=0$, then graph of $y$ must lie above all of it's tangent as $y$ is a convex decreasing function. Consequently, $y(t)> \min_{R<s<R'} v(s)$ for all $t<\infty$. Since $y$ can be extended till $0$, it in turn implies, there exists $t_1, t_2$ such that $R<t_1<t_2<R'$ and $y(t)>v(t)$ in $(t_1,t_2)$. This is a contradiction to \eqref{ap:yv}. Hence $y$
 can not be defined to the left of $\infty$ till $R$, that is, there must exist $t^*>R$ such that $\lim_{t\downarrow t^*}y(t)=\infty$. This proves the claim.
Since we have proved existence of a large solution $u$ of \eqref{2-4-3} in the ball $B_r(0)$, we use similarity transformation $T_r$ to get large solution in the unit ball $B_1(0)$. More precisely,
$U_1(x):=T_r u(x):=r^{\f{2}{q-1}-\nu}u(rx)$. This completes the proof of the theorem.
\end{proof}

\section{}
\begin{lemma}\lab{l:17-4}
Define $w({B_{t}(x)}):= \displaystyle\int_{|x-y|< t} |y|^{-2\nu}dy$. Then
\be \lab{16-4-1}
w(B_{2^{k}r}(x)) \geq C 2^{k(N-2 \nu)} w(B_{r}(x)) \ee
\end{lemma}
\begin{proof}
We prove the lemma considering into three cases. \\
Case 1:  Suppose $r\geq \frac{|x|}{10}.$  \\
Then $2^k r\geq \frac{|x|}{10}$. We claim $B_{2^kr}(0)\subset B_{2^k\cdot10r}(x)$. Indeed,
$y\in B_{2^kr}(0)$ implies $|y-x|\leq|y|+|x|\leq(2^kr+10r)\leq (2^k\cdot 10)r$. Thus the claim follows . Therefore using doubling measure  property of $|y|^{-2\nu}$, we get
$$w(B_{2^kr}(x))\geq c_1w(B_{2^k\cdot10r}(x))\geq c_1w(B_{2^kr}(0)),$$
where $c_1$ does not depend on $k, r, x$. As a consequence,
$$w(B_{2^kr}(x))\geq  c_1\om_N(2^kr)^{N-2\nu}=c_12^{k(N-2\nu)}w(B_r(0))\geq c_12^{k(N-2\nu)}w(B_r(x)).$$

Case 2: $r< \frac{|x|}{10}$ and $2^kr<\f{|x|}{10}$.\\
Then $y\in B_{2^kr}(x)$ implies  $$|x|\leq|x-y|+|y|\leq 2^kr+|y|\leq\f{|x|}{10}+|y|\Longrightarrow \f{9}{10}|x|\leq|y|.$$ Similarly it follows $|y|\leq\f{11}{10}|x|$. Thus $c_1|x|\leq|y|\leq c_2|x|$.
If $y\in B_r(x)$, then using $r< \frac{|x|}{10}$ and doing the calculation as above we get there exists positive constants $c_3, c_4$, independent of $r, x, k$ such that $c_3|x|\leq|y|\leq c_4|x|$.
Consequently,
$$w(B_{2^kr}(x))\geq  \om_Nc_2^{-2\nu}|x|^{-2\nu}(2^kr)^N\geq \om_Nc_2^{-2\nu}2^{k(N-2\nu)}|x|^{-2\nu}r^N$$
Moreover,
$$w(B_r(x))\leq \om_Nc_3^{-2\nu}|x|^{-2\nu}r^N.$$
Hence \eqref{16-4-1} holds with $C=(\f{c_2}{c_3})^{-2\nu}$.\\

Case 3: $r< \frac{|x|}{10}$ and $2^kr\geq\f{|x|}{10}$.\\
This case is similar to Case 1 and we skip the proof.
\end{proof}

\paragraph{{\bf Acknowledgement} The results in Appendix A are based on the ideas given by Prof Laurent V\'{e}ron. The authors are indebted to him for his suggestions and also for providing the references  \cite{V} and \cite{VV}. The authors also would like to thank Dr. Anup Biswas
for the idea of the proof of Theorem \ref{t:low-est}. \\
The first author is supported by INSPIRE research grant DST/INSPIRE 04/2013/000152 and the second author acknowledges funding from LMAP UMR CNRS 5142, Universit\'e Pau et des Pays de l'Adour.}

\end{document}